\documentclass{article}
\usepackage{amsfonts, amsmath, amssymb, color, mathrsfs, cite, amsthm,enumitem,verbatim,mathrsfs, multicol}

\newcommand{\norm}{\|\cdot\|}
\newcommand{\seq}{\subseteq}
\renewcommand{\epsilon}{\varepsilon}
\newcommand{\n}{\mathbb{N}}
\renewcommand{\r}{\mathbb{R}}
\newcommand{\ball}{B_{X^*}}

\def\N{{\mathbb N}}

\def\R{{\mathbb R}}
\def\ve{\varepsilon}

 2

\newtheorem{theorem}{Theorem}[section]
\newtheorem{lemma}[theorem]{Lemma}

\newtheorem{corollary}[theorem]{Corollary}
\newtheorem{proposition}[theorem]{Proposition}
\newtheorem{remark}[theorem]{Remark}
\newtheorem{question}[theorem]{Question}
\newtheorem{example}[theorem]{Example}
\newtheorem{conjecture}[theorem]{Conjecture}
\newtheorem*{thm*}{Theorem}

\begin{document}

\begin{center}
{\Large \bf James' weak compactness theorem: an exposition}
\end{center}

\medskip

\vspace{2\baselineskip}

\begin{center}{\sc  Warren B.~Moors\footnote{Corresponding author.  See back page for details} \ and \ Samuel J.~White}
 
\vspace*{0.5 true cm}

\end{center}

{\small \begin{quotation}
\noindent {\bf Abstract.}  The purpose of this paper is to provide a proof of James' weak compactness theorem that is able to be
taught in a first year graduate class in functional analysis.\end{quotation}

\medskip

\noindent {\bf AMS (2010) subject classification:} Primary 46B20; Secondary 46B10, 46B50.
\vspace{0.5true cm}

}


\section{Introduction}

The purpose of this paper is to provide a proof of James' weak compactness theorem that is able to be
taught in a first year graduate class in functional analysis.  Usually when one teaches a first course in functional analysis one teaches the
basic finite dimensional material, Hilbert space material, the open mapping theorem, the closed graph theorem, the uniform boundedness theorem and the Hahn-Banach theorem, plus applications.
Then  one might consider the spectral theory of compact normal operators, or even an introduction to $C^*$-algebras.  However, what is often neglected is the study of linear
topology, which then makes it difficult to even start to contemplate how one might prove James' theorem on weak compactness.  So, what we propose here is a way of presenting 
James' theorem on weak compactness to an audience unfamiliar with linear topology, or anything other than, the most basic facts concerning  normed linear spaces.  
 
 \medskip
 
 For the authors, James' theorem on weak compactness is one of the true delights of functional analysis. Its proof is a beautiful synthesis of linear algebra and topology. The one down-side
 of this theorem is that its proof has an unfortunate reputation of being very difficult.  We hope, among other things, to dispel this myth.   
 
 \medskip
 
 We shall start with a brief history of this problem.  Back in 1933 (see, \cite{Mazur}) S.~Mazur conjectured that a Banach space $(X,\norm)$, over the real numbers, is reflexive if, and only if, every continuous linear functional
 defined on $X$ attains its maximum value on the closed unit ball of $X$.   In 1957 (see, \cite{James57}), R.~James confirmed this conjecture for separable Banach spaces, i.e., those spaces that contain a countable dense subset.
 Later, in 1963 (see, \cite{James63}), R.~James completely confirmed the conjecture for arbitrary Banach spaces. One year after this, in \cite{James}, R.~James extended this result to show that a closed and bounded convex subset  $C$ of a Banach space $X$ is weakly compact if, and only if, every continuous linear functional defined on $X$ attains its maximum value over $C$.  The fact that this result does not extend to non-complete normed linear spaces
 was established in \cite{James-counterexample}, again by James.  Almost immediately, even in 1965 (see, \cite{Pryce}), there was a search for a simpler proof of James' weak compactness theorem.  
 The proof in \cite{Pryce} is indeed very clear and easy to read, and is in fact the basis of a lot of the work in this paper. However, \cite{Pryce} still contains a series of seven technical lemmas.  In 1972, R.~James (see, \cite{James-easy}) provided a simpler proof of 
 his own weak compactness theorem, and in \cite{S.Simons}, S.~Simons, using an inequality that now bears his name, proved the weak compactness theorem for separable Banach spaces.  
 Since these early results there have been many attempts at providing a simple proof of James' theorem.  Most of these require additional assumptions on the space.  
 One approach which is quite appealing is that of $(I)$-generation.  This first appeared in \cite{FonfLindenstraussPhelps} and then again in \cite{FonfLindenstrauss}, 
 but it has since been shown (see,  \cite{Kalenda-equivalent}) that this approach is
 essentially equivalent to the approach of S. Simons from 1972. In addition to the already mentioned papers the interested reader may also want to see the papers \cite{Cascales-simple, Godefroy, Morillion-easy, Kalenda-easy,
 WarrenSimple, MoorsWhite}, where several ``simple'' proofs of James' theorem are given.  The paper \cite{Cascales-simple} also has some interesting applications and historical facts.
 
 \medskip
 
We now return to the mathematics.  In this paper all vector spaces and all normed linear spaces will be over the field of real numbers. The key concept, which runs throughout this paper, is the notion of a convex set.  
 A subset $C$ of a vector space $(V, +, \cdot)$, over the real numbers, is called {\it convex} if, for every pair of points $x,y \in C$ and $0<\lambda <1$, we have $\lambda x + (1-\lambda)y \in C$.  We encourage the reader to follow 
 the  role that this concept plays throughout the rest of this paper.

\medskip

The structure of the reminder of this paper is as follows:  Section 2 contains the necessary background material. In particular, it contains Subsections; 2.1 on Weak topologies, 2.2 on Linear topology, 
2.3 on the Hahn-Banach Theorem, 2.4 on the Weak$^*$ topology.  Readers with a background in linear topology may wish to skip this section.  Section 3 contains three proofs of James' theorem given in three
subsections; 3.1 on James' theorem for separable Banach spaces; 3.2 on James' theorem for  spaces with a weak$^*$ sequentially compact dual ball, 3.3 the general version of James' theorem.  In Subsection 3.4 some applications
of James' theorem are given. In Section 4 a generalisation of James' theorem is given. To achieve this, this section contains Subsection 4.1 which gives the necessary background in convex analysis, then in 
Subsection 4.2 the necessary set-valued analysis is given. In Subsection 4.3 the generalisation of  James' theorem is presented. Finally, in Section 5, we give a variational principle that is based upon the 
generalised version of James' theorem.  The paper ends with an index of notation and assumed knowledge and a bibliography.


\section{Preliminaries}


In this section of the paper we will present the necessary background material that is required in order to prove James' theorem on weak compactness of closed and bounded convex subsets of a given Banach space.


\subsection{Weak topologies on sets}


An important part of general topology concerns the generation of topologies on a given set.  In this subsection we will show how to construct topologies that make a given function (set of functions) continuous.

\begin{proposition}\label{weak1} Let $f:X \to Y$ be a function between sets $X$ and $Y$.  If $\tau_Y$ is a topology on $Y$ then $\tau_X := \{f^{-1}(U):U \in \tau_Y\}$
is a topology on $X$ and $f:(X, \tau_X) \to (Y, \tau_Y)$ is continuous.  Furthermore, if $\tau$ is any topology on $X$ such that $f:(X,\tau) \to (X, \tau_Y)$ is continuous 
then $\tau_X \subseteq \tau$.  This is, $\tau_X$ is the weakest topology on $X$ that makes $f$ continuous (when $Y$ is endowed with the topology $\tau_Y$).
\end{proposition}

\begin{proof} First we will show that $\tau_X$ is a topology on $X$. Now, $\varnothing \in \tau_X$ since $\varnothing = f^{-1}(\varnothing)$ and $\varnothing \in \tau_Y$. Similarly,
$X \in \tau_X$ since $X = f^{-1}(Y)$ and $Y \in \tau_Y$.  Next, suppose that $V_1 \in \tau_X$ and $V_2 \in \tau_2$. Then, by the definition of $\tau_X$, there exists $U_1 \in \tau_Y$
and $U_2 \in \tau_Y$ such that $V_1 = f^{-1}(U_1)$ and $V_2 = f^{-1}(U_2)$. Therefore,
$$V_1 \cap V_2 = f^{-1}(U_1) \cap f^{-1}(U_2) = f^{-1}(U_1 \cap U_2).$$
Since $\tau_Y$ is a topology on $Y$, $U_1 \cap U_2 \in \tau_Y$.  Hence, $V_1 \cap V_2 \in \tau_X$.  Finally, to show that $\tau_X$ is a topology on $X$, suppose that $\{V_\alpha:\alpha \in A\} \subseteq \tau_X$.
Then, by the definition of $\tau_X$, there exist $\{U_\alpha:\alpha \in A\} \subseteq \tau_Y$ such that $V_\alpha = f^{-1}(U_\alpha)$ for each $\alpha \in A$.  Therefore,
$$\mbox{$\bigcup_{\alpha \in A}$} V_\alpha = \mbox{$\bigcup_{\alpha \in A}$} f^{-1}(U_\alpha) = f^{-1}\left(\mbox{$\bigcup_{\alpha \in A}$} U_\alpha \right)\!.$$
Since $\tau_Y$ is a topology on $Y$, $\bigcup_{\alpha \in A} U_\alpha  \in \tau_Y$.  Hence, $\bigcup_{\alpha \in A} V_\alpha \in \tau_X$.  Thus, $\tau_X$ is a topology on $X$.  To show that $f:(X, \tau_X) \to (Y, \tau_Y)$ is continuous we consider the following.  
Let $U \in \tau_Y$.  Then $f^{-1}(U) \in \tau_X$, by the definition of $\tau_X$.  Therefore, by the definition of continuity, $f:(X, \tau_X) \to
(Y, \tau_Y)$ is continuous.  For our last step of the proof, we will show that $\tau_X$ is the weakest topology on $X$ that makes $f$ continuous.  To this end, let $\tau$ be any topology
on $X$ such that $f:(X, \tau) \to(Y, \tau_Y)$ is continuous.  Let $V \in \tau_X$. Then, by the definition of $\tau_X$, there exists an $U \in \tau_Y$ such that $V = f^{-1}(U)$. Since we are
assuming that $f:(X, \tau) \to(Y, \tau_Y)$ is continuous, $V = f^{-1}(U) \in \tau$. Thus, $\tau_X \subseteq \tau$.  This completes the proof.
\end{proof}

\noindent The topology $\tau_X$ in Proposition \ref{weak1}  is called the {\it weak topology on $X$ generated by $f$ and $\tau_Y$}, or more briefly, when the context is clear,  the {\it weak topology on $X$}.

\medskip

When we have more than one function we still have the following result.

\begin{proposition}\label{weak2} Let $X$ and $Y$ be sets and let $\tau_Y$ be a topology on $Y$.  If $\mathcal F$ is a nonempty family of functions from $X$ into $Y$ then 
$$\mathcal{B} := \left\{\mbox{$\bigcap_{1 \leq k \leq n}$} f_k^{-1}(U_k): n \in \N, U_k \in \tau_Y   \mbox{ and }  f_k \in \mathcal{F}\right\}$$
is a base for a topology $\tau_X$ on $X$.  Furthermore, the topology $\tau_X$ is the weakest topology on $X$ that make each $f \in \mathcal{F}$ continuous, when $Y$ is
endowed with the topology $\tau_Y$.
\end{proposition}

\begin{proof} Firstly, it is easy to see that $\varnothing$ and $X$ are members of $\mathcal{B}$.  Indeed, since $\mathcal{F} \not= \varnothing$ we may take a function $f \in \mathcal{F}$.
Then $\varnothing = f^{-1}(\varnothing)$ and so $\varnothing \in \mathcal{B}$ since $\varnothing \in \tau_Y$.  Similarly, 
 $X = f^{-1}(Y)$ and so $X \in \mathcal{B}$ since $Y \in \tau_Y$.
Next, let us observe that $\mathcal{B}$ is closed under taking finite intersections.  Suppose $V_1 \in \mathcal{B}$ and $V_2 \in \mathcal{B}$. Then there exists $n_1 \in \N$, $f'_k \in \mathcal{F}$
and $U_k' \in \tau_Y$ for each $1 \leq k \leq n_1$ such that $V_1 = \bigcap_{1 \leq k \leq {n_1}} (f'_k)^{-1}(U'_k)$. Similarly, there exists $n_2 \in \N$, $f''_k \in \mathcal{F}$
and $U_k'' \in \tau_Y$ for each $1 \leq k \leq n_2$ such that $V_2 = \bigcap_{1 \leq k \leq {n_2}} (f''_k)^{-1}(U''_k)$.  Let $ n := n_1 + n_2$ and for each $1 \leq k \leq n_1$
let $U_k := U'_k$ and for each $1 \leq k \leq n_2$ let $U_{{n_1}+k} := U''_k$.  For each $1 \leq k \leq n_1$ let $f_k := f'_k$ and for each $1 \leq k \leq n_2$ let $f_{{n_1}+k} := f''_k$.  Then,
$$V_1 \cap V_2 = \mbox{$\bigcap_{1 \leq k \leq {n_1}}$} (f'_k)^{-1}(U'_k) \cap  \mbox{$\bigcap_{1 \leq k \leq {n_2}}$} (f''_k)^{-1}(U''_k) =  \mbox{$\bigcap_{1 \leq k \leq {n}}$} f_k^{-1}(U_k) \in \mathcal{B}.$$
 We now define $\tau_X$ to be the set of all subsets of $X$ that can be expressed as a union of members of $\mathcal{B}$.  From above we see that $\varnothing$ and $X$ are members of $\tau_X$,
 and $\tau_X$ is closed under taking finite intersections.  For the details of this last claim consider the following.  Let $V_1 \in \tau_X$ and $V_2 \in \tau_X$. Then there exist disjoint
 sets $I_1$ and $I_2$ such that $V_1 = \bigcup_{i \in I_1} B_i$ for some $B_i \in \mathcal{B}$ and $V_2 = \bigcup_{i \in I_2} B_i$ for some $B_i \in \mathcal{B}$.  Let $I := I_1 \times I_2$ then
 $$V_1 \cap V_2 = (\mbox{$\bigcup_{i \in I_1} $} B_i) \cap (\mbox{$\bigcup_{i \in I_2} $} B_i) = \mbox{$\bigcup_{(i,j) \in I}  $} B_i  \cap B_j \in \tau_X \mbox{\quad since, $B_i \cap B_j \in \mathcal{B}$}$$
 
 So it remains to show that $\tau_X$ is closed under arbitrary unions.  Suppose that $\{U_i:i \in I\} \subseteq \tau_X$.   Then for each
 $i \in I$, there exist disjoint sets  $J_i$ such that $U_i = \bigcup_{j \in J_i} B_j$, where $B_j \in \mathcal{B}$.  Let $J := \bigcup_{i \in I} J_i$.  Then $\bigcup_{I \in I} U_i = \bigcup_{j \in J} B_j \in \tau_X$.
 We now show that $\tau_X$ is the weakest topology on $X$ the makes each function in $\mathcal{F}$ continuous.  So suppose that $\tau$ is a topology on $X$ that makes each function in
 $\mathcal{F}$ continuous.  Then clearly $\mathcal{B} \subseteq \tau$ since $f^{-1}(U) \in \tau$ for each $f \in \mathcal{F}$ and each $U \in \tau_Y$.  Since $\tau_X$ is the smallest topology
 on $X$ that contains $\mathcal{B}$ we must have that $\tau_X \subseteq \tau$.  
\end{proof}

\noindent The topology $\tau_X$ in Proposition \ref{weak2}  is call the {\it weak topology on $X$ generated by $\mathcal F$ and $\tau_Y$}, or more briefly, 
when the context is clear, the {\it weak topology on $X$}. For further information on general topology see \cite{Engelking, Kelley}.


\subsection{Linear topologies}


\noindent Let $(V, +, \cdot)$ be a vector space over the field of real numbers and let $\tau$ be a topology on $V$.  Then $(V, +, \cdot, \tau)$ is called a {\it linear topological space} or a
{\it topological vector space}  if vector addition from $V \times V$ into $V$ is continuous, when $V \times V$ is considered with the product topology and scalar multiplication
from $\R \times V$ into $V$ is continuous, again when we consider $\R \times V$ with the product topology and $\R$ with the usual topology.

\medskip

An important feature of linear topological spaces is that they are always {\it regular}.  That is, if $(X, +, \cdot, \tau)$ is linear topological space, $C$ is a closed subset of $X$ and $x \in X \setminus C$ then
there exist disjoint open sets $U$ and $V$ such that $x \in U$ and $C \subseteq V$.  To see this, suppose that $x = x+0 \in X \setminus C$; which is open. Therefore, from the continuity of addition, there
exist open neighbourhoods $U$ of $x$ and $W$ of $0$ such that $U+W \subseteq X \setminus C$, i.e,. $(U+W) \cap C = \varnothing$.  Therefore, $U \cap (C + (-W)) = \varnothing$. Let $V := C + (-W) = \bigcup_{c \in C} c-W$.
Then $V$ is an open set containing the set $C$ and $U \cap V = \varnothing$.  Thus, $(X, \tau)$ is a regular topological space.

\medskip

\noindent Let $(V, +, \cdot, \tau)$ be a linear topological space over $\R$. We shall say that $(V, +, \cdot, \tau)$ is a {\it locally convex space} if for each open set $U$ in $V$, containing $0$,
there exists an open convex set $W$ such that $0 \in W \subseteq U$, or, equivalently, $(V, \tau)$ has a local base consisting of open convex sets. 

\medskip

\noindent If $(X, \norm)$ is a normed linear space and $\varnothing \not= \mathcal{F} \subseteq X^*$ - the set of all continuous linear functionals on $X$, then $\sigma(\mathcal{F}, X)$ denotes the weak topology on $X$ generated by $\mathcal F$. We shall
simply call $\sigma(X^*, X)$ the {\it weak topology on $X$} and write $(X, \mathrm{weak})$ for $(X, \sigma(X^*, X))$.

\medskip

Sometimes it is convenient to work with a more concrete representation of the $\sigma(\mathcal{F}, X)$-topology.  Fortunately such a representation exists and furthermore, the representation is very similar to the
way in which open sets are defined in metric spaces.  Let $(X, \norm)$ be a normed linear space, let $x_0 \in X$, $\varepsilon >0$ and let $F$ be a nonempty finite subset of $X^*$.  Then 
$$N_X(x_0,F,\varepsilon) :=  \mbox{$\bigcap_{f \in F}$} \{x \in X: |f(x) - f(x_0)| < \varepsilon\}.$$
Note: sometimes it is also convenient to write $N_X(x_0, f_1, f_2, \ldots, f_n, \varepsilon)$ when the finite set $F$ is enumerated as $F := \{f_1, f_2, \ldots, f_n\}$.  When the context is clear we simply write, $N(x_0,F,\varepsilon)$ or 
$N(x_0, f_1, f_2, \ldots, f_n, \varepsilon)$.

\medskip

Given a nonempty subset $\mathcal{F}$ of $X^*$ we shall say that a subset $U$ of $X$ is {\it $\mathcal{F}$-open} if for every $x_0 \in U$ there exists a nonempty finite subset $F$ of $\mathcal{F}$ and an $\varepsilon >0$
such that $N(x_0,F,\varepsilon) \subseteq U$.

\begin{proposition}\label{coincide}  If $(X, \norm)$ is a normed linear space and $\varnothing \not= \mathcal{F} \subseteq X^*$ then the set of all $\mathcal{F}$-open sets forms a topology on $X$.  
Furthermore, the set of all $\mathcal{F}$-open sets coincides with the $\sigma(\mathcal{F}, X)$-topology on $X$.
\end{proposition} 
\begin{proof} First we will show that the set of all $\mathcal{F}$-open sets forms a topology on $X$.  It is easy to see that vacuously, $\varnothing$ is $\mathcal{F}$-open.  To see that $X$ is $\mathcal{F}$-open, consider
any element $x_0 \in X$.  Now, let $x^*$ be any element in  $\mathcal{F}$.  Then $N(x_0,\{x^*\}, 1) \subseteq X$.  Therefore, $X$ is $\mathcal{F}$-open.  Next, suppose that $U$ and $V$ are both $\mathcal{F}$-open
subsets of $X$.  We will show that $U \cap V$ is also $\mathcal{F}$-open.  To this end, let $x_0 \in U \cap V$.  Then, since $x_0 \in U$, there exists a finite subset $F_U$ of $\mathcal{F}$ and an $\varepsilon_U >0$ such
that $N(x_0, F_U, \varepsilon_U) \subseteq U$. Similarly, there exists a finite subset $F_V$ of $\mathcal{F}$ and an $\varepsilon_V >0$ such that $N(x_0, F_V, \varepsilon_V) \subseteq V$.
Let $F := F_U \cup F_V$ and $\varepsilon := \min\{\varepsilon_U, \varepsilon_V\}$.  Then,
$$x_0 \in N(x_0, F, \varepsilon) \subseteq N(x_0, F_U, \varepsilon_U) \cap N(x_0, F_V, \varepsilon_V) \subseteq U \cap V.$$
So it remains to show that an arbitrary union of $\mathcal{F}$-open sets is again $\mathcal{F}$-open. Let $\{U_\alpha:\alpha \in A\}$ be a family of $\mathcal{F}$-open sets.  Let $x_0$ be any element of 
$\bigcup_{\alpha \in A} U_\alpha$.  Then there exists an $\alpha_0 \in A$ such that $x_0 \in U_{\alpha_0}$. Since $U_{\alpha_0}$ is $\mathcal{F}$-open there exists a finite subset $F$ of $\mathcal{F}$ and an
$\varepsilon >0$ such that $N(x_0, F, \varepsilon) \subseteq U_{\alpha_0}$.  Now, $U_{\alpha_0} \subseteq \bigcup_{\alpha \in A} U_\alpha$ and so $N(x_0, F, \varepsilon) \subseteq \bigcup_{\alpha \in A} U_\alpha$.
Therefore, $\bigcup_{\alpha \in A} U_\alpha$ is $\mathcal{F}$-open.  We will now show that the two topologies coincide.  Suppose that $U$ is an $\mathcal{F}$-open set.  Then, for each $x \in U$, there
exists a finite subset $F_x$ of $\mathcal{F}$ and an $\varepsilon_x >0$ such that $x \in N(x,F_x,\varepsilon_x) \subseteq U$.  Therefore, $\bigcup_{x \in U} N(x,F_x,\varepsilon_x) = U$.  Thus, to show that
$U$ is $\sigma(\mathcal{F}, X)$-open it is sufficient to show that every set of the form:  $N(x,F,\varepsilon)$ is $\sigma(\mathcal{F}, X)$-open, where $x \in X$, $F$ is a finite subset of $\mathcal{F}$ and $\varepsilon >0$.
So suppose that $x_0 \in X$, $F = \{f_1,f_2, \ldots, f_n\} \subseteq \mathcal{F}$ and $\varepsilon >0$.  Let $U_k := (f_k(x_0)-\varepsilon, f_k(x_0) + \varepsilon)$ for each $1 \leq k \leq n$. Then,
$$N(x,F, \varepsilon) = \mbox{$\bigcap_{f \in F}$} \{x \in X: |f(x) - f(x_0)| < \varepsilon\} = \mbox{$\bigcap_{1 \leq k \leq n}$} f_k^{-1}(U_k). \mbox{\quad \quad $(*)$}$$
Therefore, by the definition of the $\sigma(\mathcal{F},X)$-topology, $N(x,F, \varepsilon)$ is $\sigma(\mathcal{F},X)$-open. To show that every $\sigma(\mathcal{F},X)$-open set is $\mathcal{F}$-open it is sufficient to show that each member of $\mathcal{F}$ is continuous with respect to the topology generated by the $\mathcal{F}$-open sets. However, this is obvious from the definition of the $\mathcal{F}$-open sets. If you want to see the details, then let $f \in \mathcal{F}$, $x_0 \in X$ and $\varepsilon >0$.  Then 
$$f(N(x_0, \{f\}, \varepsilon)) \subseteq (f(x_0)-\varepsilon, f(x_0)+\varepsilon).$$
This completes the proof.
\end{proof}

\begin{remark}\label{coincide-remark} It follows from Proposition \ref{coincide} and equation $(*)$ that for  each $x \in X$, finite set $\varnothing \not= F \subseteq \mathcal{F}$ and $\varepsilon >0$, the set $N(x,F,\varepsilon)$ is $\sigma(\mathcal{F},X)$-open in $X$.
\end{remark}

By using Proposition \ref{coincide} and Remark \ref{coincide-remark}, one can easily deduce the following result.

\begin{proposition}\label{relative-topology} Let $Y$ be a subspace of a normed linear space $(X,\norm)$ and let $\varnothing \not= \mathcal{F} \subseteq X^*$.  Then a subset $U$ of $Y$ is
open in the relative $\sigma(\mathcal{F},X)$-topology on $Y$ if, and only if, for each $y \in U$ there exists a finite subset $F$ of $\mathcal{F}$ and an $\varepsilon >0$ such that
$N_X(y,F,\varepsilon) \cap Y \subseteq U$.
\end{proposition}

\begin{proposition} If $(X, \norm)$ is a normed linear space and $\mathcal{F} \subseteq X^*$, then $(X, \sigma(\mathcal{F}, X))$ is a locally convex topological space.
\end{proposition}
\begin{proof} Let us first show that  $(X, \sigma(\mathcal{F}, X))$ is a linear topology.  Let $S:X \times X \to X$ be defined by, $S(x,y) := x+y$.  We need to show that $S$ is continuous. To this end, let $W$ be a $\sigma(\mathcal{F}, X)$-open subset of $X$
and let $(x,y) \in S^{-1}(W)$, i.e., $S(x,y) \in W$.  By Proposition \ref{coincide} there exists a finite subset $F$ of $\mathcal{F}$ and an $\varepsilon >0$ such that $N(S(x,y),F,\varepsilon) \subseteq W$. 
We claim that $S(N(x,F,\varepsilon/2) \times N(y,F,\varepsilon/2)) \subseteq N(S(x,y),F,\varepsilon) \subseteq W$.  To see this, let $(x',y') \in N(x,F,\varepsilon) \times N(y,F,\varepsilon)$ and let $f \in F$. Then,
$|f(x')-f(x)| < \varepsilon/2$ and $|f(y')-f(y)|< \varepsilon/2$, and so
\begin{eqnarray*}
|f(S(x',y')) - f(S(x,y))| &=& |f(x'+y')-f(x+y)| \\
&=& |f(x'-x) + f(y'-y)| \\
&\leq& |f(x'-x)| + |f(y'-y)| \\
& =& |f(x')-f(x)| + |f(y')-f(y)| < \varepsilon/2 +  \varepsilon/2 = \varepsilon.
\end{eqnarray*}
Therefore, $S(x',y') \in N(S(x,y),F,\varepsilon)$; which proves the claim.  Now since both $N(x,F,\varepsilon/2)$ and  $N(y,F,\varepsilon/2)$ are $\sigma(\mathcal{F}, X)$-open we see that $S^{-1}(W)$ is open in $X \times X$,
with the product topology and so $S$ is continuous. Let $M:\R \times X \to X$ be defined by, $M(r,x) := rx$. We need to show that $M$ is continuous.  To this end, let $W$ be a $\sigma(\mathcal{F}, X)$-open subset of $X$ and
let $(r,x) \in M^{-1}(W)$, i.e., $M(r,x) \in W$. By Proposition \ref{coincide} there exists a finite subset $F$ of $\mathcal{F}$ and an $1 > \varepsilon >0$ such that $N(M(r,x),F,\varepsilon) \subseteq W$.  Set 
$$\varepsilon_1 := \frac{\varepsilon}{2(|f(x)| +1)} \mbox{\quad and \quad } \varepsilon_2 := \frac{\varepsilon}{2(|r|+1)}.$$
We claim that $M((r -\varepsilon_1,r+\varepsilon_1)\times  N(x,F,\varepsilon_2)) \subseteq N(M(r,x),F,\varepsilon) \subseteq W$.
To see this is true, let \\ $(r',x') \in (r -\varepsilon_1,r+\varepsilon_1)\times  N(x,F,\varepsilon_2)$ and let $f \in F$. Then,
$|r'-r| < \varepsilon_1$ and $|f(x')-f(x)|< \varepsilon_2$, and so
\begin{eqnarray*}
|f(M(r',x')) -f(M(r,x))| &=& |f(r'x')-f(rx)| \\
&=& |f(r'x') - f(r'x) - [f(rx) -f(r'x)] | \\
& \leq& |f(r'x') - f(r'x)| + |f(rx) -f(r'x)|  \\
&=& |r'| |f(x')-f(x)| + |r-r'||f(x)| \\
&\leq& (|r| +1)\varepsilon_2  + \varepsilon_1 |f(x)| <  \varepsilon \mbox{\quad \ since, $|r'| \leq |r| +\varepsilon_1 < |r|+1$}.
\end{eqnarray*}
Therefore, $M(r',x') \in N(M(r,x),F,\varepsilon)$; which proves the claim.  Now since $(r -\varepsilon_1,r+\varepsilon_1)$ is open in $\R$ and  $N(x,F,\varepsilon_2)$ is $\sigma(\mathcal{F}, X)$-open 
we see that $M^{-1}(W)$ is open in $\R \times X$, with the product topology and so $M$ is continuous.  This shows that $(X, \sigma(\mathcal{F}, X))$ is a linear topological space.  To see that 
$(X, \sigma(\mathcal{F}, X))$ is locally convex we merely appeal to Proposition \ref{coincide} and the fact that for each finite subset $F$ of $\mathcal{F}$ and $\varepsilon >0$, $N(0,F,\varepsilon)$ is
a convex open neighbourhood of $0$.
\end{proof}

The beauty of linear topology lies in the interplay between linear algebra and topology.  This is highlighted in Proposition \ref{cont-equivalence}, which is based upon the following result from linear algebra.

\begin{lemma}\label{FiniteDimensional}
	Let $V$ be a vector space over $\r$ and suppose that $(f_i)_{i=1}^n$ are linear functionals on $V$. 
	If $g$ is a linear functional on $V$ such that $\bigcap_{i=1}^n\ker(f_i)\seq\ker(g)$, then $g\in \emph{span}\lbrace f_1,\dots , f_n\rbrace$.
\end{lemma}
\begin{proof}
	Define $T:V\rightarrow\r^n$ by 
	$$T(x):=(f_1(x),\dots , f_n(x)) \ \text{for all} \ x\in V.$$
	 Observe that $T$ is clearly linear and that $\ker(T)=\bigcap_{i=1}^n\ker(f_i)$. We may assume that $(f_i)_{i=1}^n$ is a minimal (in terms of cardinality) family of functions such that $\bigcap_{i=1}^n\ker(f_i)\seq\ker(g)$, and from this we claim 
	 that $T$ is also surjective. \\ \\
	Fix $1\leq k\leq n$. Then, by the minimality assumption on $n$, we have that $\bigcap\lbrace\ker(f_i):1\leq i\leq n, \ i\neq k\rbrace\not\seq \ker(g)$, and so in particular $\bigcap\lbrace\ker(f_i):1\leq i\leq n, \ i\neq k\rbrace\not\seq \bigcap_{i=1}^n\ker(f_i)$.
	Now we may choose $$x_k\in \mbox{$\bigcap$}\lbrace\ker(f_i):1\leq i\leq n, \ i\neq k\rbrace\setminus \mbox{$\bigcap_{i=1}^n$}\ker(f_i).$$
	Then, after scaling $x_k$ if necessary, we have that $f_k(x_k)=1$ and $f_i(x_k)=0$ for $i\neq k$. Therefore, $T(x_k)=e_k$, where $e_k$ is the $k^{\textrm{th}}$ standard basis vector of $\r^n$. Then, since $1\leq k\leq n$ was arbitrary, we have that $\r^n=\text{span}\lbrace e_1,\dots, e_n\rbrace\seq T(V)$ and so $T$ is surjective as claimed. \\ \\
	Now define $g^*:\r^n\rightarrow\r$ by 
	$$g^*(x):=g(z) \ \text{for any} \ z\in T^{-1}(x).$$
	Then $g^*$ is well-defined.  Indeed, let $x\in \r^n$. Since $T$ is onto, we have that $T^{-1}(x)\neq\varnothing$. So, suppose $z_1,z_2\in T^{-1}(x)$. Then $T(z_1)=x=T(z_2)$ and so $z_1-z_2\in \ker(T)\seq\ker(g)$. Thus $g(z_1)=g(z_2)$ as required. Moreover, a routine calculation shows that $g^*$ is linear, so $g^*\in (\r^n)^*$. 
	Since $(\r^n)^*=\text{span}\lbrace e_1^*,\dots,e_n^*\rbrace$, there exist $(c_i)_{i=1}^n$ such that $g^*=\sum_{i=1}^n c_ie_i^*$, where here $e_i^*(e_j)=\delta_{ij}$, the $ij$-Kroeneker delta. \\ \\
	Finally, note that $x\in T^{-1}(T(x))$ and so $g^*(T(x))=g(x)$ for all $x\in V$. Therefore,
	$$\mbox{$g=g^*\circ T=\left(\sum_{i=1}^nc_ie_i^*\right)\circ T=\sum_{i=1}^nc_i\left(e_i^*\circ T\right)=\sum_{i=1}^nc_if_i$}$$
	since $e^*_i\circ T=f_i$ for all $1\leq i\leq n$, and thus $g\in \text{span}\lbrace f_1,\dots , f_n\rbrace$. 
\end{proof}

\begin{proposition}\label{cont-equivalence} If $(X, \norm)$ is a normed linear space and $\mathcal{F} \subseteq X^*$, then the following are equivalent:
\begin{enumerate}
\item[{\rm (i)}] $x^* \in X^*$ is $\sigma(\mathcal{F}, X)$-continuous;
\item[{\rm (ii)}] $x^* \in X^*$ is bounded on a $\sigma(\mathcal{F}, X)$ neighbourhood of $0$;
\item[{\rm (iii)}] $x^* \in \mbox{span}(\mathcal{F})$.
\end{enumerate}
\end{proposition}
\begin{proof} $(i) \Longrightarrow (ii)$. Suppose that $x^* \in X^*$ is $\sigma(\mathcal{F}, X)$-continuous on $X$.  Then, in particular, $x^*$ is continuous at $0 \in X$.  Therefore, there exists a $\sigma(\mathcal{F}, X)$-open neighbourhood
$N$ of $0$ such that $|x^*(x)| = |x^*(x)-x^*(0)| <1$ for all $x \in N$.  Hence $x^*$ is bounded on $N$.  $(ii) \Longrightarrow (iii)$. Suppose that $x^*$ is bounded on a $\sigma(\mathcal{F}, X)$-open neighbourhood $N$ of $0$.
Then there exists a finite subset $G$ of $\mathcal{F}$ and an $0 < \varepsilon$ such that $N(0, G,\varepsilon) \subseteq N$.  Let $S := \bigcap_{y^* \in G} \mathrm{ker}(y^*)$.  Then $S$ is a subspace of $X$ and furthermore, 
$S \subseteq N(0, G,\varepsilon) \subseteq N$. Hence, $x^*|_S$ is bounded on $S$. Thus, $x^*|_S \equiv 0$ and so $\bigcap_{y^* \in G} \mathrm{ker}(y^*) = S \subseteq \mathrm{ker}(x^*)$. The result now
follows from Lemma \ref{FiniteDimensional}. $(iii) \Longrightarrow (i)$.  Suppose that $x^* = \sum_{k=1}^n \lambda_k y^*_k$, where $\lambda_k \in \R$ and $y^*_k \in \mathcal{F}$ for all $1 \leq k \leq n$. Define $f:\R^n \to \R$
by, $f(x_1, x_2, \ldots, x_n) := \sum_{k=1}^n \lambda_k x_k$ and $F:X \to \R^n$ by, $F(x) := (y^*_1(x), y^*_2(x), \ldots, y^*_n(x))$.  Then $f$ is a continuous function on $\R^n$ and $F$ is a $\sigma(\mathcal{F}, X)$-continuous function on $X$.  
Therefore, $x^* = f \circ F$, is also a $\sigma(\mathcal{F}, X)$-continuous function on $X$.
\end{proof}

\begin{remark} It follows from Proposition \ref{cont-equivalence} that for any normed linear space $(X, \norm)$ and any $\mathcal{F} \subseteq X^*$, $\sigma(\mathcal{F}, X) = 
\sigma(\mbox{span}(\mathcal{F}), X)$.  To see this, first note the general fact that if $\mathcal{F} \subseteq \mathcal{F}'$ then $\sigma(\mathcal{F}, X) \subseteq \sigma(\mathcal{F}',X)$
(i.e., to make more functions continuous you need more open sets) and then the equivalence of (i) and (iii) above.
\end{remark}

\begin{proposition}\label{linearmap} If $T:X \to Y$ is a continuous linear operator acting between normed linear spaces $(X\norm_X)$ and $(Y,\norm_Y)$ then $T:(X,\mathrm{weak}) \to (Y,\mathrm{weak})$ is also continuous.
\end{proposition}
\begin{proof} Let $W$ be a weak open subset of $Y$.  We will show that $T^{-1}(W)$ is open in the weak topology on $X$. To this end, let $x_0 \in T^{-1}(W)$.  Then $T(x_0) \in W$ and so
by Proposition \ref{coincide} there exist $\{y_1^*, y_2^*, \ldots, y_n^*\} \subseteq Y^*$ and $\varepsilon >0$ such that $N_Y(T(x_0),\{y_1^*, y_2^*, \ldots, y_n^*\},\varepsilon) \subseteq W$. For each
$1 \leq k \leq n$ let $x_k^* := y^*_k \circ T$.  Then $\{x_1^*, x_2^*, \ldots, x_n^*\} \subseteq X^*$.  We claim that 
$T(N_X(x_0, \{x_1^*, x_2^*, \ldots, x_n^*\}, \varepsilon)) \subseteq N_Y(T(x_0),\{y_1^*, y_2^*, \ldots, y_n^*\},\varepsilon) \subseteq W$.  To see this, let $y \in T(N_X(x_0, \{x_1^*, x_2^*, \ldots, x_n^*\}, \varepsilon))$.
Then there exists a $x \in N_X(x_0, \{x_1^*, x_2^*, \ldots, x_n^*\}, \varepsilon)$ such that $y = T(x)$.  Fix $1 \leq k \leq n$.  Then,
\begin{eqnarray*}
|y_k^*(y) - y_k^*(T(x_0))| &=& |y_k^*(T(x)) - y_k^*(T(x_0))| \\
&=& |x_k^*(x) - x_k^*(x_0)| < \varepsilon, \mbox{\quad since $x \in N_X(x_0, \{x_1^*, x_2^*, \ldots, x_n^*\}, \varepsilon)$.}
\end{eqnarray*}
Therefore, $y \in N_Y(T(x_0),\{y_1^*, y_2^*, \ldots, y_n^*\},\varepsilon) \subseteq W$.  This completes the proof of the claim. Hence 
$$x_0 \in N_X(x_0, \{x_1^*, x_2^*, \ldots, x_n^*\}, \varepsilon) \subseteq T^{-1}(W).$$
Thus, by Proposition \ref{coincide}, $T^{-1}(W)$ is open in the weak topology on $X$.
\end{proof}


\subsection{Hahn-Banach Theorem}


A real-valued function $p$ defined on a vector space $V$ is called
{\it sublinear} if for every $x,y \in V$ and $0 \leq \lambda < \infty$,
$p(\lambda x) = \lambda p(x)$ and $p(x+y) \leq p(x) +p(y)$. 

\medskip

Although it is easy, using linear algebra, to construct linear functionals on a vector space, it is not so easy to construct continuous linear functions on a linear topological space.  The key to constructing
continuous linear functionals on locally convex spaces is given next.

\begin{theorem}[Hahn-Banach Theorem \cite{DunfordSchwartz1}] \label{Hahn-Banach}Let $Y$ be a subspace of a vector space $V$ (over $\mathbb{R}$) and let
$p:V \rightarrow \mathbb{R}$ be a sublinear functional on $V$. If $f$ is a linear functional
on $Y$ and $f(y) \leq p(y)$ for all $y \in Y$ then there exists a linear functional $F:V \rightarrow
\mathbb{R}$ such that $F|_Y = f$ and $F(x) \leq p(x)$ for all $x \in V$.
\end{theorem}

\begin{proof} Let $\mathscr P$ be the collection of all ordered pairs $(M',f')$, where $M'$ is a subspace of $V$
containing $Y$ and $f':M' \rightarrow \mathbb{R}$ is a linear functional defined on $M'$
such that $f'|_Y = f$ and satisfies $f'(x) \leq p(x)$ for all $x \in M'$. $\mathscr P$ is
nonempty because $(Y,f) \in \mathscr{P}$. We partially order $\mathscr P$ by, 
$(M',f') \leq (M'', f'')$ if $M' \subseteq M''$ and $f''|_{M'} = f'$. If $\{(M_\alpha , f_\alpha): \alpha \in A \}$
is a nonempty totally ordered sub-family of $\mathscr P$, then set $M' := \bigcup \{M_\alpha:\alpha \in A \}$
and define the linear functional $f':M' \rightarrow \mathbb{R}$ by, $f'(x) := f_\alpha(x)$ if
$x \in M_\alpha$. Then $(M',f') \in \mathscr{P}$ and $(M_\alpha, f_\alpha) \leq (M',f')$ for all
$\alpha \in A$. Therefore, by Zorn's lemma, $\mathscr P$ has a maximal element $(M,F)$. We must show that $M =V$.
So suppose, in order to obtain a contradiction, that $M \not= V$ and pick $x_0 \in V \setminus M$ and put
$M^* := \mbox{span}\{M,x_0\}$. We will define $F^*:M^*\rightarrow \mathbb{R}$ so that $(M^*,F^*) \in \mathscr{P}$
and $(M,F) < (M^*,F^*)$; which will be our desired contradiction. For each $\alpha \in \mathbb{R}$ we define
$F_\alpha$ on $M^*$ by, $F_\alpha(m+\lambda x_0) := f(m) + \lambda \alpha$.  It is easy to check that
$F_\alpha$ is well defined and linear on $M^*$. Moreover, $F_\alpha|_M =f$. So it remains to show that
$F_{\alpha}(x) \leq p(x)$ for all $x \in M^*$. To achieve this, we need to select the right value of $\alpha \in \mathbb{R}$.

\medskip

{\it Selection of $\alpha$:}  For any $m_1, m_2 \in M$ and $0< \lambda_1 < \infty$ and $0< \lambda_2 < \infty$
we have:
$$f(\lambda^{-1}_1m_1+\lambda_2^{-1}m_2) \leq p(\lambda^{-1}_1m_1+\lambda_2^{-1}m_2) \leq p(\lambda^{-1}_1m_1 -x_0)+p(\lambda_2^{-1}m_2+x_0).$$
Therefore,
$$f(\lambda^{-1}_1m_1)- p(\lambda^{-1}_1m_1 -x_0) \leq p(\lambda_2^{-1}m_2+x_0) - f(\lambda_2^{-1}m_2)$$
for all $m_1, m_2 \in M$ and $0< \lambda_1 < \infty$, $0< \lambda_2 < \infty$. Hold $m_2$ and $\lambda_2$ fixed and take the supremum
over $m_1 \in M$ and $0<\lambda_1 < \infty$. Then for each $m_2 \in M$ and $0< \lambda_2 < \infty$ we have that:
$$\sup_{\stackrel{m \in M}{0<\lambda<\infty}}\left(f(\lambda^{-1}m)-p(\lambda^{-1}m-x_0)\right) \leq
p(\lambda_2^{-1}m_2 +x_0)-f(\lambda_2^{-1}m_2).$$
Now we take the infimum over $m_2 \in M$ and $0< \lambda_2< \infty$. Then,
$$a := \sup_{\stackrel{m \in M}{0<\lambda<\infty}}\left(f(\lambda^{-1}m)-p(\lambda^{-1}m-x_0)\right)
\leq \inf_{\stackrel{m \in M}{0<\lambda<\infty}}\left(p(\lambda^{-1}m+x_0) - f(\lambda^{-1}m)\right) =:b.$$
Choose $\alpha^* \in [a,b]$. Then from the left-hand side of the equation we get that:
$$f(m) +(-\lambda )\alpha^* \leq p(m+(-\lambda) x_0) \mbox{\quad for all $m \in M$ and $0<\lambda<\infty$.}$$
From the right-hand side of the equation we get that:
$$f(m) +\lambda\alpha^* \leq p(m+\lambda x_0) \mbox{\quad for all $m \in M$ and $0<\lambda<\infty$.}$$
From these two equations we see that:
$$F^*(x) := F_{\alpha^*}(x) \leq p(x) \mbox{\quad for all $x \in M^*$.}$$
That is, $(M,F) < (M^*,F^*) \in \mathscr{P}$. 
\end{proof}

We now give some applications of this famous theorem.

\begin{corollary}\label{HB in norm} Let $Y$ be a subspace of a normed linear space $(X,\norm)$ (over $\mathbb{R}$). If
$f \in Y^*$ then there exists an $F \in X^*$ such that $F|_Y =f$ and $\|F\|=\|f\|$. 
\end{corollary}

\begin{proof} Consider the sublinear functional $p:X \to \R$ defined by, $p(x) := \|f\|\|x\|$. Then $f(y) \leq p(y)$ for all $y \in Y$. 
By the Hahn-Banach Theorem, (Theorem \ref{Hahn-Banach}) there exists a linear functional $F:X \to \R$ such that
$F|_Y =f$ and $F(x) \leq p(x)$ for all $x \in X$. Therefore, $-F(x) = F(-x) \leq p(-x) = p(x)$ for all $x \in X$ too.  Thus,
$|F(x)| \leq p(x)$ for all $x \in X$. This in turn implies that $\|F\| \leq \|f\|$.  On the other hand, since $F$ is an extension of $f$,
we must also have that $\|f\| \leq \|F\|$.
\end{proof}

\medskip

\begin{corollary} \label{HB} Let $(X,\|\cdot\|)$ be a normed linear space. For every 
$x \in X \setminus \{0\}$ there exists an $f \in S_{X^*}$ such that $f(x) = \|x\|$.
\end{corollary}

\begin{proof} Let $Y := \mbox{span}\{x\}$ and define $f \in Y^*$ by, $f(\lambda x) := \lambda \|x\|$.
Clearly, $\|f\|=1$ and $f(x) = \|x\|$. By Corollary \ref{HB in norm} there exists an $F \in X^*$
such that $\|F\| = \|f\| =1$ and $F|_Y = f$. Therefore, in particular we have that $F(x) = f(x) = \|x\|$. 
\end{proof}

\begin{proposition}\label{subspace topology} Let $Y$ be a subspace of a normed linear space $(X,\norm)$.  Then the topology $\sigma(Y^*,Y)$ on $Y$ coincides with the relative $\sigma(X^*,X)$ topology
on $Y$.   
\end{proposition}
\begin{proof} Let us first show that every relatively $\sigma(X^*,X)$-open set in $Y$ is $\sigma(Y^*,Y)$ open.  To this end, let $U$ be a relatively $\sigma(X^*,X)$-open set in $Y$.
Let $y \in U$. By Proposition \ref{relative-topology}, there exists a finite set $\{x^*_1, x^*_2, \ldots, x_n^*\} \subseteq X^*$ and an $\varepsilon >0$ such that $N_X(y,\{x^*_1, x^*_2, \ldots, x_n^*\}, \varepsilon) \cap Y \subseteq U$. For each $1 \leq k \leq n$ let $y_k^* := x_k^*|_Y$.  Then 
$$N_Y(y, \{y^*_1, y^*_2, \ldots, y_n^*\}, \varepsilon)= N_X(y,\{x^*_1, x^*_2, \ldots, x_n^*\}, \varepsilon) \cap Y   \subseteq U.$$  
Thus, by Proposition \ref{coincide}, $U$ is $\sigma(Y^*,Y)$-open.
Now, suppose that $U$ is a $\sigma(Y^*,Y)$-open subset of $Y$.  Then, by Proposition \ref{coincide}, there exists a finite set $\{y^*_1, y^*_2, \ldots, y_n^*\} \subseteq Y^*$ and an $\varepsilon >0$ 
such that $N_Y(y,\{y^*_1, y^*_2, \ldots, y_n^*\}, \varepsilon) \subseteq U$. By Corollary \ref{HB in norm}, for each $1 \leq k \leq n$, there exists an $x^*_k \in X^*$ such that $x^*_k|_Y = x^*_k$.
Then 
$$N_X(y,\{x^*_1, x^*_2, \ldots, x_n^*\}, \varepsilon) \cap Y  = N_Y(y,\{y^*_1, y^*_2, \ldots, y_n^*\}, \varepsilon) \subseteq U.$$  
Therefore, by Proposition \ref{relative-topology}, $U$ is open in the relative $\sigma(X^*,X)$-topology on $Y$.
\end{proof}

Next we will show how to use the Hahn-Banach Theorem to obtain some geometric properties of locally
convex spaces. 

\medskip

Let $S$ be a nonempty subset of a vector space $V$. We shall say that a point $x \in S$ is a
{\it core point} of $S$ if for every $v \in V$ there exists a $0<\delta < \infty$ such that
$x+\lambda v \in S$ for all $0 \leq \lambda < \delta$.  The set of all core points of $S$ is called the {\it core} of $S$ and is denoted by
{\it $\mathrm{Cor}(S)$}.

\medskip

Let $C$ be a convex set in a vector space $V$ with $0 \in \mathrm{Cor}(C)$. Then the functional
$\mu_C:V \rightarrow \mathbb{R}$ defined by, 
$$\mu_C(x) := \inf\{\lambda >0:x \in \lambda C\}$$
is called the {\it Minkowski functional} generated by the set $C$.

\medskip

\begin{theorem}\label{Core} Let $C$ be a convex subset of a vector space $V$ with $0$ in the core of $C$.
Then $\mu_C:V \rightarrow \mathbb{R}$ is a sublinear functional. Moreover, 
$$\{x \in V: \mu_C(x) <1\} \subseteq C \subseteq \{x \in V:\mu_C(x) \leq 1\}.$$
\end{theorem}

\begin{proof} Given $\alpha >0$ and $\lambda >0$, clearly $x \in \lambda C$ if, and only if,
$\alpha x \in \lambda \alpha C$. Therefore, $\mu_C(\alpha x) = \alpha \mu_C(x)$ and thus $\mu_C$
is {\it positively homogeneous}. We claim that $\mu_C$ is {\it subadditive}; that is, $\mu_C(x+y) \leq
\mu_C(x) +\mu_C(y)$. Fix any $s>\mu_C(x)$ and $t > \mu_C(y)$. We have that there is some $s_0 <s$ such
that $x \in s_0C$. Note that $s_0C \subseteq sC$. Indeed, $0 \in sC$ and if $c \in C$, then by the convexity
of $sC$, 
$$s_0c = \frac{s_0}{s}\left(sc\right) + \left(1 - \frac{s_0}{s}\right)\!0 \in sC.$$
We see that $x \in sC$ and similarly $y \in tC$. Then $x+y \in sC+tC$ and thus by the convexity of $C$,
$$x+y \in (s+t)\!\left(\frac{s}{s+t}C+\frac{t}{s+t}C\right) \subseteq (t+s)C.$$
Therefore, $\mu_C(x+y) \leq s+t$ and so by the choice of $s$ and $t$ we have that
$\mu_C(x+y) \leq \mu_C(x) + \mu_C(y)$.

\medskip

If $\mu_C(x) <1$ then $x \in \lambda C$ for some $0<\lambda <1$ and so $(1/\lambda)x \in C$. Since
$0 \in C$ and $C$ is convex, 
$$x = \lambda \left(\frac{x}{\lambda}\right) + \left(1 - \lambda\right)0 \in C.$$
If $x \in C$ then $\mu_C(x) \leq 1$ by the definition of the Minkowski functional. 
\end{proof}

\begin{remark}\label{not in} If the set $C$ in Theorem \ref{Core} is a closed and convex subset of a topological vector space $(V,\tau)$, with $0 \in \mathrm{Cor}(C)$ and 
$x_0 \not\in C$ then it is an easy exercise to show that $1 < \mu_C(x_0)$.
\end{remark}

We now give the geometric version of the Hahn-Banach Theorem.

\begin{theorem}[Separation Theorem]\label{separation-theorem} Suppose that $(X,\tau) $ is a locally convex space over $\R$ and $C$ is a nonempty
closed convex subset of $X$. If $x_0 \not\in C$ then there exists a continuous linear functional $x^*$ on $X$ such that 
$$\sup \{x^*(c):c \in C\}< x^*(x_0).$$
\end{theorem}

\begin{proof}  We may assume, without loss of generality, that $0 \in C$; because otherwise we would consider
$C-x$ and $x_0-x$ for some $x \in C$.  Since vector addition is continuous and $x_0 + 0 \not\in C$ there exist convex
open neighbourhoods $U$ of $x_0$ and $V$ of $0$ such that $(U+V) \cap C = \varnothing$. Thus, $U \cap [C +(-V)] = \varnothing$.
Now, $-V$ is also a convex open neighbourhood of $0$ and so $C + (-V)$ is a convex open set containing the set $C$ and disjoint from $U$.
Let $D:= \overline{C +(-V)}$, then $D$ is a closed and convex set with $0 \in \mathrm{int}(D)$ and $x_0 \not\in D$.
Let $\mu_D$ be the Minkowski functional for $D$.
Since $D$ is closed and $x_0 \not\in D$ we have $\mu_D(x_0) >1$ (see Remark \ref{not in}). Define a linear functional on $\mbox{span}\{x_0\}$ by,
$f(\lambda x_0) := \lambda \mu_D(x_0)$. Then on $\mbox{span}\{x_0\}$ we have that $f(\lambda x_0) \leq \mu_D(\lambda x_0)$.
Indeed, for $0 \leq \lambda$ it is clear from the definition of $f$; whereas for $\lambda <0$ we have
$f(\lambda x_0) = \lambda \mu_D(x_0) <0$ while $\mu_D(\lambda x_0) \geq 0$. By using the Hahn-Banach Theorem we may extend $f$ onto $X$ so that
$f(x) \leq \mu_D(x)$ for all $x \in X$. If $x \in D$ then $\mu_D(x) \leq 1$ and thus, $f(x) \leq \mu_D(x) \leq 1$.
Since $D$ contains a neighbourhood of the origin we have that $f$ is a bounded on a neighbourhood of $0$ and so by Proposition \ref{cont-equivalence}, 
$f \in X^*$. Since $f(x_0) = \mu_D(x_0) > 1$ we get that $\sup\{f(x):x \in C\} \leq \sup\{f(x):x \in D\} \leq 1 < f(x_0)$.
\end{proof}

An immediate consequence of the Separation Theorem is the following result, which is sometimes known as Mazur's Theorem.

\begin{proposition}\label{weakly-closed} 
Let $C$ be a closed convex subset of a normed linear space $(X,\norm)$. Then $C$ is also closed with respect to the weak topology on $X$.
\end{proposition}
\begin{proof}
If $C$ is empty or the whole space, then $C$ is weakly closed, so let us suppose otherwise. 
Let $x_{0}\in X \setminus C$. Since $C$ is closed and convex, we have, by the Separation Theorem (Theorem \ref{separation-theorem}), the existence of an $f_{x_{0}}\in X^{*}$ such that
 $f_{x_{0}}(x_{0})>\sup_{x\in C}f_{x_{0}}(x)$. Thus, $x_{0}\in f_{x_{0}}^{-1}\big(\!\left(\sup_{x\in C}f_{x_{0}}(x), \infty\right)\!\big)$, 
 which, being the inverse image of an open set, is weakly open. It is then straightforward to check that
$X\setminus C=\bigcup_{x_{0}\in X\setminus C}f_{x_{0}}^{-1}\big((\sup_{x\in C} f_{x_{0}}(x), \infty)\big)$.
Hence, $X\setminus C$, being the union of weakly open sets, is weakly open. Thus, $C$ is weakly closed. 
\end{proof}


\subsection{Weak$^*$ topology}


Let $(X,\norm)$ be a normed linear space. For each $x \in X$ we define, $\widehat{x} \in X^{**} := (X^{*})^{*}$ by, 
$\widehat{x}(x^*):= x^{*}(x)$ for all $x^* \in X^*$. To show that $\widehat{x}$ is really in $X^{**}$ we must first check that it is
linear and then check that it is continuous. So suppose that $x^*$ and $y^*$ are in $X^*$, then
$$\widehat{x}(x^{*}+y^{*}) =  (x^{*}+y^{*})(x) = x^{*}(x) + y^{*}(x) = \widehat{x}(x^{*}) + \widehat{x}(y^{*}).$$
Also, for any  $\lambda \in \mathbb{R}$ and $x^* \in X^*$ we have that
$$\widehat{x}(\lambda x^*) = (\lambda x^*)(x) = \lambda x^*(x) = \lambda \widehat{x}(x^*).$$
Now, $|\widehat{x}(x^{*})| = |(x^{*})(x)| \leq \|x^{*}\| \cdot \|x\|$. Therefore, $\|\widehat{x}\| \leq \|x\|$ and so $\widehat{x}\in X^{**}$.

\medskip

\begin{proposition} Let $(X, \norm)$ be a normed linear space.  Then the mapping $x \mapsto \widehat{x}$ is a linear isometry from $X$ into $X^{**}$.  
\end{proposition}

\begin{proof}  The mapping $x \mapsto \widehat{x}$ from $X$ into $X^{**}$ is linear, since for all  $x^{*} \in X^{*}$
$$\widehat{(x+y)}(x^{*})  =  x^{*}(x+y)= x^{*}(x) + x^{*}(y)  = \widehat{x}(x^{*}) + \widehat{y}(x^{*}).$$
Therefore, $\widehat{x+y} = \widehat{x} + \widehat{y}$. Also, for any  $\lambda \in \mathbb{R}$ and $x^* \in X^*$,
$$\widehat{(\lambda x)}(x^{*})  =  x^{*}(\lambda x) = \lambda x^{*}(x) =  \lambda \widehat{x}(x^{*}).$$
Therefore, $\widehat{(\lambda x)} = \lambda \widehat{x}$.  Next we show that $x \mapsto \widehat{x}$ is an isometry.
For each $x \in X$, we have by Corollary \ref{HB}, a linear function 
$x^{*} \in S_{X^*}$ such that $x^{*}(x) = \|x\|$. Therefore, 
$\|\widehat{x}\|  \geq \frac{|\widehat{x}(x^{*})|}{\|x^{*}\|}  =  |\widehat{x}(x^{*})| 
= |x^{*}(x)|= \|x\|$.
\end{proof} 

\medskip

If $(X,\norm)$ is a Banach space then $\widehat{X}$ is a closed subspace of $X^{**}$ where $\widehat{X}$ is defined as 
$\{\widehat{x} : x \in X\}$. We call $\widehat{X}$ the {\it natural embedding} of $X$ into $X^{**}$ and we
call $x \mapsto \widehat{x}$ from $X$ into $X^{**}$ the {\it natural embedding mapping}.

\medskip

An important topology for our concerns is the weak$^*$ topology.  Suppose that $(X,\norm)$ is a normed linear space. Then we call the topology $\sigma(\widehat{X}, X^*)$ on $X^*$, 
the {\it weak$^*$ topology} on $X^*$ and we write $(X^*, \mathrm{weak}^*)$  for $(X^*,\sigma(\widehat{X}, X^*))$. It follows form Proposition \ref{cont-equivalence} that $F \in X^{**}$ is weak$^*$ continuous if, and only if, $F \in \widehat{X}$.

\medskip

Let $(X, \norm)$ be a normed linear space and let $A\seq X$. We define the {\it (upper) polar of $A$} to be the subset  $A^\circ$ of $X^*$ defined by 
$$A^\circ:=\lbrace x^*\in X^*: x^*(a)\leq 1 \ \text{for all} \ a\in A\rbrace.$$ 

\begin{proposition}\label{PolarProperties}
Let $(X, \norm)$ be a normed linear space and let $A\seq X$. Then $A^\circ$ is convex, weak*-closed, and contains $0$.
\end{proposition}

\begin{proof}
	The fact that $0\in A^\circ$ is trivial. To see that $A^\circ$ is weak*-closed and convex, note that $A^\circ=\bigcap_{a\in A} \widehat{a}^{-1}(-\infty, 1]$ is the intersection of weak*-closed and convex sets, 
	and so is itself, weak*-closed and convex.
\end{proof}

There are many interesting properties of polars that can be easily verified.  For example, (i) if $A \subseteq B$ then $B^\circ \subseteq A^\circ$, (ii) $(B_X)^\circ = B_{X^*}$,
(iii) for any $r>0$, $(rA)^\circ = r^{-1}A^\circ$.   By combining these we see that if $0 \in \mathrm{int}(A)$ then $A^\circ$ is bounded and if $A$ is bounded then $0 \in \mathrm{int}(A^\circ)$. 

\medskip

There is also a dual version of (upper) polars. Let $(X, \norm)$ be a normed linear space and let $A\seq X^*$. We define the {\it (lower) polar of $A$} to be the subset  $A_\circ$ of $X$ defined by 
$$A_\circ:=\lbrace x\in X: a^*(x) \leq 1 \ \text{for all} \ a^* \in A\rbrace.$$ 
There are many interesting relationships between these polars. For example, for any $\varnothing \not= A \subseteq X$, $(A^\circ)_\circ = \overline{\mathrm{co}}(A \cup \{0\})$ and
for any $\varnothing \not= B \subseteq X^*$, $(B_\circ)^\circ = \overline{\mathrm{co}}^{w^*}(B \cup \{0\})$.

\medskip

Perhaps the most famous theorem concerning polars is the following theorem.

\begin{theorem}[Bipolar Theorem] Let $C$ be a closed, convex subset of a normed linear space $(X, \norm)$ with $0 \in C$.  Then $C^{\circ\circ} := (C^\circ)^\circ = \overline{\widehat{C}}^{w^*}$.
\end{theorem}

\begin{proof}
	It follows directly from the definition of $(C^\circ)^\circ$ that $\widehat{C}\seq C^{\circ\circ}$. Moreover, by Proposition \ref{PolarProperties} we know that $C^{\circ\circ}$ is weak$^*$-closed 
	and so $\overline{\widehat{C}}^{w^*}\seq C^{\circ\circ}$.
	Now suppose, in order to obtain a contradiction, that  $\overline{\widehat{C}}^{w^*}\subsetneq C^{\circ\circ}.$ Then there exists an $F_0\in C^{\circ\circ}\setminus \overline{\widehat{C}}^{w^*}$.
	By Theorem \ref{separation-theorem}, applied in $(X^{**},\mathrm{weak}^*)$, there exists an $x^* \in X^*$ such that 
	$$0 \leq \sup \{F(x^*): F \in \overline{\widehat{C}}^{w^*}\} =  \sup \{\widehat{x^*}(F):F \in \overline{\widehat{C}}^{w^*}\} < \widehat{x^*}(F_0) =F_0(x^*).$$
	If necessary, we may replace $x^*$ by $\lambda x^*$, (for some $0<\lambda$ and relabelling), so that 
	$$\sup\{x^*(c):c \in C\} = \sup\{\widehat{c}(x^*): \widehat{c} \in \widehat{C}\} \leq \sup \{F(x^*): F \in \overline{\widehat{C}}^{w^*}\} \leq 1 <  F_0(x^*).$$
	Therefore, $x^* \in C^\circ$. However, this implies that $F_0(x^*) \leq 1$ since $F_0 \in C^{\circ\circ}$, which contradicts the earlier inequality: $1 <  F_0(x^*)$.
     \end{proof}
	
An important application of the Bipolar Theorem is given next.

\begin{corollary}[Goldstine's Theorem] \label{Goldstine's Theorem} Let $(X, \norm)$ be a normed linear space then $B_{\widehat{X}}$ is weak$^*$ dense in $B_{X^{**}}$.
\end{corollary}
\begin{proof} We apply twice, the general fact (observed before) that if $B_Y$ is the closed unit ball of a normed linear space $Y$  then $B_{Y^*} = (B_Y)^\circ$ to obtain  
$$B_{X^{**}} = (B_{X^*})^\circ = ((B_X)^\circ)^\circ = (B_X)^{\circ\circ}$$ 
and then apply the Bipolar Theorem. 
\end{proof}

Perhaps the main reason for the interest in the weak$^*$ topology is contained in the next theorem.  It says that, although it is too much to ask that the dual ball  be compact with respect to the norm topology
(unless the space is finite dimensional), it is possible that it is compact with respect to a weaker topology.
 
\begin{theorem}[Banach-Alaoglu Theorem \cite{BanachAlaoglu}]\label{Banach-Alaoglu}  Let $(X, \norm)$ be a normed linear space.  Then $(B_{X^*},\mbox{weak$^*$})$ is compact.
\end{theorem}
\begin{proof} For each $x \in X$, let $I_x := [-x,x]$ and let $Y := \prod_{x \in X} I_x$ be endowed with the product topology.  By Tychonoff's Theorem, $Y$ is compact. 
It follows from the definition of the product topology and 
Proposition \ref{coincide} that $\pi:B_{X^*} \to Y$, defined by, $\pi(x^*)(x) := x^*(x)$ for all $x \in X$, is a homeomorphic embedding of $B_{X^*}$ into $Y$.  So to show that $(B_{X^*}, \mathrm{weak^*})$
is compact it is sufficient to show that $\pi(B_{X^*})$ is a closed subset of $Y$, that is, it is sufficient to show that $\overline{\pi(B_{X^*})} \subseteq \pi(B_{X^*})$.  To this end, let $g \in \overline{\pi(B_{X^*})}$.
We will show that $g$ is ``linear''.  Let $x,y \in X$ and $\varepsilon >0$. Then there exists $x^* \in B_{X^*}$ such that $|g(x) - \pi(x^*)(x)| < \varepsilon/3$, $|g(y) - \pi(x^*)(y)| < \varepsilon/3$ and
$|g(x+y) - \pi(x^*)(x+y)| < \varepsilon/3$.  Then, since $x^*$ is linear,  
\begin{eqnarray*} 
\big|g(x+y) - [g(x)+g(y)]\big| &=& \big|[g(x+y) - \pi(x^*)(x+y)] + \pi(x^*)(x+y)- [g(x)+g(y)]\big| \\
&=&  \big|[g(x+y) - \pi(x^*)(x+y)] +[\pi(x^*)(x)-g(x)]+[\pi(x^*)(y)-g(y)\big | \\
&\leq& |g(x+y) - \pi(x^*)(x+y)| + |\pi(x^*)(x)-g(x)|+|\pi(x^*)(y)-g(y)|\\
&\leq& \varepsilon/3 + \varepsilon/3 + \varepsilon/3 = \varepsilon.
\end{eqnarray*}
Since $\varepsilon >0$ was arbitrary, $g(x+y) = g(x) + g(y)$.
Next, let $x \in X$,  $\lambda \in \R$ and $\varepsilon >0$. Then there exists $x^* \in B_{X^*}$ such that $|g(\lambda x)-\pi(x^*)(\lambda x)| < \varepsilon/2$ and $|g(x) - \pi(x^*)(x)|  < \varepsilon/2(|\lambda|+1)$. Then, since $x^*$ is linear,
\begin{eqnarray*} 
|g(\lambda x) - \lambda g(x)| &=& \big|[g(\lambda x) - \pi(x^*)(\lambda x)] + [\pi(x^*)(\lambda x) - \lambda g(x)]\big| \\
&=& \big|[g(\lambda x) - \pi(x^*)(\lambda x)] + [\lambda \pi(x^*)(x) -\lambda g(x)]\big| \\
&\leq& |g(\lambda x) - \pi(x^*)(\lambda x)|+ |[\lambda \pi(x^*)(x) -\lambda g(x)]|  \\
&\leq& \varepsilon/2 + |\lambda| |\pi(x^*)(x) -g(x)| < \varepsilon/2 + \varepsilon/2 = \varepsilon.
\end{eqnarray*}
Since $\varepsilon >0$ was arbitrary, $g(\lambda x) = \lambda g(x)$. Thus, if we define $y^*:X \to \R$ by, $y^*(x) := g(x)$ for all $x \in X$, then $y^* \in B_{X^*}$ and $g = \pi(y^*) \subseteq \pi(B_{X^*})$.
\end{proof}


\begin{proposition}\label{relative-weak-and-weak*} Let $(X,\norm)$ be a normed linear space.  Then the relative weak topology and the relative weak$^*$ topology coincide on the subspace $\widehat{X}$ of $X^{**}$.
\end{proposition}
\begin{proof} It follows immediately from the definitions that each relatively weak$^*$ open subset of $\widehat{X}$ is open in the relative weak topology on $\widehat{X}$.  
So we need only consider the converse statement.  Suppose that $U$ is a relatively weak open subset of $\widehat{X}$.  Let $\widehat{x}$ be any element of $U$. 
Then by, Proposition \ref{relative-topology} there exists a finite subset $\{\mathscr{F}_1, \mathscr{F}_2, \ldots,\mathscr{F}_n\}$ of $X^{***}$ and an $\varepsilon >0$ such that 
$N(\widehat{x}, \mathscr{F}_1, \mathscr{F}_2, \ldots, \mathscr{F}_N, \varepsilon) \cap \widehat{X} \subseteq U$.
For each $1 \leq k \leq n$ let $f_k:X \to \R$ be defined by, $f_k(x) := \mathscr{F}_k(\widehat{x})$.  Then $f_k \in X^*$, in fact $\|f_k\| \leq \|\mathscr{F}_k\|$ for each $1 \leq k \leq n$. We claim that
$$N(\widehat{x}, \widehat{f_1}, \widehat{f_2}, \ldots, \widehat{f_n}, \varepsilon) \cap \widehat{X} \subseteq  N(\widehat{x}, \mathscr{F}_1, \mathscr{F}_2, \ldots, \mathscr{F}_N, \varepsilon) \cap \widehat{X}  \subseteq U.$$
To see this, let $F \in N(\widehat{x}, \widehat{f_1}, \widehat{f_2}, \ldots, \widehat{f_n}, \varepsilon) \cap \widehat{X}$. Then $F = \widehat{y}$ for some $y \in X$ and $\widehat{y} \in N(\widehat{x}, \widehat{f_1}, \widehat{f_2}, \ldots, \widehat{f_n}, \varepsilon)
\cap \widehat{X}$.
Fix $1 \leq k \leq n$. Then
$$|\mathscr{F}_k(F)-\mathscr{F}_k(\widehat{x})| = |\mathscr{F}_k(\widehat{y})-\mathscr{F}_k(\widehat{x})| = |f_k(y)-f_k(x)| = |\widehat{f_k}(\widehat{y})-\widehat{f_k}(\widehat{x})| = 
|\widehat{f_k}(F)-\widehat{f_k}(\widehat{x})| < \varepsilon.$$
Therefore, $F \in N(\widehat{x}, \mathscr{F}_1, \mathscr{F}_2, \ldots, \mathscr{F}_N, \varepsilon) \cap \widehat{X}$; which completes the proof of the claim.  The result now follows from Proposition \ref{relative-topology}.
\end{proof} 

In what follows we will often use (without saying) the fact that if $(Z, \tau)$ is a topological space and $A \subseteq Y \subseteq Z$, then $A$ is compact in $Z$ if, and only if, $A$ is compact in $Y$, with respect to the relative topology on $Y$.

\begin{remark}\label{reflex-remark} Together, Theorem \ref{Banach-Alaoglu} and Proposition \ref{relative-weak-and-weak*} are essential for our future endeavours, as they provide a method for showing that a closed and bounded convex
subset $C$ of a normed linear space $(X,\norm)$ is weakly compact. Namely, to show that $C$ is weakly compact it is sufficient (and necessary) to show that $\overline{\widehat{C}}^{w^*} \subseteq \widehat{X}$.
The reason for this is as follows: $\overline{\widehat{C}}^{w^*}$ is a weak$^*$ compact subset of $(X^{**},\mathrm{weak}^*)$, by Theorem \ref{Banach-Alaoglu}, and hence compact with respect to the
relative weak$^*$ topology on $\widehat{X}$.  Therefore, by Proposition \ref{relative-weak-and-weak*}, $\overline{\widehat{C}}^{w^*}$ is compact with respect to the relative weak topology on $\widehat{X}$. Further, 
by Proposition \ref{subspace topology}, $\overline{\widehat{C}}^{w^*}$ is compact with respect to the $\sigma((\widehat{X})^*, \widehat{X})$-topology on $\widehat{X}$ (i.e., the weak topology on $\widehat{X}$).
Let $j:X \to \widehat{X}$ be the linear isometry defined by, $j(x) := \widehat{x}$ for all $x \in X$. Then $j^{-1}:\widehat{X} \to X$ is also a linear isometry. Thus, by Proposition \ref{linearmap}, $C \subseteq j^{-1}(\overline{\widehat{C}}^{w^*})$ is
weakly compact.  Since $C$ is closed and bounded it is closed in the weak topology on $X$ (see, Proposition \ref{weakly-closed}).  Hence $C$ is compact with respect to the weak topology on $X$. 
\end{remark}

As an example of this approach, we will give our first characterisation of reflexivity in terms of the weak compactness of the unit ball.

\begin{theorem}[\!\!\cite{DunfordSchwartz1}]\label{reflexive} Let $(X,\norm)$ be a normed linear space. Then $X$ is reflexive (i.e., $X^{**} = \widehat{X}$) if, and only if, $B_X$ is compact with respect to the weak topology on $X$.
\end{theorem}
\begin{proof} Suppose that  $B_X$ is compact with respect to the weak topology on $X$. Then, by Proposition \ref{linearmap}, $B_{\widehat{X}}$ is compact with respect to the weak topology on $X^{**}$
as: (i) the mapping, $x \mapsto \widehat{x}$, is a bounded linear operator from $X$ into $X^{**}$ and (ii) the general fact that the continuous image of a compact set is compact. Now, since the weak$^*$
topology on $X^{**}$ is weaker (and certainly no stronger) than the weak topology on $X^{**}$,  $B_{\widehat{X}}$ is compact with respect to the weak$^*$ topology on $X^{**}$.  Furthermore, since the
weak$^*$ topology is Hausdorff, $B_{\widehat{X}}$ is closed with respect to the weak$^*$ topology on $X^{**}$. Thus, by Goldstine's Theorem, (Theorem \ref{Goldstine's Theorem}) 
$$B_{\widehat{X}} = \overline{B_{\widehat{X}}}^{w^*} = B_{X^{**}}.$$
So, $X^{**} = \bigcup_{n \in \N} nB_{X^{**}} = \bigcup_{n \in \N} nB_{\widehat{X}} = \widehat{X}$.  

\medskip

Conversely, suppose that $X^{**} = \widehat{X}$. Then $B_{X^{**}} = B_{\widehat{X}}$ and so by Theorem \ref{Banach-Alaoglu}, $(B_{\widehat{X}}, \mathrm{weak}^*)$ is compact.  Since
$B_{\widehat{X}} \subseteq \widehat{X}$ we have by Proposition \ref{relative-weak-and-weak*} that $(B_{\widehat{X}}, \mathrm{weak})$ is compact. Finally, since $x \mapsto \widehat{x}$ is a
linear isometry from $X$ onto $X^{**}$ (since we are assuming that $X^{**} = \widehat{X}$), its inverse is a continuous linear operator (in fact an isometry as well) and so by Proposition \ref{linearmap},
$(B_X,\mathrm{weak})$ is compact too, as the continuous image of a compact set is compact.
\end{proof}


\section{James' Theorem on weak compactness }


In this section we will provide three proofs of James' Theorem on weak compactness, listing them in order of increasing generality.  First we provide a proof that is valid in all separable Banach spaces, then 
we give a proof that is valid in any Banach space whose dual ball is weak$^*$ sequentially compact, and then finally, we will present a proof that holds in all Banach spaces.  It is our hope that this
incremental approach to the full James' Theorem will make the final proof more accessible and less intimidating to the reader.


\subsection{James' Theorem on weak compactness: the separable case}


Convexity is the key to all our proofs of James' Theorem.

\medskip

Let $A$ be a nonempty convex subset of a vector space $V$  and let $\varphi:A \rightarrow\R$ be a function. We say that $\varphi$ is {\it convex} if 
		$$\varphi(\lambda x+(1-\lambda)y)\leq \lambda\varphi(x)+(1-\lambda)\varphi(y)$$
		for all $x,y\in A$ and all $0\leq\lambda\leq 1$.

\begin{lemma}[\!\! \cite{MoorsWhite}]\label{convex}
	Let $0< \beta$, $0< \beta'$ and suppose that $\varphi:[0,\beta+\beta']\rightarrow\mathbb{R}$ is a convex function.	Then
$$ \frac{\varphi(\beta)-\varphi(0)}{\beta} \leq \frac{\varphi(\beta+\beta')-\varphi(\beta)}{\beta'}.$$
\end{lemma}	
\begin{proof}
	The inequality given in the statement of the lemma follows by rearranging the inequality \\
\hspace*{3.5cm} $\displaystyle \varphi(\beta) \leq \frac{\beta}{\beta + \beta'} \varphi(\beta + \beta') + \frac{\beta'}{\beta + \beta'}\varphi(0).$  \hspace*{2.0cm}
\end{proof}

 Our first application of convexity is given next.  It plays an important role in all three proofs of James' theorem.
 
\begin{lemma}[\!\!\cite{MoorsWhite}]\label{increasing}
	Let $V$ be a vector space (over $\mathbb{R}$) and let $\varphi:A\rightarrow\mathbb{R}$ be a convex function defined on a convex set $A$ with $0 \in A$.
If $(A_n:n\in\n)$ is a decreasing sequence of nonempty, convex subsets of $V$, $(\beta_n:n \in \n)$ is any sequence of strictly positive numbers such that $(\sum_{n=1}^\infty \beta_n)A_1 \subseteq A$, $r\in \mathbb{R}$ and 
$$\beta_1r+\varphi(0) < \inf_{a\in A_1}\varphi(\beta_1a),$$ then there exists a sequence $(a_n:n \in \n)$ in $V$ such that, for all $n\in\n$:
\begin{enumerate}
\item[{\rm (i)}] $a_n\in A_n$ and 
\item[{\rm (ii)}] $\displaystyle \varphi($\mbox{$\sum_{i=1}^n\beta_ia_i$}$)+\beta_{n+1}r<\varphi($\mbox{$\sum_{i=1}^{n+1}\beta_ia_i)$}.
\end{enumerate}
\end{lemma}
\begin{proof}
	We proceed in two parts. Firstly we prove that if  $\beta_nr+\varphi(u)\!<\! \inf_{a\in A_n}\!\varphi(u+\beta_na)$ 
for some $n\in\n$ and some $u\in (\sum_{i=0}^{n-1} \beta_1)A_1$, where $\beta_0 :=0$, then there exists an $a_n\in A_n$, such that $$\beta_{n+1}r+\varphi(u+\beta_na_n) < \inf_{a\in A_n}\varphi(u+\beta_na_n+\beta_{n+1}a).$$
To see this, suppose that $u\in (\sum_{i=0}^{n-1} \beta_1)A_1$ and that $\beta_nr+\varphi(u)  < \inf_{a\in A_n}\varphi(u+\beta_na)$.
Then there exists an $\epsilon>0$ such that 
$$r+2\epsilon < \frac{\inf_{a\in A_n}\varphi(u+\beta_na)-\varphi(u)}{\beta_n}. \mbox{\hspace*{3.0cm} $(*)$}$$
So, choose $a_n\in A_n$ such that 
$\varphi(u+\beta_na_n)<\inf_{a\in A_n}\varphi(u+\beta_na)+\beta_{n+1}\epsilon$.
Let $a\in A_n$. Then \\ 
$v:=(\beta_na_n+\beta_{n+1}a)/(\beta_n+\beta_{n+1})\in A_n$ (since $A_n$ is convex) and so,
\begin{eqnarray*}
r +2\epsilon &<& \frac{\varphi(u+\beta_nv)-\varphi(u + 0v)}{\beta_n} \mbox{\quad \quad (by $(*)$ and the fact that $v \in A_n$)}    \\
&\leq& \frac{\varphi(u+(\beta_n+\beta_{n+1})v)-\varphi(u+\beta_nv)}{\beta_{n+1}}. \mbox{\quad \quad (by Lemma \ref{convex}.)}
\end{eqnarray*}
Rearranging gives  $$\beta_{n+1}(r +\epsilon)+[\varphi(u+\beta_nv) + \beta_{n+1}\epsilon]  < \varphi(u+\beta_na_n+\beta_{n+1}a),$$ for all $a\in A_n$.  Since
$\varphi(u+\beta_na_n)< [\varphi(u+\beta_nv)+\beta_{n+1}\epsilon]$, the desired inequality follows. \\ \\
From this, we may inductively construct a sequence $(a_n:n \in\n)$ with the requisite properties (i) and (ii). 
For the first step, we set $u:=0$ and then, by  hypothesis, we have that 
$$\beta_1r+\varphi(0) < \inf_{a\in A_1}\varphi(\beta_1a) = \inf_{a\in A_1}\varphi(0+\beta_1a).$$ 
So, by the first result, there exists an $a_1\in A_1$, such that $\displaystyle \beta_2r+\varphi(\beta_1a_1)<\inf_{a\in A_1} \varphi(\beta_1a_1+\beta_2a)$. \\ \\
For the $n^{\rm{th}}$ step, set $u:=\sum_{i=1}^{n-1}\beta_i a_i$. Since $A_n\subseteq A_{n-1}$, and by the way $a_{n-1}$ was constructed, we have that 
$$\beta_nr + \varphi(u) <   \inf_{a\in A_{n-1}}\varphi(u+\beta_n a) \leq \inf_{a\in A_n}\varphi(u+\beta_n a).$$
So, by the first result again, there exists $a_n\in A_n$, such that 
$\beta_{n+1}r + \varphi\left(\mbox{$\sum_{i=1}^{n}\beta_i a_i$}\right) <  \inf_{a\in A_n} \varphi\left(\mbox{$\sum_{i=1}^{n}\beta_i a_i+\beta_{n+1}a$}\right)$
which completes the induction. The sequence $(a_n:n\in\n)$ has the properties claimed above. 
\end{proof} 

For the proof of James' theorem we will only require the following special case of this lemma.

\begin{lemma}\label{increasing2}
	Let $V$ be a vector space (over $\mathbb{R}$) and let $\varphi:V\rightarrow\mathbb{R}$ be a sub-linear function. 
If $(A_n:n\in\n)$ is a decreasing sequence of nonempty, convex subsets of $V$, $(\beta_n:n \in \n)$ is any sequence of strictly positive numbers, $r >0$ and 
$$r < \inf_{a\in A_1}\varphi(a),$$ then there exists a sequence $(a_n:n \in \n)$ in $V$ such that, for all $n\in\n$:
\begin{enumerate}
\item[{\rm (i)}] $a_n\in A_n$ and 
\item[{\rm (ii)}] $\displaystyle \varphi($\mbox{$\sum_{i=1}^n\beta_ia_i$}$)+\beta_{n+1}r<\varphi($\mbox{$\sum_{i=1}^{n+1}\beta_ia_i)$}.
\end{enumerate}
\end{lemma}

In order to formulate our first version of James' theorem on weak compactness we need to introduce the following notions.  

\medskip

Let $K$ be a weak$^*$ compact convex subset of the dual of a Banach space $(X,\norm)$. A subset $B$ of $K$ is
called a {\it boundary} of $K$ if for every $\widehat{x} \in \widehat{X}$ there exists a $b^* \in B$ such that $\widehat{x}(b^*) = \sup\{\widehat{x}(y^*):y^* \in K\}$.
We shall say $B$, {\it $(I)$-generates $K$}, if for every countable cover $(C_n:n\in\N)$ of $B$ by weak$^*$ compact
convex subsets of $K$, the convex hull of $\bigcup_{n \in \N}C_n$ is norm dense in $K$. The following proof is found in \cite{MoorsWhite}.

\begin{theorem}[\!\!\cite{FonfLindenstraussPhelps, FonfLindenstrauss}] \label{Theorem 1} Let $K$ be a weak$^*$ compact convex subset of the dual of a Banach space $(X,\norm)$ and let $B$
be a boundary of $K$. Then $B$, $(I)$-generates $K$.
\end{theorem}

\begin{proof}  After possibly translating $K$, we may assume that $0 \in B$.
Let $\{C_n:n\in\N\}$ be weak* compact, convex subsets of $K$ such that $B\subseteq\bigcup_{n\in\N}C_n$ and suppose, for a contradiction, 
that co$[\bigcup_{n\in\N}C_n]$ is not norm dense in $K$. 
Then there must exist an $0<\ve$ and $y^*\in K$ such that 
$$y^*\in K\backslash(\mbox{co}[\mbox{$\bigcup_{n\in\N}C_n]$}+\ve\ball)$$
Since, for all $n\in\N$, co$[\bigcup_{j=1}^n C_j]$ is weak* compact and convex, there exist $(\widehat{x}_n: n \in\N)$ in $\widehat{X}$ such that  for every $n \in \N$, $\|\widehat{x}_n\|=1$ and 
$$\max\lbrace\widehat{x}_n(x^*):x^*\in \text{co}\mbox{$[\bigcup_{j=1}^n C_j]$}\rbrace +\ve =\max\lbrace\widehat{x}_n(x^*):x^*\in \text{co}\mbox{$[\bigcup_{j=1}^n C_j]$}+\ve\ball\rbrace <  
\widehat{x}_n(y^*). \mbox{\quad \quad $(**)$}$$
Now,  $(\widehat{x}_n(y^*):n\in\N)$ is a bounded sequence of real numbers and thus has a convergent subsequence $(\widehat{x}_{n_k}(y^*): k\in\N)$. 
Let $\displaystyle s := \lim_{k\rightarrow\infty}\widehat{x}_{n_k}(y^*)$. Then, $\ve \leq s$ and, after relabelling the sequence $(\widehat{x}_n:n \in \N)$ if 
necessary, we may assume that $|\widehat{x}_n(y^*)-s|< \ve/3$ for all $n\in\N$. 
Note that this relabelling does not disturb the inequality in ($**$). 

\medskip

We define $A_n:=\text{co}\lbrace\widehat{x}_k:n\leq k\rbrace$ for all $n\in\N$ and note that: (i) $(A_n:n \in \N)$ is a decreasing sequence of nonempty convex subsets of $\widehat{X}$ and 
(ii) if $N <n$ and  $b^* \in C_N$ then
$$g(b^*) < [g(y^*) - \ve] \mbox{\quad \quad for all $g \in A_n$} \hspace*{3.0cm} (*\!*\!*)$$
since, $\{\widehat{x}_k: n \leq k \} \subseteq \{\widehat{x} \in \widehat{X}: \widehat{x}(b^*-y^*) < -\ve\}$;  which is convex.
Next, we define $p:\widehat{X}\rightarrow\R$ by, 
$$p(\widehat{x}):=\sup_{x^*\in K}\widehat{x}(x^*) \mbox{ \quad  for all }  \widehat{x} \in \widehat{X}.$$ 
Then $p$ defines a sublinear functional on $\widehat{X}$. Moreover, for all $g\in A_1$, we have $(s - \ve/3) < g(y^*) \leq p(g)$ since $\{\widehat{x}_n: n \in \N\} \subseteq 
\{\widehat{x} \in \widehat{X}: (s - \ve/3) < \widehat{x}(y^*)\}$; which is convex and $y^* \in K$.

\medskip

Let $(\beta_n:n \in \N)$ be any sequence of positive numbers such that $\displaystyle \lim_{n\rightarrow\infty}\left(\mbox{$\sum_{i=n+1}^\infty$} \beta_i\right)/\beta_n=0$. Now,
$$(s-\ve/2) < (s - \ve/3) \leq \inf_{g \in A_1} p(g).$$ 
Therefore, by Lemma \ref{increasing2}, there exists a sequence $(g_n:n\in\N)$ in $\widehat{X}$ such that $g_n\in A_n$ and 
$$p(\mbox{$\sum_{i=1}^n$}\beta_ig_i)+\beta_{n+1}(s-\ve/2)<p(\mbox{$\sum_{i=1}^{n+1}$}\beta_ig_i) \mbox{ \quad for all $n\in\N$.}$$
Since $\|g_n\| \leq 1$ for all $n \in \N$, we have that $\sum_{i=1}^{\infty}\|\beta_ig_i\|\leq \sum_{i=1}^{\infty}\beta_i<\infty$. As $X$ is a Banach space, this implies 
that $g:=\sum_{i=1}^{\infty}\beta_ig_i\in\widehat{X}$.
Because $p$ is continuous, this implies that $(p(\sum_{i=1}^n\beta_ig_i):n\in\n)$ is a convergent - and hence bounded - sequence in $\r$. Moreover, Lemma \ref{increasing2} gives that $(p(\sum_{i=1}^n\beta_ig_i):n\in\n)$ is an increasing sequence.
Therefore, by the Convergence Theorem, $(p(\sum_{i=1}^n\beta_ig_i):n\in\n)$ converges to its supremum. That is,
	$$\sup_{n\in\n}p(\mbox{$\sum_{i=1}^n$}\beta_ig_i)=\lim_{n\rightarrow\infty}p(\mbox{$\sum_{i=1}^n$}\beta_ig_i)=p(\lim_{n\rightarrow\infty}\mbox{$\sum_{i=1}^n$}\beta_ig_i)=p(g).$$
Since $g\in\widehat{X}$, and since $B$ is a boundary for $K$, there must exist a $b^*\in B$ such that 
$$g(b^*)=\sup\{g(x^*):x^*\in K\}=p(g).$$ Then,
\begin{align*}
	\beta_n(s-\varepsilon/2) &<  p\!\left(\mbox{$\sum_{i=1}^n\beta_ig_i$}\right)-p\!\left(\mbox{$\sum_{i=1}^{n-1}\beta_ig_i$}\right) \\
	&\leq  p(g)-p\!\left(\mbox{$\sum_{i=1}^{n-1}\beta_ig_i$}\right)  \\
	 &\leq g(b^*)-\mbox{$\sum_{i=1}^{n-1}$}\beta_ig_i(b^*) =\mbox{$\sum_{i=n}^{\infty}$}\beta_ig_i(b^*).
\end{align*}
Since $B\subseteq \bigcup_{n\in\N}C_n$, $b^*\in C_N$ for some $N\in\N$. Thus, if $N<n$, then 
$$(s-\varepsilon/2)< \frac{1}{\beta_n}\!\Big(\mbox{$\sum_{i=n+1}^{\infty}$}\beta_ig_i(b^*)\Big) + g_n(b^*) < \frac{1}{\beta_n}\!\Big(\mbox{$\sum_{i=n+1}^{\infty}$}\beta_ig_i(b^*)\Big) + [g_n(y^*)-\varepsilon] \mbox{\quad by $(*\!*\!*),$}$$
since $g_n \in A_n$.
By taking the limit as $n$ tends to infinity we get that $(s - \varepsilon/2) \leq (s-\varepsilon)$; which is impossible. Therefore, $B$, ($I$)-generates $K$.
\end{proof}

\begin{remark} If $\displaystyle \beta_n := \frac{1}{n !}$ for all $n \in \N$ or, $\displaystyle \beta_n := \frac{1}{2^{n^2}}$ for all $n \in \N$, then $\displaystyle \lim_{n\rightarrow\infty}\frac{\mbox{$\sum_{i=n+1}^\infty$} \beta_i}{\beta_n}=0$.
\end{remark}

\begin{theorem}[James' Theorem: version 1, \cite{James57}]\label{James' Theorem: version 1} Let $C$ be a closed and bounded convex subset of a Banach space $(X,\norm)$. If $C$ is separable and
every continuous linear functional on $X$ attains its supremum over $C$, then $C$ is weakly compact.
\end{theorem}

\begin{proof} Let $K := \overline{\widehat{C}}^{w^*}$.
To show that $C$ is weakly compact it is sufficient to show $K \subseteq \widehat{X}$, (see Remark \ref{reflex-remark}).  In fact, since $X$ is a Banach space and $x \mapsto \widehat{x}$ is a linear isometry, we have that $\widehat{X}$
is a Banach subspace of $X^{**}$ and so a closed subspace of $X^{**}$.  Therefore, it is sufficient to show 
that for every $0<\varepsilon$, $K \subseteq \widehat{X}+2\varepsilon B_{X^{**}}$. To this end, fix $0<\varepsilon$
and let $\{x_n:n \in \N\}$ be a dense subset of  $C$.
For each $n \in \N$, let $K^\ve_n := K \cap [\widehat{x_n}+\varepsilon B_{X^{**}}]$. Then $(K^\ve_n:n\in\N)$ is a cover of $\widehat{C}$
by weak$^*$ closed convex subsets of $K$. Since $\widehat{C}$ is a boundary of $K$, we have that
$K \subseteq \overline{\mbox{co}}\bigcup_{n \in \N} K^\ve_n \subseteq \widehat{X}+ 2\varepsilon B_{X^{**}}$; which completes the proof.  \end{proof}

By working a bit harder, we could extend this approach to proving James' theorem, via ($I$)-generation, to spaces whose dual ball is weak$^*$ sequentially compact.
Indeed, this is done in the paper \cite{WarrenSimple}.  
However, in this paper we will take another tack.
 We will prove James' theorem, in the case when the dual ball is weak$^*$ sequentially compact, in a way
that naturally extends to the general case, albeit requiring several extra technical results regarding the extraction of subsequences with ``small'' sets of cluster points.

\medskip One of the strengths of Theorem \ref{James' Theorem: version 1} is that it essentially only relies upon a separation argument (Theorem \ref{separation-theorem}) and Lemma \ref{increasing2}.  
In this way we see that this proof is very elementary.


\subsection{James' Theorem on weak compactness: the weak$^*$ sequentially compact case}


We shall shall start this subsection with two simple preliminary results.

\begin{proposition}\label{FD is closed}
	Let $(X,\norm)$ be a normed linear space. Then every finite-dimensional subspace of $X^*$ is weak$^*$-closed.
\end{proposition}

\begin{proof}
	Suppose that $Y:=\text{span}\lbrace x^*_1,\dots , x^*_n\rbrace$ is a finite-dimensional subspace of $X^*$ and let $x_0^*\notin Y$.
	Then, by Lemma \ref{FiniteDimensional},
	we have that $\bigcap_{i=1}^n \ker(x^*_i)\not\seq\ker(x_0^*)$. 
	So, let $$x\in \mbox{$\bigcap_{i=1}^n$} \ker(x^*_i)\backslash \ker(x_0^*).$$
	Then, taking $- x$ if need be, we may assume that   $x_0^*(x)>0$, while $x^*_i(x)=0$ for all $1\leq i\leq n$.
	Observe that $Y=\text{span}\lbrace x^*_1,\dots , x^*_n\rbrace\seq \ker(\widehat{x})$, since $\ker(\widehat{x})$ is a subspace and $x^*_i\in\ker(\widehat{x})$ for all $1\leq i\leq n$. So $y^*(x)=0$ for all $y^*\in Y$.
	Thus, 
	$$\{x^*\in X^*:x^*(x)>0\}=\widehat{x}^{-1}(0,\infty)$$
	is a weak$^*$-open neighbourhood of $x_0^*$, which is disjoint from $Y$.
	Since $x_0^*$ was arbitrary, we have that $Y$ is weak$^*$-closed.
\end{proof}

\begin{lemma}\label{Separation}
	Let $(X,\norm)$ be a normed linear space, let $Y$ be a finite-dimensional subspace of $X^*$, and let $\epsilon>0$. If $x^*\in X^*$ and $\emph{dist}(x^*,Y)>\epsilon$, then there exists an $x\in S_X$ such that $x^*(x)>\epsilon$ and $y^*(x)=0$ for all $y^*\in Y$.
\end{lemma}
\begin{proof}
    Let $Y$ be a finite-dimensional subspace of $X^*$ such that $\text{dist}(x^*, Y)>\epsilon>0$. Then we have that $x^*\notin Y+\epsilon\ball$. Since $Y$ is weak$^*$-closed (Proposition \ref{FD is closed}) and convex, and $\ball$ is weak$^*$-compact (Theorem \ref{Banach-Alaoglu}) and convex, 
    we have that $Y+\epsilon\ball$ is also weak$^*$-closed and convex. Therefore, by Theorem \ref{separation-theorem},
    there exists an $x\in S_X$ such that 
    \begin{align*}
     x^*(x)&>\sup\{y^*(x):y^*\in Y+\epsilon\ball\} \\
     &=\sup\{y^*(x):y^*\in Y\}+\epsilon\geq \epsilon.
     \end{align*}
     Finally observe that for this $x$, we have that $\widehat{x}(Y)$ is bounded above, and since $Y$ is a subspace, the only way this is possible is if $\widehat{x}(y^*)=y^*(x)=0$ for all $y^*\in Y$.
\end{proof}

\begin{theorem}[James' Theorem: version 2]\label{JamesSequential}
		Let $C$ be a closed, bounded, convex subset of a Banach space $X$. If $(\ball, \mbox{weak}^*)$ is sequentially compact,
and every $x^*\in X^*$ attains its supremum over $C$, then $C$ is weakly compact.
\end{theorem}
\begin{proof}
To show that $C$ is weakly compact, it is sufficient to show that $K:=\overline{\widehat{C}}^{w*}\seq\widehat{X}$ (see, Remark \ref{reflex-remark}). Suppose, for a contradiction, that this is not the case.
	Then, there exists an $F\in K\backslash\widehat{X}$. Since $X$ is a Banach space, $\widehat{X}$ is a closed subspace of $X^{**}$, and so there must exist an $0<\epsilon<\text{dist}(F,\widehat{X})$. Let $(\beta_n:n\in\n)$ be a sequence of strictly positive numbers such that $\lim_{n\rightarrow\infty}\frac{1}{\beta_n}\sum_{i=n+1}^\infty\beta_i=0$. \\ \\
	\underline{\textbf{Part I:}} Let $f_0:=0$. We inductively create sequences $(f_n:n\in\n)$ in $S_{X^*}$ and $(\widehat{x}_n:n\in\n)$ in $\widehat{C}$, such that the statements 
	\begin{itemize}
		\item $(A_n): -$ $|(F-\widehat{x}_n)(f_j)|<\epsilon/2$ for all $0\leq j<n.$
		\item $(B_n):-$ $F(f_n)>\epsilon$ and $\widehat{x}_j(f_n)=0$ for all $1\leq j\leq n.$ 
	\end{itemize} 
	are true for all $n\in\n$. For the first step, choose any $\widehat{x}_1\in \widehat{C}$. Then it is clear that $|(F-\widehat{x}_1)(f_0)|=0<\epsilon/2$.
	Now note that $$\text{dist}(F,\text{span}\{\widehat{x}_1\})\geq \text{dist}(F,\widehat{X})>\epsilon.$$
	And so, by Lemma \ref{Separation}, there exists $f_1\in S_X$ such that $F(f_1)>\epsilon$ and $\widehat{x}_1(f_1)=0$. So the statements $(A_1)$ and $(B_1)$ hold. \\ \\
	Now fix $k\in\n$. Suppose that we have created $\{\widehat{x}_1,\dots,\widehat{x}_k\}$ and $\{f_1,\dots, f_k\}$ such that the statements $(A_k)$ and $(B_k)$ hold true.
	Then consider the set 
	$$W:=\mbox{$\bigcap_{j=0}^k$} \{G\in X^{**}:|(F-G)(f_j)|<\epsilon/2\}.$$
	Since $W$ is a weak$^*$-open neighbourhood of $F$, and $F\in \overline{\widehat{C}}^{w^*}$, we can choose $\widehat{x}_{k+1}\in \widehat{C}$ such that $\widehat{x}_{k+1}\in W$ 
	i.e., such that the statement $(A_{k+1})$ holds.
	Next, observe that 
	$$\text{dist}(F,\text{span}\{\widehat{x}_1,\dots,\widehat{x}_{k+1}\})\geq \text{dist}(F,\widehat{X})>\epsilon.$$
	So, by Lemma \ref{Separation}, there exists $f_{k+1}\in S_X$ such that $F(f_{k+1})>\epsilon$ and $\widehat{x}_j(f_{k+1})=0$ for all $1\leq j\leq k+1$.
	Therefore the statement $(B_{k+1})$ also holds. This completes the induction. \\ \\
	\underline{\textbf{Part II:}} Now let $(n_k:k\in\n)$ be a strictly increasing sequence of natural numbers. Then for all $k\in\n$, define $f'_k:=f_{n_k}$ and $x'_k:=x_{n_k}$. Also define $f'_0:=0$.
	Then the sequences $(\widehat{x}'_n:n\in\n)$ and $(f'_n:n\in\n)$ still satisfy $(A_n)$ and $(B_n)$ for all $n\in\n$. Therefore, passing to a subsequence does not disturb the statements $(A_n)$ and $(B_n)$.\\ \\
	Now, as $(\ball,\mbox{weak}^*)$ is sequentially compact, and $(f_n:n\in\n)$ is a sequence in $\ball$, we have that $(f_n:n\in\n)$ has a weak$^*$-convergent subsequence.
	So, by passing to subsequences and relabelling if necessary, we may assume that $(f_n:n\in\n)$ is weak$^*$-convergent to some $f_\infty\in\ball$.
	By the above, we know that the statements $(A_n)$ and $(B_n)$ remain true for all $n\in\n$.\\ \\
	\underline{\textbf{Part III:}} Let $k\in\n$. For any $n\geq k$, we have that that $\widehat{x}_k(f_n)=0$ by the statement $(B_n)$. Therefore, it follows that $\widehat{x}_k(f_\infty)=0$. Since $k$ was arbitrary, 
	this is true for all $k\in\n$. \\ \\
	On the other hand, let $k\in\n$ and let $n>k$. Then, by the statement $(A_n)$, we have that $|(F-\widehat{x}_n)(f_k)|<\epsilon/2$. Moreover, from $(B_k)$, we know that $F(f_k)>\epsilon$. 
	Combining these, we get that 
	$$\widehat{x}_n(f_k)=F(f_k)+(\widehat{x}_n-F)(f_k)>\epsilon/2$$
	for all $n>k$. Therefore $\widehat{x}_n(f_k-f_\infty)>\epsilon/2,$ for all $n>k$. \\ 
	
	\underline{\textbf{Part IV:}} For each $n\in\n$, define $C_n:=\text{co}\{f_k:k\geq n\}-f_\infty$ and note that $(C_n:n\in\n)$ is a decreasing sequence of nonempty, convex subsets of $X^*$.  
	Define $p:X^*\rightarrow\r$ to be $p(x^*)=\sup\{x^*(c):c\in C\}$ for all $x^*\in X^*$. Then $p$ is a sublinear function and $\inf_{f\in C_1}p(f)>\epsilon/4.$ \\ \\
	 To see this, let $f\in C_1$. Then $f=\sum_{i=1}^k\lambda_{i}f_{n_i}-f_\infty$ where $\lambda_i\geq 0$ for all $1\leq i\leq k$ and $\sum_{i=1}^k\lambda_{i}=1$. Let $m>\max\{n_1,\dots ,n_k\}$. Then
	 \begin{align*}
	 	p(f)\geq f(x_m)=\widehat{x}_m\!\left(\mbox{$\sum_{i=1}^k$}\lambda_{i}f_{n_i}-f_\infty\right)=\mbox{$\sum_{i=1}^k$}\lambda_{i}\widehat{x}_m(f_{n_i}-f_\infty)>\epsilon/2.
	 \end{align*}
	 Therefore, since $f\in C_1$ was arbitrary, we have that $\inf_{f\in C_1}p(f)>\epsilon/4$ as claimed. 
	 So, by Lemma \ref{increasing2},
	there exists a sequence $(g_n:n\in\n)$ such that for all $n\in\n$:
	\begin{enumerate}[label=(\roman*)]
	\item $g_n\in \text{co}\lbrace f_k:k\geq n\rbrace$ and
	\item $\displaystyle p($\mbox{$\sum_{i=1}^n\beta_i(g_i-f_\infty)$}$)+\beta_{n+1}\epsilon/4<p($\mbox{$\sum_{i=1}^{n+1}\beta_i(g_i-f_\infty))$.} \qquad \mbox{$(*)$}  \\ 
\end{enumerate}

	\underline{\textbf{Part V:}}  Now, since $(X^*,\mathrm{weak}^*)$ is a locally convex space $(g_n:n\in\N)$ also converges to $g_\infty := f_\infty$. Indeed, if $W$ is any convex weak$^*$ open neighbourhood
	of $f_\infty$ then there exists an $N \in \N$ such that $f_n \in W$ for all $n \geq N$.  Therefore, $\mathrm{co}\{f_k:k \geq N\} \subseteq W$. Since $g_k \in \mathrm{co}\{f_i:i \geq N\}$ for all $k \geq N$
	we have that $g_k \in W$ for all $k \geq N$.   This shows that $(g_k:k \in \N)$ converges to $g_\infty = f_\infty$.
	 Set $g:=\sum_{i=1}^\infty\beta_i(g_i-f_\infty)$. Since $\|g_i-f_\infty\|\leq 2$ for all $i\in\n$, we have that 
	 $$\sum_{i=1}^\infty\|\beta_i(g_i-f_\infty)\|= \sum_{i=1}^\infty\beta_i\|g_i-f_\infty\|\leq 2\sum_{i=1}^\infty\beta_i<\infty.$$ 
	Therefore, $g\in X^*$ since $X^*$ is a Banach space. As $p$ is continuous, it is clear that $(p(\sum_{i=1}^n\beta_i(g_i-f_\infty)):n\in\n)$ is a convergent - and in particular bounded - sequence in $\r$.    
	Moreover, the statement $(*)$ above gives that this is also an increasing sequence. Therefore, by the Monotone Convergence Theorem, $(p(\sum_{i=1}^n\beta_i(g_i-f_\infty)):n\in\n)$ converges to its supremum. That is,
	$$\sup_{n\in\n}p( \mbox{$\sum_{i=1}^n$}\beta_i(g_i-f_\infty))\!=\!\lim_{n\rightarrow\infty}\!p(\mbox{$\sum_{i=1}^n$}\beta_i(g_i-f_\infty))\!=\!p(\lim_{n\rightarrow\infty}\mbox{$\sum_{i=1}^n$}\beta_i(g_i-f_\infty))\!=\!p(g).$$  
	\underline{\textbf{Part VI:}} Since $g\in X^*$, there exists a  $c\in C$ such that $\widehat{c}(g)=g(c)=\sup\{g(x):x\in C\}=p(g).$ Then, for any $n>1$,
	\begin{align*}
	\beta_n \epsilon/4&<p(\mbox{$\sum_{i=1}^n$}\beta_i (g_i-f_\infty))-p(\mbox{$\sum_{i=1}^{n-1}$}\beta_i (g_i-f_\infty)) \qquad \mbox{(by $(*)$)}\\
	&\leq p(g)-p(\mbox{$\sum_{i=1}^{n-1}$}\beta_i (g_i-f_\infty)) \qquad\mbox{(since $\displaystyle p(g)=\sup\{p(\mbox{$\sum_{i=1}^n$}\beta_i(g_i-f_\infty)):n \in \N\}$} \\
	&= \widehat{c}(g)-p(\mbox{$\sum_{i=1}^{n-1}$}\beta_i (g_i-f_\infty))\\
	&\leq \widehat{c}(g)-\widehat{c}(\mbox{$\sum_{i=1}^{n-1}$}\beta_i (g_i-f_\infty))  
	=\widehat{c}(\mbox{$\sum_{i=n}^{\infty}$}\beta_i (g_i-f_\infty)) 
	=\beta_n\widehat{c}(g_n-f_\infty)+\mbox{$\sum_{i=n+1}^{\infty}$}\beta_i \widehat{c}(g_i-f_\infty).
\end{align*}	
Rearranging gives that 
$$\epsilon/4<\widehat{c}(g_n-f_\infty)+\frac{1}{\beta_n}\mbox{$\sum_{i=n+1}^{\infty}$}\beta_i \widehat{c}(g_i-f_\infty)\leq\widehat{c}(g_n-f_\infty)+\frac{2\|\widehat{c}\|}{\beta_n}\mbox{$\sum_{i=n+1}^{\infty}$}\beta_i.$$
Taking $n\rightarrow\infty$ we get that 
$$\epsilon/4\leq\lim_{n\rightarrow\infty}g_n(c)-f_\infty(c)+2\|\widehat{c}\|\left(\lim_{n\rightarrow\infty}\frac{1}{\beta_n}\mbox{$\sum_{i=n+1}^{\infty}$}\beta_i\right)=\lim_{n\rightarrow\infty}g_n(c)-f_\infty(c),$$
which contradicts the fact that $\displaystyle f_\infty(c)=\lim_{n\rightarrow\infty}g_n(c)$. Therefore, $K\seq\widehat{X}$ and so $C$ is weakly compact. 
\end{proof}

The power of this result stems from the fact that the class of all Banach spaces whose dual ball is weak$^*$ sequentially compact is very large.
Indeed, in addition to all the separable Banach spaces (whose dual ball is weak$^*$ metrisable), it contains all Asplund spaces, \cite{Larman} 
(i.e., spaces in which every separable subspace has a separable dual space) and all spaces that admit an equivalent smooth norm, \cite{Hagler} (which includes all WCG spaces, \cite{renorming}).  
In fact, it contains all Gateaux differentiability spaces, \cite{Larman}.  


\subsection{James' Theorem on weak compactness: the general case}


The short-coming of the previous subsection is that the dual ball of a Banach space need not be weak$^*$ sequentially compact.  For example, the dual ball  of $(C(\beta \N),\|\cdot\|_\infty)$ is not weak$^*$ sequentially compact,
as it contains a copy of $\beta \N$ - the Stone-Cech compactification of the natural numbers, endowed with the discrete topology and this space is known to have no nontrivial (i.e., not eventually constant) convergent sequences, \cite{Engelking}.

\medskip

So the method of passing to a subsequence which is weak$^*$ convergent must be abandoned.  However we can, by passing to a suitable subsequence, insist that 
$K := \bigcap_{n \in \N} \overline{ \{f_k : k \geq n\}}^{w^*}$ is ``small'' in the sense that for countably many weak$^*$ lower semicontinuous real-valued functions $(p_n:n \in \N)$, the sets $p_n(K)$ are singletons.  
In this way, the set $K$ of all weak$^*$ cluster points of the sequence $(f_n:n \in \N)$ ``acts'' like a singleton set in  \underline{\textbf{Part V}} and \underline{\textbf{Part VI}} of the proof of Theorem \ref{JamesSequential}.

\medskip

So next we will show how to extract ``nice'' subsequences from a given sequence.  The approach we adopt is very general and will provide much more than needed, but these
technical results may possibly be of some independent interest.

\medskip

We shall start with the precise definition of lower semicontinuity. Let $(X,\tau)$ be a topological space. We say a function $f:X \rightarrow \mathbb{R} \cup \{\infty\}$ is {\it lower semicontinuous} 
if for every $\alpha \in \mathbb{R}$,  $\{x\in X: f(x) \leq \alpha \}$ is a closed set.

\medskip

Since we will be working extensively with subsequences we will introduce some concise notation for a subsequence of a given sequence.
Let $\widetilde{x}:\n\rightarrow X$ be the sequence $(x_n:n\in\n)$ and let $J$ be an infinite subset of $\n$, 
i.e., $J=\lbrace n_k:k\in\n\rbrace$ with $n_k<n_{k+1}$ for all $k\in\n$. Then the subsequence $(x_{n_k}:k\in\n)$ will be denoted by $\widetilde{x}|_J$.
We will also be working with the set of all cluster points of a given sequence and so it is worth our while to introduce some notation for the set of all cluster points (and another related set as well).
Let $(X,\tau)$ be a linear topological space and let $\widetilde{x}:\n\rightarrow X$ be the sequence $(x_n:n\in\n)$. We define 
	$$cl_{\tau}(\widetilde{x}):=\mbox{$\bigcap_{n=1}^\infty$}\overline{\{x_k:k\geq n\}}^\tau.$$
	That is, $cl_{\tau}(\widetilde{x})$ is the set of all $\tau$-cluster points of $\widetilde{x}$. Further, define $K_\tau(\widetilde{x}):=\overline{\mathrm{co}}^\tau(cl(\widetilde{x}))$. When there is no ambiguity concerning the topology, we will simply write $cl(\widetilde{x})$ and $K(\widetilde{x})$.

\begin{lemma}\label{SubsequenceBounded}
If $\varphi:A\rightarrow\r$ is a convex lower-semicontinuous function defined on a nonempty closed and convex subset $A$ of a Hausdorff locally convex space $(X,+,\cdot,\tau)$, then for every sequence 
$\widetilde{x}:=(x_n:n\in\n)$ in $A$, there is a subsequence $\widetilde{x}|_{J}$ of $ \widetilde{x}$ such that $\varphi(K(\widetilde{x}|_{J}))$ is either empty, or bounded.
\end{lemma}

\begin{proof}
	Let $\widetilde{x}:=(x_n:n\in\n)$ be a sequence in $A$. Suppose that $\widetilde{x}$ has no subsequence, $\widetilde{x}|_{J}$, such that $\varphi(K(\widetilde{x}|_{J}))$ is empty.
	First, we construct an infinite subset $J'$ of $\N$ such that $\varphi(K(\widetilde{x}|_{J'}))$ is bounded below. \\ \\
	Let $x\in cl(\widetilde{x})$. Then $x\notin \varphi^{-1}(-\infty,\varphi(x)-1]$, which is closed and convex. 
	Therefore, there exists a closed and convex neighbourhood, $N$ of $x$ such that $N\cap\varphi^{-1}(-\infty,\varphi(x)-1]=\varnothing$ i.e., $\varphi(N)\seq (\varphi(x)-1,\infty)$. 
	Since $x$ is a cluster point of $\widetilde{x}$, we may choose an infinite set $J'\seq\n$ such that $x_j\in N$ for all $j\in J'$.
	Then, because $N$ is closed and convex, $K(\widetilde{x}|_{J'})\seq N$ and so $\varphi(K(\widetilde{x}|_{J'}))\seq\varphi(N)\seq (\varphi(x)-1,\infty)$. Hence $\varphi(K(\widetilde{x}|_{J'}))$ is bounded below. 
	
	\medskip
	
	We now claim that $J'$ possesses an infinite subset $J$ such that $\varphi(K(\widetilde{x}|_{J}))$ is bounded above.  
	Indeed, suppose in order to obtain a contradiction, that this is not the case. Then we inductively proceed as follows. First, there must be $x\in cl(\widetilde{x}|_{J'})$ with $\varphi(x)>1$, otherwise 
	$\varphi(K(\widetilde{x}|_{J'}))\seq (- \infty,1]$ and we would be done.
	So, we may choose a closed, convex neighbourhood, $N$ of $x$ such that $N\cap\varphi^{-1}(-\infty,1]=\varnothing$. 
    Then, since $x$ is a cluster point of $\widetilde{x}|_{J'}$, we can choose an infinite subset $J_1\seq J'$ such that $x_j \in N$ for all $j\in J_1$.
	Because $N$ is closed and convex, we have that  $K(\widetilde{x}|_{J_1})\seq N$ and so $K(\widetilde{x}|_{J_1}) \cap \varphi^{-1}(-\infty,1] = \varnothing$.
	
	\medskip
	
	In general, suppose that we have chosen infinite subsets $J_n\seq\dots\seq J_1\seq J'$ such that for all $1\leq i\leq n$: $K(\widetilde{x}|_{J_i}) \cap \varphi^{-1}(-\infty,i] = \varnothing$. 
	
	\medskip
	
	For the $(n+1)^{\textrm{th}}$ step, we suppose that $cl(\widetilde{x}|_{J_n}))\not\seq\varphi^{-1}(-\infty,n+1]$, otherwise $\varphi(K(\widetilde{x}|_{J_n}))\seq(-\infty,n+1]$ is bounded above and we are done. 
	Therefore, we can choose $x\in cl(\widetilde{x}|_{J_n})$ such that $\varphi(x)>n+1$, and a closed, convex neighbourhood, $N$ of $x$ such that $N \cap \varphi^{-1}(-\infty,n+1]=\varnothing$.
	Then, since $x$ is a cluster point of $\widetilde{x}|_{J_n}$, we can choose an infinite subset $J_{n+1}\seq J_n$ such that $x_j \in N$ for all $j \in J_{n+1}$. 
	Because $N$ is closed and convex, we have that  $K(\widetilde{x}|_{J_{n+1}})\seq N$ and so $K(\widetilde{x}|_{J_{n+1}}) \cap \varphi^{-1}(-\infty,n+1] = \varnothing$.
	This completes the induction. 
	
	\medskip
	
	Lastly, we apply the so-called diagonalisation argument. Define $J^{''}:=\{n_k:k\in\n\}\seq\n$ such that $n_k<n_{k+1}$  and $n_k \in J_k$ for all $k \in \N$.
	
	\medskip
	
    Consider the subsequence of $\widetilde{x}$ given by $\widetilde{x}|_{J^{''}} = (x_{n_k}:k\in\n)$. 
    Then, since $J_{n+1}\seq J_n$ for all $n\in\n$, we have that $n_k\in J_m$ for all $k\geq m$. 
	Let $m\in\n$. Then $$K(\widetilde{x}|_{J^{''}})=K(\{x_{n_k}:k\geq m\})\seq K(\widetilde{x}|_{J_m})\seq X \setminus \varphi^{-1}(-\infty, m],$$ 
	and so $K(\widetilde{x}|_{J^{''}})\cap \varphi^{-1}(-\infty,m]=\varnothing$. 
	Since $m$ was arbitrary, this holds for all $m\in\n$ and so we have that $\varphi(K(\widetilde{x}|_{J^{''}}))=\varnothing$, which contradicts our original assumption. Thus, there exists a subsequence $\widetilde{x}|_J$ of $\widetilde{x}$ such that $\varphi(K(\widetilde{x}|_J))$ is bounded. 
	\end{proof}
	
We can further refine Lemma \ref{SubsequenceBounded} as follows.

\begin{lemma}\label{Singleton}
If $\varphi:A\rightarrow\r$ is a convex lower-semicontinuous function defined on a nonempty closed and convex subset $A$ of a Hausdorff locally convex space $(X,+,\cdot,\tau)$, then
 for every sequence $\widetilde{x}:=(x_n:n\in\n)$ in $A$, there is a subsequence $\widetilde{x}|_J$ of $ \widetilde{x}$ such that $\varphi(K(\widetilde{x}|_J))$ is at most a singleton. 
\end{lemma}

\begin{proof}
	Suppose that $\widetilde{x}$ has no subsequence, $\widetilde{x}|_J$, such that $\varphi(K(\widetilde{x}|_J))$ is empty. Then, by Lemma \ref{SubsequenceBounded}, and by passing to a subsequence if necessary, 
	we may assume that $\varphi(K(\widetilde{x}))$ is bounded. 
	Let $\alpha_1,\beta_1\in\r$ denote $\inf\varphi(K(\widetilde{x}))$ and $\sup\varphi(K(\widetilde{x}))$ respectively and let $J_0:=\n$. Of course if $\alpha_1=\beta_1$, then $\varphi(K(\widetilde{x}))$ is a singleton and we are done. If not, we inductively construct 
	a decreasing sequence of infinite subsets $(J_n:n \in \N)$ of $\N$ such that $\text{diam}(\varphi(K(\widetilde{x}|_{J_n}))\leq (\beta_1-\alpha_1)/2^{n}$ for all $n \in \N$.
	
	\medskip
	
	We begin as follows. Set $\delta_1:=(\alpha_1+\beta_1)/2$. Since $\varphi$ is convex and lower-semicontinuous, we have that $\varphi^{-1}(-\infty,\delta_1]$ is a closed, convex set. 
	Then, we can pick $x\in cl(\widetilde{x})$ such that $\delta_1<\varphi(x)\leq\beta_1$. Indeed, if not, then $\varphi(cl(\widetilde{x}))\seq (-\infty, \delta_1]$ and so $\varphi(K(\widetilde{x}))\seq (-\infty, \delta_1]$ also. 
	However, this contradicts the fact that $\beta_1=\sup  \varphi(K(\widetilde{x}))$. \\ \\ 
	Therefore $x\notin\varphi^{-1}(-\infty,\delta_1]$, and so there exists a closed, convex neighbourhood, $N$ of $x$ such that $N\cap\varphi^{-1}(-\infty,\delta_1]=\varnothing$.
	As $x\in cl(\widetilde{x})$, there is an infinite set $J_1\seq \n$ such that $x_{j}\in N$ for all $j\in J_1$.
	 In particular, $K(\widetilde{x}|_{J_1})\seq N$, since $N$ is closed and convex, and so $\inf \varphi(K(\widetilde{x}|_{J_1}))\geq\delta_1$. Also, because $\widetilde{x}|_{J_1}$ is a subsequence of 
	 $\widetilde{x}$, we have that $\sup \varphi(K(\widetilde{x}|_{J_1}))\leq\beta_1$. Therefore, $\text{diam}(\varphi(K(\widetilde{x}|_{J_1}))\leq\beta_1-\delta_1= (\beta_1-\alpha_1)/2$. 
	 
	 \medskip

	Suppose now that we have created the infinite subsets $J_n \subseteq J_{n-1} \subseteq \cdots \subseteq J_0$ such that 
	$$\text{diam}(\varphi(K(\widetilde{x}|_{J_i}))\leq (\beta_1-\alpha_1)/2^{i}  \mbox{\quad for all $1 \leq i \leq n$.}$$
	
	Set $\alpha_n:=\inf\varphi(K(\widetilde{x}|_{J_n}))$, $\beta_n :=\sup\varphi(K(\widetilde{x}|_{J_n}))$ and $\delta_n:=(\alpha_n+\beta_n)/2$. 
	Then, 
	$$\text{diam}(\varphi(K(\widetilde{x}|_{J_n}))=\beta_n-\alpha_n\leq(\beta_1-\alpha_1)/2^n,$$ by construction. If  $\alpha_n = \beta_n$ then let $J_{n+1} := J_n$ and we are done. Otherwise,
	we can choose (as above) $x\in cl(\widetilde{x}|_{J_n})$ such that $x\notin\varphi^{-1}(-\infty,\delta_n]$, because if not, $\varphi(K(\widetilde{x}|_{J_n}))\seq (-\infty,\delta_n]$, 
	which contradicts the fact that $\beta_n=\sup\varphi(K(\widetilde{x}|_{J_n}))$.
	Therefore, there exists a closed, convex neighbourhood, $N$ of $x$ such that $N\cap\varphi^{-1}(-\infty,\delta_{n}]=\varnothing.$ 
	Since $x\in cl(\widetilde{x}|_{J_n})$, there is an infinite set $J_{n+1}\seq J_n$ such that  $x_j\in N$ for all $j \in J_{n+1}$.
In particular, since $N$ is closed and convex, $K(\widetilde{x}|_{J_{n+1}})\seq N$ and so $\inf \varphi(K(\widetilde{x}|_{J_{n+1}}))\geq\delta_n$. Therefore,
 $$\text{diam}(\varphi(K(\widetilde{x}|_{J_{n+1}}))\leq\beta_n-\delta_n=(\beta_n-\alpha_n)/2\leq (\beta_1-\alpha_1)/2^{n+1}.$$ 
 Thus, by induction, we have created a decreasing sequence of infinite subsets $(J_n:n \in \N)$ of $\N$ such that 
 $$\text{diam}(\varphi(K(\widetilde{x}|_{J_n}))\leq (\beta_1-\alpha_1)/2^{n} \mbox{\quad for all $n \in \N$.}$$
 
 \medskip
 
 Lastly, define $J:=\{n_k:k\in\n\}$ such that $n_k<n_{k+1}$ and $n_k\in J_k$ for all $k \in \N$.
	Consider the subsequence of $\widetilde{x}$ given by $\widetilde{x}|_J = (x_{n_k}:k\in\n)$. 
	Then, since $J_{n+1}\seq J_n$ for all $n\in\n$, we have that $n_k\in J_m$ for all $k\geq m$. 
	Let $m\in\n$. Then, 
	$$K(\widetilde{x}|_J)=K(\{x_{n_k}:k\geq m\})\seq K(\widetilde{x}|_{J_m}),$$ which gives that $\text{diam}(\varphi(K(\widetilde{x}|_J)))\leq\text{diam}(\varphi(K(\widetilde{x}|_{J_m})))\leq  (\beta_1-\alpha_1)/2^{m} $. 
	Since $m\in\n$ was arbitrary, we conclude that $\varphi(K(\widetilde{x}|_J))$ is a singleton as required.
	\end{proof}


For our next result we need to recall the definition of the topology of pointwise convergence. If $X$ is a nonempty set and $A$ is a nonempty subset of $X$ then we may put a
topology on the vector space $\R^X$ of all real-valued functions defined on $X$ endowed with pointwise addition and pointwise scalar multiplication. We will call the weak topology
on $\R^X$ generated by $\{\delta_a:a \in A\}$ the {\it topology of pointwise convergence on $A$}, where for each $a \in A$, $\delta_a:\R^X \to \R$ is defined by, $\delta_a(f) := f(a)$.
We shall denote the topology of pointwise convergence on $A$ by $\tau_p(A)$.

\begin{corollary}\label{Diagonal}
 For each $n \in \N$, let $\varphi_n:A\rightarrow\r$ be a convex lower-semicontinuous function defined on a nonempty closed and convex subset $A$ of a Hausdorff locally convex space $(X,+,\cdot,\tau)$, then
 for every sequence $\widetilde{x}:=(x_n:n\in\n)$ in $A$, there exists a subsequence, $\widetilde{x}|_J$, of $\widetilde{x}$ such that $\varphi(K(\widetilde{x}|_J))$ is at most a singleton for all 
	$\varphi \in \overline{\{\varphi_n:n\in\n\}}^{\tau_p(A)}$.
\end{corollary}

\begin{proof}
	 Let $J_0:=\n$. We inductively construct a decreasing sequence of infinite subsets $(J_n:n \in \N)$ of $\N$ such that
	  $\varphi_n(K(\widetilde{x}|_{J_n}))$ is at most a singleton for each $n \in \N$.
	  
	  \medskip
   
    We begin as follows. Since $\varphi_1$ is convex and lower-semicontinuous, there exists,  by Lemma \ref{Singleton}, an infinite subset $J_1$ of $\N$ such that $\varphi_1(K(\widetilde{x}|_{J_1}))$ is at most a singleton. 
    
    \medskip
    
    Now, suppose that we have created infinite subsets $J_n \subseteq J_{n-1} \subseteq \cdots \subseteq J_1 \subseteq \N$ such that $\varphi_i(K(\widetilde{x}|_{J_i}))$ is at most a singleton for all $1 \leq i \leq n$.
    
    \medskip
   
    Then, for the $(n+1)^{\rm{th}}$ step choose, using Lemma \ref{Singleton},  an infinite subset $J_{n+1}$ of $J_n$ such that $\varphi_{n+1}(K(\widetilde{x}|_{J_{n+1}}))$ is at most a singleton.
    
    \medskip
    
    Now, define $J:=\{n_k:k\in\n\}\seq\n$ such that $n_k<n_{k+1}$ and $n_k\in J_k$ for all $k \in \N$.
    Consider the subsequence of $\widetilde{x}$ given by $\widetilde{x}|_J = (x_{n_k}:k\in\n)$. Then, since $J_{n+1}\seq J_n$ for all $n\in\n$, we have that $n_k\in J_m$ for all $k\geq m$. 
	Let $m\in\n$. Then, 
	$$K(\widetilde{x}|_J)=K(\{x_{n_k}:k\geq m\})\seq K(\widetilde{x}|_{J_m}),$$ 
	and so $\big|\varphi_m(K(\widetilde{x}|_J))\big|\leq \big|\varphi_m(K(\widetilde{x}|_{J_m}))\big|\leq 1$. 
	Since $m$ was arbitrary, this gives that $\varphi_m(K(\widetilde{x}|_J))$ is at most a singleton for all $m\in\n$. 
	
	\medskip
	
	Now, let $\varphi\in\overline{\{\varphi_n:n\in\n\}}^{\tau_p(A)}$ and let $x,y\in K(\widetilde{x}|_J)$. Suppose, for a contradiction, that $\varphi(x)>\varphi(y)$. 
	Then 
	$$N:=\left\{F\in \r^X:F(x)>(1/2)[\varphi(x)+\varphi(y)]\right\}\cap\left\{F\in \r^X:F(y)<(1/2)[\varphi(x)+\varphi(y)]\right\}$$
	is a $\tau_p(A)$-neighbourhood of $\varphi$. Since $\varphi\in\overline{\{\varphi_n:n\in\n\}}^{\tau_p(A)}$ there must exist $k\in\n$ such that $\varphi_k\in N$. However, this is impossible as $\varphi_k(x)=\varphi_k(y)$ for all $k\in\n$, 
	and so $\varphi(K(\widetilde{x}|_J))$ is at most a singleton.
\end{proof}

By applying Corollary \ref{Diagonal} we obtain the following technical result that is needed (i.e., provides the required subsequence) in the proof of the general version of James' weak compactness theorem.

\begin{corollary}\label{SubsequenceUsefulJames}
		Let $\varphi:X \to \R$ be a $\tau$-continuous convex function defined on a locally convex space $(X, +,\cdot, \tau)$.  If $\tau'$ is a Hausdorff locally convex topology on $X$ such that (i) $\tau' \subseteq \tau$
	and (ii) $\varphi$ is $\tau'$-lower semicontinuous then, for every sequence $\widetilde{x}:=(x_n:n\in\n)$ in $X$, there exists a subsequence, $\widetilde{x}|_J$, of $\widetilde{x}$ such that 
	$\varphi(y-aK_{\tau'}(\widetilde{x}|_J))$ is at most a singleton for all $y\in\emph{span}\lbrace x_n:n\in\n\rbrace$ and all $a\in\r$.
\end{corollary}

\begin{proof}
	Observe that $Y:=\text{span}\lbrace x_n:n\in\n\rbrace$ is separable, so let $\{y_n:n \in \n \}$ be a countable, dense subset of $Y$. Moreover, let $\{q_n : n \in \n\}$ be an enumeration of 
	$\mathbb{Q}\backslash\{0\}$. Now for all $m,n\in\n$, define $\varphi_{n}^m:X\rightarrow\r$ by 
	$$\varphi_n^m(x)=\varphi(y_n-q_mx) \quad \text{for all} \ x\in X.$$
	Since $x\mapsto (y_n-q_mx)$ is a continuous affine function and $\varphi$ is convex and $\tau'$-lower semicontinuous, we have that $\varphi_n^m$ is $\tau'$-lower-semicontinuous and convex for all $m,n\in\n$. 
	Then, by Corollary \ref{Diagonal}, there exists a subsequence, $\widetilde{x}|_{J}$, of $\widetilde{x}$ such that $\psi(K_{\tau'}(\widetilde{x}|_{J}))$ is at most a singleton for all $\psi$ in the $\tau_p(X)$-closure of 
	$\{\varphi_n^m:m,n\in\n\}$.
	
	\medskip
	
	Now observe that, for all $a\in\r$ and all $y\in Y$, the function $\varphi^a_y:X\rightarrow\r$ given by $\varphi^a_y(x):=\varphi(y-ax)$ is in the $\tau_p(X)$-closure of $\{\varphi_n^m:m,n \in \n \}$.
 Therefore, $\varphi_y^a(K_{\tau'}(\widetilde{x}|_J))=\varphi(y-aK_{\tau'}(\widetilde{x}|_J))$ is at most a singleton for all $y\in\text{span}\lbrace x_n:n\in\n\rbrace$ and all $a\in\r$, as required.
\end{proof}

The last result we need before we can prove the full version of James' theorem concerns the convergence of the subsequences that we constructed in \underline{\textbf{Part IV}} of the proof of Theorem \ref{JamesSequential}.

\begin{proposition}\label{ClusterConvex}
Let $(X,\tau)$ be a locally convex space and let $\widetilde{x}:=(x_n:n\in\n)$ be a sequence in a \\
$\tau$-compact convex subset $K$ of $X$. If $\widetilde{y}:=(y_n:n\in\n)$ is any sequence such that $y_k\in \emph{co}\{x_n:n\geq k\}$ for all $k \in \N$, then $cl(\widetilde{y})\seq K(\widetilde{x})$. 
\end{proposition}

\begin{proof}
	It is sufficient to show that for any open, convex neighbourhood, $W$ of $0$,  $cl(\widetilde{y})\seq K(\widetilde{x})+\overline{W}$. To this end, let $W$ be an open, convex neighbourhood of $0$.
	Then note that for $k$ sufficiently large, 
	$$\{x_n:n\geq k\}\seq cl(\widetilde{x})+W.$$
	Indeed if this is not the case, then we could construct a subsequence $(x_{n_k}:k \in \N)$ of $(x_n:n \in \N)$ such that $x_{n_k}\notin cl(\widetilde{x})+W$ for all $k \in \N$,
	However, since $X\setminus [cl(\widetilde{x})+W]$ is a closed set containing $\{x_{n_k}:k \in \N\}$ we have that $\overline{\{x_{n_k}:k \in \N\}} \cap [cl(\widetilde{x})+W] = \varnothing$,
	 but this is impossible since $\overline{\{x_{n_k}:k\in\n\}}\cap cl(\widetilde{x})\neq \varnothing$. Thus, we have a contradiction. Therefore, if $y\in cl(\widetilde{y})$, then for $k$ sufficiently large, 
	 we have that
	$$y\in \overline{\{y_n: n \geq k\}} \subseteq \overline{\mathrm{co}\{x_n:n \geq k\}}\seq K(\widetilde{x})+\overline{W} \mbox{\quad since,  $K(\widetilde{x})+\overline{W}$ is closed and convex.}$$
	Hence, $cl(\widetilde{y})\seq K(\widetilde{x})+\overline{W}$ as required.
\end{proof}

\begin{theorem}[James' Theorem: version 3, \cite{James}]\label{JamesFull}
	Let $C$ be a closed, bounded, convex subset of a Banach space $X$. If every $x^*\in X^*$ attains its supremum over $C$, then $C$ is weakly compact.
\end{theorem}	
	\begin{proof}
	To show that $C$ is weakly compact, it suffices to show that $K:=\overline{\widehat{C}}^{w*}\seq\widehat{X}$ (see Remark \ref{reflexive}). Suppose, for a contradiction, that this is not the case.
	Then there exists $F\in K\backslash\widehat{X}$. Since $\widehat{X}$ is a closed subspace of $X^{**}$, this means there must exist $0<\epsilon<\text{dist}(F,\widehat{X})$. Let $(\beta_n:n\in\n)$ be a sequence of strictly positive numbers such that $\lim_{n\rightarrow\infty}\frac{1}{\beta_n}\sum_{i=n+1}^\infty\beta_i=0$.\\ \\
	\underline{\textbf{Part I:}} We inductively create  the two sequences $(\widehat{x}_n:n\in\n)$ in $\widehat{C}$, and $(f_n:n\in\n)$ in $S_{X^*}$, which satisfy the statements $(A_n)$ and $(B_n)$, exactly as in  Part I of the proof of Theorem \ref{JamesSequential}. \\ \\
	\underline{\textbf{Part II:}} 
	Define $p:X^*\rightarrow\r$ to be $p(x^*)=\sup\{x^*(c):c\in C\}$ for all $x^*\in X^*$. Then $p$ is norm-continuous, weak$^*$-lower-semicontinuous and convex.
	Just as in Part II of the proof of Theorem \ref{JamesSequential}, passing to a subsequence does not disturb the statements $(A_n)$ and $(B_n)$.  \\ \\
	 So, by passing to a subsequence and relabelling if necessary, by Corollary \ref{SubsequenceUsefulJames}
we may assume that for all $f\in\text{span}\lbrace f_n:n\in\n\rbrace$ and all $a\in \r$, the set $p(f-aK_{w^*}(f_n:n\in\n))$ is at most a singleton.  
Since $(f_n:n\in\n)$ is a sequence in $\ball$ (which is weak$^*$-compact), it must a have a weak$^*$-cluster point, call it $f_\infty$. \\ \\ 
	\underline{\textbf{Part III:}} This step is exactly the same as Part III of the proof of Theorem \ref{JamesSequential} - we deduce that $\widehat{x}_n(f_k-f_\infty)>\epsilon/2$ for all $n>k$. \\ \\ 
	\underline{\textbf{Part IV:}} As in the proof of Theorem \ref{JamesSequential}, we use Lemma \ref{increasing2} to construct a sequence $(g_n:n\in\n)$ such that for all $n\in\n$:
	\begin{enumerate}[label=(\roman*)]
	\item $g_n\in \text{co}\lbrace f_k:k\geq n\rbrace$ and
	\item $\displaystyle p($\mbox{$\sum_{i=1}^n\beta_i(g_i-f_\infty)$}$)+\beta_{n+1}\epsilon/4<p($\mbox{$\sum_{i=1}^{n+1}\beta_i(g_i-f_\infty))$}.  
\end{enumerate} 
	\underline{\textbf{Part V:}} Since $(g_n:n\in\n)$ is a sequence in $\ball$ (which is weak$^*$-compact), it must a have a weak$^*$-cluster point, call it $g_\infty$.  Then, by Proposition \ref{ClusterConvex}, we have that $g_\infty\in K_{w^*}(f_n:n\in\n)$. 
	While it may no longer be the case that $f_\infty=g_\infty$ as in Theorem \ref{JamesSequential}, we do have that, for all $n\in\n$,
$$ p(\mbox{$\sum_{i=1}^n$}\beta_i(g_i-g_\infty))=p(\mbox{$\sum_{i=1}^{n}$}\beta_ig_i-\mbox{$\sum_{i=1}^{n}$}\beta_i\cdot g_\infty)
  =p(\mbox{$\sum_{i=1}^{n}$}\beta_ig_i-\mbox{$\sum_{i=1}^{n}$}\beta_i\cdot f_\infty) 
  =p(\mbox{$\sum_{i=1}^n$}\beta_i(g_i-f_\infty)) \qquad\mbox{$(**)$} $$
  since $g_\infty\in K_{w^*}(f_n:n\in\n)$ and for all $f\in\text{span}\lbrace f_n:n\in\n\rbrace$ and all $a\in \r$, the set $p(f-aK_{w^*}(f_n:n\in\n))$ is a singleton. As in Part V of the proof of Theorem \ref{JamesSequential}, 
  we set $g:=\sum_{i=1}^\infty\beta_i(g_i-g_\infty)$ and deduce that $g\in X^*.$ \\ \\
	\underline{\textbf{Part VI:}}
	This final step is almost the same as Part VI of the proof of Theorem \ref{JamesSequential}, with two small changes that we note here. We may replace $f_\infty$ with $g_\infty$ throughout the inequalities, 
	not because $f_\infty = g_\infty$ but because of  statement $(**)$ above.
	Lastly, the final contradiction is not because $\lim_{n\rightarrow\infty}g_n(c) = g_\infty(c)$ necessarily, but because  $\liminf_{n\rightarrow\infty}g_n(c) \leq g_\infty(c)$. This still gives a contradiction.
\end{proof}


\subsection{James' Theorem: applications}


\begin{theorem}[\!\!\cite{James63}] \label{James' Theorem Reflexivity} Let $(X,\norm)$ be a Banach space. Then $X$ is reflexive if, and only if, every continuous linear functional $x^*$ on $X$ attains its norm
(i.e., there exists an $x \in B_X$ such that $\|x^*\| = x^*(x)$).
\end{theorem}
\begin{proof} By Theorem \ref{reflexive}, $X$ is reflexive if, and only if, $B_X$ is weakly compact.  So the result now follows from Theorem \ref{JamesFull} once one remembers that every
continuous linear functional on $X$ is continuous with respect to the weak topology on $X$.
\end{proof}
Note: if $X$ is reflexive then one can use the Hahn-Banach Theorem to directly show that every continuous linear functional on $X$ attains its norm.  Indeed, suppose
that $x^*$ is a nonzero continuous linear functional on $X$.  Then by Corollary \ref{HB} there exists an $x^{**} \in S_{X^{**}}$ such that $x^{**}(x^*) = \|x^*\|$.  However,
since $X$ is reflexive, $x^{**} = \widehat{x}$ for some $x \in S_X$.  Hence, $\|x^*\| = x^{**}(x^*) = \widehat{x}(x^*) = x^*(x)$.  This shows that $x^*$ attains its norm. 

\medskip

We now recall a geometric concept in Banach space theory. We say that a Banach space, $(X,\norm)$, is \emph{uniformly convex} if, for any $\epsilon>0$, 
there exists $\delta_\epsilon>0$ with the following property: if $x,y\in B_X$ and $\|x+y\|>2-\delta_\epsilon$, then $\|x-y\|<\epsilon$.

\begin{theorem}[\!\!\cite{Pettis}]
	Let $(X,\norm)$ be a Banach space. If $(X,\norm)$ is uniformly convex, then $(X,\norm)$ is reflexive.
\end{theorem}

\begin{proof}
	Let $x^*\in S_{X^*}$ and define $(x_n:n\in\n)$ in $B_X$ so that $x^*(x_n)>1-\frac{1}{n}$. Let $\epsilon >0$ and choose $\delta_\epsilon>0$ such that if $x,y\in B_X$ and $\|x+y\|>2-\delta_\epsilon$, 
	then $\|x-y\|<\epsilon$ Then, for $n,m \in \N$ greater than $N_0 := 2/\delta_\epsilon$, we have that $2\geq \|x_n+x_m\|\geq x^*(x_n+x_m)>2-\delta_\epsilon$.
	By the uniform convexity of $X$, this gives that for $n,m>N_0$, we have $\|x_n-x_m\|\leq \epsilon.$ So, $(x_n:n\in\n)$ is a Cauchy sequence in $X$.  
	Therefore, $(x_n:n \in \n)$ is convergent to some $x\in B_X$. It is clear that for this $x$, $x^*(x)=1=\|x^*\|$.
	Since $x^*$ was arbitrary in $S_{X^*}$, every $x^*$ in $X^*$ attains its norm, and so, by Theorem \ref{James' Theorem Reflexivity}, $X$ is reflexive.
\end{proof}

Another interesting application of Theorem \ref{JamesFull} is the Krein-Smulian theorem.

\begin{corollary}[Krein-Smulian Theorem,\cite{KreinSmulian}]\label{Closed convex hull}
	Let $C$ we a weakly compact subset of a Banach space $(X,\norm)$.  Then $\overline{\mbox{co}}(C)$ is also weakly compact.
\end{corollary}
\begin{proof}
	Let $K:=\overline{\text{co}}(C)$. Since $C$ is weakly compact, every $x^*\in X^*$ must attain its supremum over $C$ i.e. for every $x^*\in X^*$, there exists $c\in C\seq K$ such that $x^*(c)=\sup_{x\in C}x^*(x)$.
	However, for every $x^*\in X^*$, it is a routine observation that 
	$$\sup_{x\in C}x^*(x)=\sup_{x\in \text{co}(C)}x^*(x)=\sup_{x\in K}x^*(x).$$ 
	And so, every $x^*\in X^*$ attains its supremum over $K$ too. Therefore, by James' Theorem (Theorem \ref{JamesFull}), $K$ is weakly compact. 
\end{proof}

Using Theorem \ref{Theorem 1} we can prove some well-known results of S.~Simons, see \cite{S.Simons}. For a detailed survey of Simons' results and applications thereof, see \cite{Cascales-simple}.

\begin{theorem}[Simons, \cite{S.Simons}]\label{Simons1}
	Let $K$ be a weak$^*$-compact, convex subset of the dual of a Banach space $(X,\|\cdot\|)$, let $B$ be a boundary for $K$, and let $f_n:K\rightarrow\r$ be a weak$^*$-lower-semicontinuous, convex function for all $n\in\n$. If $(f_n:n\in\n)$ is equicontinuous with respect to the norm, and $\displaystyle\limsup_{n\rightarrow\infty}f_n(b^*)\leq 0$ for all $b^*\in B$, then $\displaystyle\limsup_{n\rightarrow\infty}f_n(x^*)\leq 0$ for all $x^*\in K$. 
\end{theorem}

\begin{proof}
	Let $\epsilon>0$. For each $n\in\n$, define:
	$$C_n:=\bigcap_{k\geq n}\{y^*\in K:f_k(y^*)\leq \epsilon/2\}.$$
	Let $k\in\n$. Since $f_k:K\rightarrow\r$ is weak$^*$-lower-semicontinuous and convex, the set $\{y^*\in K:f_k(y^*)\leq\epsilon/2\}$ is weak$^*$-closed and convex. It follows that for all $n\in\n$, $C_n$ is the intersection of weak$^*$-closed and convex sets, and so is weak$^*$-closed and convex itself.
	Then, since $C_n\seq K$ for all $n\in\n$, we have that $C_n$ is weak$^*$-compact and convex for all $n\in\n$. 
	Moreover, if $b^*\in B$, then $\limsup_{n\rightarrow\infty}f_n(b^*)\leq 0$ and so $b^*\in C_N$ for some $N\in\n$. Hence, $(C_n:n\in\n)$ is a countable cover of $B$ by weak$^*$-compact, convex subsets of $K$. \\ \\
	Therefore, since $B$ is a boundary for $K$, by Theorem \ref{Theorem 1} we have that $\text{co}[\bigcup_{n\in\n}C_n]=\bigcup_{n\in\n}C_n$ (since $C_n \subseteq C_{n+1}$ for all $n \in \N$) is norm-dense in $K$. 
	Let $x^*\in K$. Since $(f_n:n\in\n)$ is equicontinuous with respect to the norm, there exists a $\delta>0$ such that $f_n(x^*)<f_n(y^*)+\epsilon/2$ for all $n\in\n$ and all $y^*\in B(x^*,\delta)$.  \\ \\
	However since $\bigcup_{n\in\n}C_n$ is norm-dense in $K$, there exists $N\in\n$ such that $B(x^*,\delta)\cap C_N\neq\varnothing$. 
	Therefore, $f_n(x^*)<\epsilon$ for all $n>N$ and so $\limsup_{n\rightarrow\infty}f_n(x^*)\leq\epsilon$. Since $\epsilon>0$ and $x^*\in K$ were arbitrary, we have that $\limsup_{n\rightarrow\infty}f_n(x^*)\leq 0$ for all $x^*\in K$ as claimed.
\end{proof}

\medskip
\begin{corollary}\label{RainwaterSimons}
	Let $K$ be a weak$^*$-compact, convex subset of the dual of a Banach space $(X,\|\cdot\|)$ and let $B$ be a boundary for $K$. Let $(x_n:n\in\n)$ be a bounded sequence in $X$ and let $x\in X$. 
	If $\displaystyle \lim_{n \to \infty} b^*(x_n)= b^*(x)$ for all $b^*\in B$, then $\displaystyle \lim_{n \to \infty} x^*(x_n) = x^*(x)$ for all $x^*\in K$.
\end{corollary}

\begin{proof}
	For all $n\in\n$, define $f_n:K\rightarrow\r$ to be given by
	$$f_n(x^*):=|x^*(x_n)-x^*(x)| = |\widehat{(x_n - x)}(x^*)|  \quad \text{for all} \ x^*\in K.$$
	Then $f_n:K\rightarrow\r$ is a weak$^*$-lower-semicontinuous and convex function for all $n\in\n$, as $x^* \mapsto \widehat{(x_n - x)}(x^*)$ is weak$^*$ continuous and linear (into $\R$) and $r \mapsto |r|$ is continuous
	and convex. Furthermore, $(f_n:n\in\n)$ is equicontinuous with respect to the norm. 
	Finally, $\limsup_{n\rightarrow\infty}f_n(b^*)\leq 0$ for all $b^*\in B$ and so, by Theorem~\ref{Simons1},  $\limsup_{n\rightarrow\infty}f_n(x^*)\leq 0$ for all $x^*\in K$. From this it is clear that $\displaystyle \lim_{n \to \infty} x^*(x_n) =x^*(x)$ for all $x^*\in K$.
\end{proof}

Sometimes called the Rainwater-Simons Theorem, Corollary \ref{RainwaterSimons} is due to S.~Simons (although he proved it differently). It generalises a famous result of J.~Rainwater, originally from \cite{rainwater}. 

\begin{corollary}[Simons]\label{Simons2}
	Let $K$ be a weak$^*$-compact, convex subset of the dual of a Banach space $(X,\|\cdot\|)$, let $B$ be a boundary for $K$, and let $(x_n:n\in\n)$ be a bounded sequence in $X$. Then
	$$\sup_{b^*\in B}\left\{\limsup_{n\rightarrow\infty}\widehat{x}_n(b^*)\right\}=\sup_{x^*\in K}\left\{\limsup_{n\rightarrow\infty}\widehat{x}_n(x^*)\right\}.$$
\end{corollary}

\begin{proof}
	Since $B\seq K$, clearly
	$$\sup_{b^*\in B}\left\{\limsup_{n\rightarrow\infty}\widehat{x}_n(b^*)\right\}\leq\sup_{x^*\in K}\left\{\limsup_{n\rightarrow\infty}\widehat{x}_n(x^*)\right\}.$$
	So it only remains to show that 
	$$\sup_{b^*\in B}\left\{\limsup_{n\rightarrow\infty}\widehat{x}_n(b^*)\right\}\geq\sup_{x^*\in K}\left\{\limsup_{n\rightarrow\infty}\widehat{x}_n(x^*)\right\}.$$
	To this end, let 
	$$r:=\sup_{b^*\in B}\left\{\limsup_{n\rightarrow\infty}\widehat{x}_n(b^*)\right\},$$
	and for each $n\in\n$, let $f_n:K\rightarrow\r$ be defined by
	$f_n(x^*):=\sup\{\widehat{x}_k(x^*):k\geq n\}-r$ for all $x^*\in K$. Then, for all $n\in\n$, $f_n$ is weak$^*$-lower-semicontinuous and convex, as the pointwise supremum of a family of convex
	functions is again convex and the pointwise supremum of a family of lower semi-continuous functions is again lower semi-continuous. Furthermore,  $(f_n:n\in\n)$ is equicontinuous with respect to the norm and
	moreover, $\lim_{n\rightarrow\infty}f_n(b^*)\leq 0$ for all $b^*\in B$. Therefore, by Theorem \ref{Simons1}, $\lim_{n\rightarrow\infty}f_n(x^*)\leq 0$ for all $x^*\in K$. From this, the result is immediate.
\end{proof}

\medskip

In the next part of this subsection we will show that in order to deduce that a closed and bounded convex subset $C$ of Banach space $(X,\norm)$ is weakly compact it is not
necessary to show that {\bf all} the elements of $X^*$ attain their maximum value of $C$, but only a ``large'' subset of $X^*$.  To achieve this goal we need some more definitions.

\medskip

Let $K$ be a subset of the dual of a normed linear space $(X,\norm)$.  A point $x^* \in K$ is called a {\it weak$^*$ exposed point of $K$} if there exists a $x \in X \setminus \{0\}$ such that
$\widehat{x}(x^*) \geq  \sup_{y^* \in K} \widehat{x}(y^*)$.  There are some simple, but useful, facts that we can easily deduce about weak$^*$ exposed points.

\medskip

Firstly, (i) if $x^*$ is a weak$^*$ exposed point of $K$ then $\lambda x^*$ is a weak$^*$ exposed point of $\lambda K$ for any $\lambda \in \R\setminus\{0\}$; (ii) if  $x^*$ is a weak$^*$ exposed point of $K$ then $x^* + y^*$
is a weak$^*$ exposed point of $K + y^*$ for any $y^* \in X^*$; (iii) if $x^* \in A \subseteq K$ is a weak$^*$ exposed point of $K$ then $x^*$ is a weak$^*$ exposed point of $A$.

\medskip

The next result shows that weak$^*$ exposed points are directly related to weak compactness.

\begin{proposition}\label{weak* exposed} Let $K$ be a closed and convex subset of the dual of a Banach space $(X,\norm)$.  If $0 \in \mathrm{int}(K)$ and every point of $\mathrm{Bb}(K)$ is a weak$^*$ exposed point of $K$,
then $K_\circ$ is a weakly compact subset of $X$.
\end{proposition}
\begin{proof} We shall appeal directly to Theorem \ref{JamesFull}. To this end, let $x^* \in X^*\setminus\{0\}$. We consider two cases.

{\bf Case (I)} Suppose that for every $0< \lambda$, $\lambda x^* \in K$.  Let $k \in K_\circ$ and let $0< \lambda$. Then 
$$\lambda x^*(k) = (\lambda x^*)(k) \leq 1 \mbox{\quad since, $\lambda x^* \in K \subseteq (K_\circ)^\circ$.}$$ 
Therefore, $x^*(k) \leq \lambda^{-1}$. Since $0<\lambda$ was arbitrary, $x^*(k) \leq 0$ and so $x^*$ attains its maximum value over $K_\circ$ at $0 \in K_\circ$.

\medskip

{\bf Case(II)} Suppose that for some $0<\lambda$, $(\lambda x^*) \not\in K$. Let $\Lambda := \{r \in [0, \infty): rx^* \in K\}$. Then $\Lambda$ is a closed and bounded interval of $[0, \infty)$ since $K$ is closed
and convex and $\lambda \not\in \Lambda$.  Let $\lambda_0 := \max_{r \in \Lambda} r$. Then $\lambda_0x^* \in \mathrm{Bd}(K)$.  Hence there exists a $x \in X \setminus \{0\}$ such that
$$\lambda_0\widehat{x}(x^*) = \widehat{x}(\lambda_0x^*) = \sup_{y^* \in K} \widehat{x}(y^*) >0 \mbox{\quad since, $0 \in \mathrm{int}(K)$.}$$
By replacing $x$ by $\mu x$ for some $\mu >0$ and relabelling if necessary, we can assume that
$$1 = \widehat{x}(\lambda_0x^*) = \sup_{y^* \in K} \widehat{x}(y^*) =  \sup_{y^* \in K}  y^*(x).$$
Therefore $x \in K_\circ$. On the other hand, since $\lambda_0 x^* \in K \subseteq (K_\circ)^\circ$ we have that
$$(\lambda_0x^*)(k) \leq 1 = (\lambda_0x^*)(x) \mbox{\quad for all $k \in K_\circ$.}$$
Therefore $\lambda_0x^*$ attains its maximum value over $K_\circ$ at $x$, and hence so does $x^*$.  Therefore, by Theorem \ref{JamesFull}, $K_\circ$ is weakly compact.
\end{proof}

\begin{theorem}[\!\!\cite{Moreno}]\label{Weak^*Interior}
	Let $(X,\norm)$ be a Banach space. If there exists a weak$^*$ open subset $U$ of $X^*$ such that $\varnothing \not= S_{X^*} \cap U$ and every member of
	 $S_{X^*} \cap U$ attains its norm on $X$, then $X$ is reflexive.
\end{theorem}

\begin{proof} Suppose that $\{x_1, x_2, \ldots, x_n\} \subseteq X$, $\varepsilon >0$ and $x_0^* \in X^*$ are chosen so that if 
$$W' := \bigcap_{k=1}^n \{x^*\in X^* : |x^*(x_k) - x_0^*(x_k)| < \varepsilon\}  \mbox{ \quad and \quad }  W := \bigcap_{k=1}^n \{x^*\in X^* : |x^*(x_k) - x_0^*(x_k)| \leq \varepsilon\}$$
then $\varnothing \not= S_{X^*} \cap W'$ and every member of $S_{X^*} \cap W$ attains its norm on $X$ (i.e., every point of  $S_{X^*} \cap W$ is a weak$^*$ exposed point of $B_{X^*}$ - 
just consider the $x \in B_X$ such that $\widehat{x}(x^*) = \|x^*\| =1$). Let $K' := W \cap B_{X^*}$. Then $K'$ is closed and bounded and convex. Furthermore, $\mathrm{int}(K') \not= \varnothing$.
Let us now recall some basic facts from general topology.  If $A$ and $B$ are closed subsets of a topological space $(T, \tau)$ then 
$$\mathrm{Bd}(A \cap B) \subseteq (\mathrm{Bd}(A) \cap B) \cup (\mathrm{Bd}(B) \cap A) \subseteq \mathrm{Bd}(A) \cup \mathrm{Bd}(B).$$
Perhaps the easiest way to convince yourself of this is to first show that $\mathrm{Bd}(A \cap B) \subseteq \mathrm{Bd}(A) \cup \mathrm{Bd}(B)$.  Then
\begin{eqnarray*}
\mathrm{Bd}(A \cap B) &=& [\mathrm{Bd}(A \cap B)]	\cap [\mathrm{Bd}(A) \cup \mathrm{Bd}(B)]	\\
&\subseteq&  [A \cap B] \cap [\mathrm{Bd}(A) \cup \mathrm{Bd}(B)] =  (A \cap B) \cap \mathrm{Bd}(A) \cup (A \cap B) \cap \mathrm{Bd}(B) \\
&\subseteq& [B \cap \mathrm{Bd}(A)]  \cup [A \cap \mathrm{Bd}(B)].
\end{eqnarray*}
So $\mathrm{Bd}(K') \subseteq [\mathrm{Bd}(W) \cap B_{X^*}] \cup [S_{X^*} \cap W]$. We claim that every point
of $\mathrm{Bd}(K')$ is a weak$^*$ exposed point.  To see this, suppose that $x^* \in \mathrm{Bd}(W) \cap B_{X^*} \subseteq \mathrm{Bd}(W)$.  Then clearly, $x^*$ is a weak$^*$ exposed point of the set
$W$ (exposed by $\widehat{x_k}$ for some $1 \leq k \leq n$).  Then by property (iii) above, $x^*$ is a weak$^*$ exposed point of $W \cap B_{X^*}$.  If $x^* \in S_{X^*} \cap W$ then
by the way $W$ was chosen, $x^*$ is a weak$^*$ exposed point of $B_{X^*}$ and hence by property (iii) above, also a weak$^*$ exposed point of $B_{X^*} \cap W$. Choose $x^* \in \mathrm{int}(K')$ and
let $K := K' - x^*$.  Then $0 \in \mathrm{int}(K)$ and by property (ii) above, each point of $\mathrm{Bd}(K)$ is a weak$^*$ exposed point. Thus, by Proposition~\ref{weak* exposed}, $K_\circ$ is weakly compact.
Now since $K$ is bounded, $0 \in \mathrm{int}(K_\circ)$.  Hence $X$ is reflexive.
\end{proof}

The proof of the next theorem can be found in \cite{moors-sublevel}, (see also \cite{OrihuelaCoercive, OrihuelaLebesgue}).

\begin{theorem}[\!\!\cite{SaintRaymond}]\label{sub level sets} Let $(X,\norm) $ be a Banach space and let $f:X \to \mathbb{R}\cup\{\infty\}$ be a proper function on $X$.  If $f - x^*$ attains minimum for every $x^* \in X^*$ then for each $a \in \mathbb{R}$, 
$S(a) := \{(y,s) \in X\times \mathbb{R}: f(y) \leq s \leq a\}$ is relatively weakly compact.
\end{theorem} 

\begin{proof} In this proof we will identify the dual of $X \times \R$ with $X^*\times \R$.  We will also consider $X \times \R$ endowed with the norm $\|(x,r)\|_1 := \|x\| + |r|$ and note that with this norm, $(X \times \R, \|\cdot\|_1)$ is a Banach space.  We shall apply James' theorem, (Theorem \ref{JamesFull}),  
in $X \times \R$.  Let $H := \{(x,r) \in X\times\R:r =0\}$ and define $T:(X \times \R) \setminus H \to (X \times \R)\setminus H$ by, 
$T(x,r) := r^{-1}(x,-1)$. Then $T$ is a bijection.  In fact, $T$ is a homeomorphism when $(X\times \R)\setminus H$ is considered with the relative weak topology.  Note that since $f$ is bounded below we may assume, after
possibly translating, that $1 = \inf_{x \in X} f(x)$.  Our proof relies upon the {\it Fenchel conjugate}, $f^*:X^* \to \R$ {\it of $f$}, which is defined by, 
$$f^*(x^*) := \sup_{x \in X} [x^*(x)-f(x)] = -\inf_{x \in X} [f(x)-x^*(x)]  = -\min_{x \in X} [f(x)-x^*(x)] =  \max_{x \in X} [x^*(x)-f(x)].$$  
It is routine to check that $f^*$ is convex on $X^*$. 
We claim that $\overline{\mbox{co}}[T(\mbox{epi}(f)) \cup \{(0,0)\}]$ is weakly compact. To show this, it is sufficient, because of James' theorem, to show that every non-zero continuous linear functional attains
its maximum value over $T(\mbox{epi}(f)) \cup \{(0,0)\}$.  To this end, let $(x^*, r) \in (X^* \times \R)\setminus\{(0,0)\}$.  We consider two cases.

\medskip

{\bf Case (I)} Suppose that for every $0< \lambda$, $f^*(\lambda x^*) \leq \lambda r$. Then $x^*(x) - \lambda^{-1}f(x) \leq r$ for all $x \in X$ and 
all $0< \lambda$. Let $(y, s) \in \mbox{epi}(f)$ and let $0 < \lambda$. Then,
$$(x^*,r)(T(y,s)) = s^{-1}(x^*(y) - r) \leq s^{-1}(x^*(y) - [x^*(y) - \lambda^{-1}f(y)])  =s^{-1} f(y) \lambda^{-1} \leq \lambda^{-1}$$
since $f(y) \leq s$. As $0<\lambda$ was arbitrary,  $(x^*,r)(T(y,s)) \leq 0 = (x^*,r)(0,0)$.  Thus, $(x^*,r)$ attains its maximum value over $T(\mbox{epi}(f)) \cup \{(0,0)\}$ at $(0,0)$.

\medskip

{\bf Case(II)} Suppose that for some $0<\lambda$, $\lambda r < f^*(\lambda x^*)$. Then, since the mapping, $\lambda' \mapsto f^*(\lambda'x^*)$, is real-valued and convex, it is continuous.
Furthermore, 
it follows, from the intermediate value theorem applied to the function $g:[0,\lambda] \to \R$, defined by, 
$$g(\lambda') := f^*(\lambda'x^*) - \lambda'r \mbox{\quad for all $\lambda' \in [0,\lambda]$,}$$
 that there exists a $0 < \mu < \lambda$ such that 
$g(\mu) =0$, i.e., $f^*(\mu x^*) = \mu r$, since $g(0) = -1 < 0 <g(\lambda)$. 
Thus, $\mu(x^*, r) = (\mu x^*, f^*(\mu x^*))$.  Choose $z \in X$ such that 
$f^*(\mu x^*) = \mu x^*(z) - f(z)$.  We claim that $(x^*,r)$ attains its maximum value over $T(\mbox{epi}(f)) \cup \{(0,0)\}$ at $T(z, f(z)) = f(z)^{-1}(z,-1)$. Now,
\begin{eqnarray*}
(x^*,r)(T(z, f(z))) &=& f(z)^{-1}(x^*(z) - r) = f(z)^{-1}(x^*(z) - [\mu^{-1}f^*(\mu x^*)]) \\
&=& f(z)^{-1}(x^*(z) - [x^*(z) - \mu^{-1}f(z)]) = \mu^{-1} >0.
\end{eqnarray*}
On the other hand, if $(y,s) \in \mbox{epi}(f)$ then 
\begin{eqnarray*}
(x^*,r)(T(y, s)) &=& s^{-1}(x^*(y) - r) = s^{-1}(x^*(y) - [\mu^{-1}f^*(\mu x^*)]) \\
&\leq & s^{-1}(x^*(y) - [x^*(y) - \mu^{-1}f(y)]) = s^{-1}f(y)\mu^{-1} \leq \mu^{-1}= (x^*,r)(T(z, f(z)))
\end{eqnarray*}
since $f(y) \leq s$. Note also that $(x^*,r)(0,0) = 0 < \mu^{-1} = (x^*,r)(T(z, f(z)))$.  Therefore, by James' Theorem \ref{JamesFull} , $\overline{\mbox{co}}[T(\mbox{epi}(f)) \cup \{(0,0)\}]$ is weakly compact.

\medskip

Let $1 \leq a$, then $T(S(a)) \subseteq \overline{\mbox{co}}[T(\mbox{epi}(f)) \cup \{(0,0)\}] \cap \{(x,r)\in X \times \R: r \leq -a^{-1}\}$; which is weakly compact.  Therefore,
$$S(a) \subseteq T^{-1}(\overline{\mbox{co}}[T(\mbox{epi}(f)) \cup \{(0,0)\}] \cap \{(x,r)\in X \times \R:  r \leq -a^{-1}\});$$ 
which completes the proof. 
\end{proof}

\medskip

For each $a \in \R$, let $L(a) := \{x \in X: f(x) \leq a\}$.  It follows from Theorem \ref{sub level sets} that if $(X, \norm)$ is a Banach space, $f:X \to \R\cup\{\infty\}$ is a proper function on $X$ and $f - x^*$ attains minimum for every $x^* \in X^*$  then, for each $a \in \R$,  $L(a)$ is relatively weakly compact, since $L(a) = \pi(S(a))$, where $\pi: X\times \R \to X$ is defined by, $\pi(x,r) := x$ for all $(x,r) \in X \times \R$ and is weak-to-weak continuous, 
(see Proposition \ref{linearmap}).

\medskip

An interesting corollary of this result is the following.

\begin{corollary}[\!\!\cite{SaintRaymond}] Let $\varphi:U \to \R$ be a continuous convex function defined on a nonempty open convex subset $U$ of a Banach space $(X,\norm)$. If $\varphi - x^*$ attains minimum for every $x^* \in X^*$ then
$X$ is reflexive.
\end{corollary}
\begin{proof} For each $n \in \N$, let $F_n:= \{x \in U:\varphi(x) \leq n\}$.  Then each set $F_n$ is closed and $U = \bigcup_{n \in \N} F_n$.  Since $X$ is a Banach space, $U$ is of the second
Baire category. Thus, there exists an $n_0 \in \N$ such that $\mathrm{int}(F_{n_0}) \not= \varnothing$.  In particular, there exists an $x_0 \in F_{n_0}$ and a $\delta_0 >0$ such that $B[x_0,\delta_0] \subseteq F_{n_0}$. Therefore, by Theorem \ref{sub level sets}, $B[x_0,\delta_0] = x_0 + \delta_0B_X$ is compact with respect to the weak topology, and hence so is $B_X$.  The result now follows from Theorem \ref{reflexive}.
\end{proof}

For any nonempty bounded subset $A$ of a Banach space $(X,\norm)$ and any $x^* \in X^*$ we shall denote by, $\sup(x^*,A) := \sup \{x^*(a):a \in A\}$ and by  $\inf(x^*,A) := \inf \{x^*(a):a \in A\}$.

\begin{lemma}[\!\!\cite{Moors-1-sided}] \label{Lemma 1} Let $(Y,\|\cdot\|)$ be a Banach space and $C$ be a nonempty bounded subset of $Y \times \R$, endowed with the norm $\|(y,r)\|_1 := \|y\| +|r|$.  If for every $x^* \in Y^*$, 
$\max\{(x^*,-1)(y,s): (y,s) \in C\}$ exists then 
$C$ is relatively weakly compact.
\end{lemma}
\begin{proof} Let $\pi:Y \times \R \to Y$ be defined by $\pi(y,r) :=y$, $A := \pi(C)$ and $f:Y \to \R \cup\{\infty\}$ be defined by,
$$f(y) := \left\{\begin{array}{ll}
\inf\{s \in \R:(y,s) \in C\} & \mbox{if $y \in A$} \\
\infty & \mbox{if $y \not\in A$.}
\end{array} \right.$$
Then $f$ is a proper function on $Y$ and $x^*-f$ attains it maximum for every $x^* \in Y^*$.  To see this, consider the following. Let $x^* \in Y^*$, then
\begin{eqnarray*}
\sup_{y \in Y} (x^* -f)(y) &=& \sup_{(y,s) \in \mathrm{epi}(f)} (x^*,-1)(y,s) \\
&=&  \sup_{(y,s) \in C} (x^*,-1)(y,s) =   \max_{(y,s) \in C} (x^*,-1)(y,s)\\
&=&  \max_{(y,s) \in \mathrm{epi}(f)} (x^*,-1)(y,s) \\
&=& \max_{y \in Y} (x^* -f)(y).
\end{eqnarray*}
Therefore, by  Theorem \ref{sub level sets}, for each $a \in \R$, $S(a) := \{(y,s) \in Y \times \R: f(y) \leq s \leq a\}$
is relatively weakly compact.  Since $C$ is bounded there exists an $a \in \R$ such that $C \subseteq  S(a)$. 
\end{proof}

\begin{theorem}[\!\!\cite{Moors-1-sided}] \label{one-sided} Let $(X,\norm)$ be a Banach space and let $A$ and $B$ be bounded, closed and convex sets with $\mbox{dist}(A,B) >0$. 
If every $x^* \in X^*$ with $\sup(x^*, B) < \inf(x^*, A)$ attains its infimum on $A$ and its supremum on $B$, then both $A$ and $B$ are weakly compact.
\end{theorem}


\begin{proof} To show that both $A$ and $B$ are weakly compact it is sufficient (and necessary) to show that $B-A$ is weakly compact. This will be our approach. 
From the hypotheses it follows that if $C := \overline{B-A}$, then $C$ is a bounded nonempty closed and convex subset of $X$ with $0 \not\in C$. Furthermore, it follows that each $x^* \in X^*$ with
$\sup(x^*, C) <0$ attains it supremum on $C$. Choose $y^* \in X^*$ such that $\sup(y^*,C) <0$.  Note that such a functional exists by the Hahn-Banach theorem.  Let $Y := \mbox{ker}(y^*)$ and choose
$x_0 \in C$. Define $S: Y \times \R \to X$ by, $S(y,r) := y +rx_0$ and let us consider $Y \times \R$ endowed with the norm $\|(y,r)\|_1 := \|y\| +|r|$.  Then $S$ is an isomorphism and there exists an $0<\varepsilon$
such that $S^{-1}(C) \subseteq \{(y,r) \in Y \times \R: \varepsilon \leq r\}$.  Moreover, each $(x^*,r) \in (Y \times \R)^*$ with $\sup((x^*,r),S^{-1}(C)) <0$ attains its supremum over $S^{-1}(C)$.
Let $\pi:Y \times \R \to Y$ be defined by $\pi(y,r) :=y$, $A := \pi(S^{-1}(C))$ and $f:Y \to \R \cup\{\infty\}$ be defined by,
$$f(y) := \left\{\begin{array}{ll}
\inf\{s \in \R:(y,s) \in S^{-1}(C)\} & \mbox{if $y \in A$} \\
\infty & \mbox{if $y \not\in A$.}
\end{array} \right.$$
Next, we define $T:Y \times (\R \setminus \{0\}) \to Y \times (\R \setminus\{0\})$ by
$T(y,s) := s^{-1}(y,-1)$. Then $T$ is a bijection.  In fact, $T$ is a homeomorphism when $Y\times (\R\setminus\{0\})$ is considered with the relative weak topology.
Let $f^*:Y^* \to \R$ be defined by, 
$$f^*(x^*) := \sup_{y \in Y} [x^*(y)-f(y)] = \sup((x^*,-1),S^{-1}(C)).$$  
It is routine to check that $f^*$ is real-valued and convex on $Y^*$. 
To show that $C$ is weakly compact it is sufficient to show that $T(S^{-1}(C))$ is a relatively weakly compact subset of $Y \times \R$. To achieve this we appeal to Lemma \ref{Lemma 1}. 
First note that $T(S^{-1}(C))$ is a nonempty bounded subset of $Y \times \R$.  Then consider any  $x^* \in Y^*$. We consider two cases.

\medskip

{\bf Case (I)} Suppose that for every $0< \lambda$, $f^*(\lambda x^*) \leq -\lambda$. Then $x^*(y) - \lambda^{-1}f(y) \leq -1$ for all $y \in Y$ and 
all $0< \lambda$.  In particular, $-\lambda^{-1}f(0) \leq -1$ for all $0 < \lambda$, i.e., $\lambda \leq f(0)$ for all $0< \lambda$. On the other hand,
$S(0,1) = x_0 \in C$, i.e., $(0,1) \in S^{-1}(C)$ and so $f(0) \leq 1$. Thus, Case (I) does not occur.

\medskip

{\bf Case(II)} Suppose that for some $0<\lambda$, $-\lambda < f^*(\lambda x^*)$. Then, since the mapping, $\lambda' \mapsto f^*(\lambda'x^*)$, is real-valued and convex, it is continuous.
Furthermore,  it follows from the intermediate value theorem applied to the function $g:[0,\lambda] \to \R$, defined by, 
$$g(\lambda') := f^*(\lambda'x^*) + \lambda' \mbox{\quad for all $\lambda' \in [0,\lambda]$,}$$
 that there exists a $0 < \mu < \lambda$ such that 
$g(\mu) =0$, i.e., $f^*(\mu x^*) = -\mu$, since 
$$g(0) = f^*(0x^*) = -\inf_{y \in Y} f(y) \leq -\varepsilon < -0 =  0 <g(\lambda).$$
 Thus, $\mu(x^*, -1) = (\mu x^*, f^*(\mu x^*))$ and so $f^*(\mu x^*) = \sup((\mu x^*, -1), S^{-1}(C)) = -\mu <0$.

\medskip

Choose $(z,s) \in S^{-1}(C)$ such that $(\mu x^*, -1)(z,s) = \sup((\mu x^*, -1), S^{-1}(C)) = f^*(\mu x^*)$. Note that $z \in A$ and $s = f(z)$.
We claim that $(x^*,-1)$ attains its maximum value over $T(S^{-1}(C))$ at $T(z, f(z)) = f(z)^{-1}(z,-1)$. Now,
\begin{eqnarray*}
(x^*,-1)(T(z, f(z))) &=& f(z)^{-1}(x^*(z) +1) = f(z)^{-1}(x^*(z) - [\mu^{-1}f^*(\mu x^*)]) \\
&=& f(z)^{-1}(x^*(z) - [x^*(z) - \mu^{-1}f(z)]) = \mu^{-1}.
\end{eqnarray*}
On the other hand, if $(y,s) \in S^{-1}(C)$ then 
\begin{eqnarray*}
(x^*,-1)(T(y, s)) &=& s^{-1}(x^*(y) +1 ) = s^{-1}(x^*(y) - [\mu^{-1}f^*(\mu x^*)]) \\
&\leq & s^{-1}(x^*(y) - [x^*(y) - \mu^{-1}f(y)]) = s^{-1}f(y)\mu^{-1} \leq \mu^{-1}= (x^*,-1)(T(z, f(z)))
\end{eqnarray*}
since $f(y) \leq s$.  This completes the proof. 
\end{proof}

\begin{remark} It might be interesting to note the following: If $(X,\|\cdot\|)$ is a Banach space, $A$ and $B$ are nonempty bounded, closed and convex sets such that
every $x^* \in X^*$ with $\inf(x^*, A) <  \sup(x^*, B)$ attains its infimum on $A$ and its supremum on $B$, then both $A$ and $B$ are weakly compact.
To see this, note that $C := \mbox{co}[\{0\} \cup \overline{B-A}]$ is a closed and bounded convex subset of $X$ with the property that every continuous linear function attains it supremum
over $C$. 
\end{remark}

A special case of the previous theorem was given in \cite{CascalesOneSided}.  

\begin{example}\label{Example} Let $(X,\|\cdot\|)$ be a non-trivial normed linear space. Then there exists an equivalent norm $|\!|\!|\cdot |\!|\!|$ on $X$ and a nonempty open subset $U$ of $X^*$
such that every member of $U$ attains its norm on $(X, |\!|\!|\cdot |\!|\!|)$.
\end{example}
\begin{proof} Choose $x_0 \in X$ with $\|x_0\|=2$.  Then, by the Hahn-Banach theorem, there exists a continuous linear functional $x^* \in S_{X^*}$ such that $x^*(x_0) = 2$. Let $U := \{y^* \in X^*:\|y^*-x^*\| <1/3\}$
and let $B := \mathrm{co}(B_X \cup \{x_0,-x_0\})$. Then $B$ is convex, bounded, symmetric and $0 \in \mathrm{int}(B)$.  Therefore,  $B$ is the closed unit ball of some equivalent
norm $|\!|\!|\cdot |\!|\!|$ on $X$.  Furthermore, every member of $U$ attains its maximum value over $B$ at $x_0$. Indeed, if $y^* \in U$ then 
$$y^*(x_0) = x^*(x_0) + [y^*(x_0) - x^*(x_0)] \geq 2 - \|y^*-x^*\|\|x_0\| > 4/3.$$ 
On the other hand, for any $x \in B_X$, 
$$y^*(x) = x^*(x) + [y^*(x)-x^*(x)] \leq 1 + \|y^*-x^*\|\|x\| < 1+ 1/3 = 4/3 < y^*(x_0)$$ 
and $y^*(-x_0) = -y^*(x_0) = -2 < y^*(x_0)$.  Therefore, $y^*$ attains its maximum value over $B$ at $x_0$.
\end{proof}

Together, Example \ref{Example} and Theorem \ref{Weak^*Interior}  give rise to the following conjecture.

\begin{conjecture} Let $(X,\norm)$ be a Banach space. If there exists a weak open subset $U$ of $X^*$ such that $\varnothing \not= S_{X^*} \cap U$ and every member of
	 $S_{X^*} \cap U$ attains its norm on $X$, then $X$ is reflexive.
\end{conjecture}

A special case of this conjecture was proven in \cite{Moreno}. For some further results in this direction see \cite{Debs}.


\section{Convex analysis and minimal uscos}


In this section we prove a generalisation of James' weak compactness theorem.  Unfortunately, to achieve this generalisation we will need to take an excursion into
convex analysis and set-valued analysis.  Hopefully, some of the results along the way are of some interest in their own right.

\subsection{Convex functions and monotone operators}

We shall need the following very important fact regarding the continuity of convex functions.

\begin{proposition}[\!\!\mbox{\cite[Proposition 1.6]{PhelpsBook}}]\label{Lipschitz}
	Let $U$ be a nonempty open convex subset of a Banach space $(X,\norm)$ and let $\varphi:U\rightarrow\r$ be a convex function. 
	If $\varphi$ is {\it locally bounded above on $U$}, that is, for every $x_0 \in U$ there exists an $M>0$ and a $\delta >0$ such that $B(x_0,\delta)\seq U$ and 
	$\varphi(x) \leq M$ for all $x \in B(x_0,\delta)$, then it is {\it locally Lipschitz on $U$}; that is, for every $x_0\in U$, there exists an $L>0$ and $\delta>0$ such that $B(x_0,\delta)\seq U$ and 
	$$|\varphi(x)-\varphi(y)|\leq L\|x-y\|$$
	for all $x,y\in B(x_0,\delta)$.
\end{proposition}
\begin{proof} Let $x_0 \in U$.  Choose $M^* >0$ and $\delta >0$ such that $B(x_0,2\delta) \subseteq U$ and $\varphi(x) \leq M^*$ for all $x \in B(x_0,2\delta)$. Then for all $x \in B(x_0,\delta)$ we have that
$2x_0-x=x_0-(x-x_0) \in B(x_0,\delta)$ and $x_0 = (1/2)(2x_0-x) + (1/2)x$.  Hence,
$$\varphi(x_0) \leq \frac{\varphi(2x_0-x) + \varphi(x)}{2} \leq \frac{M^* + \varphi(x)}{2},$$
so $-\varphi(x) \leq M^* +2|\varphi(x_0)|$; that is, $|\varphi(x)| \leq (M^* + 2|\varphi(x_0)|) =:M'$ for all $x \in B(x_0,\delta)$.  So $|\varphi|$ is bounded by $M'$ on  $B(x_0,\delta)$. Let $\delta' := \delta/2$.
If $x$ and $y$ are distinct points in $B(x_0,\delta')$, let $\alpha := \|x-y\|$ and let $z := y + (\delta/\alpha)(y-x)$.  Note that $z \in B(x_0,2\delta')$. Since  
$y = [\alpha/(\alpha + \delta')]z + [\delta'/(\alpha + \delta')]x$ is a convex combination (lying in $B(x_0, 2\delta')$), we have that
$\varphi(y) \leq [\alpha/(\alpha + \delta')] \varphi(z) +  [\delta'/(\alpha + \delta')] \varphi(x)$ and so
$$\varphi(y) -\varphi(x) \leq [\alpha/(\alpha + \delta')] (\varphi(z) -\varphi(x)) +  [\delta'/(\alpha + \delta')](\varphi(x) -\varphi(x)) \leq (\alpha/\delta')2M' = (2M'/\delta')\|x-y\|.$$
Interchanging $x$ and $y$ gives the desired result, with $M := 2M'/\delta'$.
\end{proof}

Suppose that $f:C \to \R$ is a convex function defined on a nonempty convex subset of a normed linear
 space $(X, \norm)$ and $x \in C$.  Then we define the {\it subdifferential $\partial f(x)$} by, 
$$\partial f(x) := \{x^* \in X^*: x^*(y) - x^*(x) \leq f(y) - f(x) \mbox{ for all $y \in C$}\}.$$

We can also define the subdifferential in terms of the right-hand derivative of $f$.  Suppose that $f:U \to \R$ is a convex function defined on a nonempty
open convex subset $U$ of a normed linear space $(X,\norm)$. Let $x_0 \in U$ and let $v \in X$. Then the {\it right-hand directional derivative of $f$, at the point $x_0 \in U$, in the direction $v$},
is defined to be $$ f_+'(x_0;v) := \lim_{\lambda \to 0^+} \frac{f(x_0+\lambda v)-f(x_0)}{\lambda}.$$
Now there is a subtlety that we have overlooked. Namely, how do we know if the limit exists?  Well, if we revisit Lemma \ref{convex}, then we can see why.
So suppose $f$, $x_0$ and $v \not= 0$ are as in the definition of $f_+'(x_0;v)$ and suppose that $0<\beta$ and $0< \beta'$ Then,
\begin{eqnarray*}
\frac{f(x_0 + (\beta + \beta')v) - f(x_0)}{\beta + \beta'} &=& \frac{[f(x_0 + (\beta + \beta')v)- f(x_0 + \beta v)] + [f(x_0 + \beta v) -f(x_0)]}{\beta + \beta'} \\
&\geq& \frac{1}{\beta + \beta}\left(\frac{\beta'}{\beta}[f(x_0+\beta v) - f(x_0)] + [f(x_0+\beta v) - f(x_0)]\right) \mbox{\quad by Lemma \ref{convex}.}\\
&=& \frac{1}{\beta + \beta}\left(\frac{\beta + \beta'}{\beta}[f(x_0+\beta v) - f(x_0)]\right) \\
&=& \frac{f(x_0+\beta v) - f(x_0)}{\beta}.
\end{eqnarray*}
Therefore, $t \mapsto \frac{f(x_0+tv) - f(x_0)}{t}$ is an increasing function over $(0,\delta)$ for some $\delta >0$ small enough so that
$x_0 + tv \in U$ whenever $0 <t < \delta$. Since one can also use Lemma~\ref{convex} to  show that
$$\frac{f(x_0+sv) - f(x_0)}{s} \leq \frac{f(x_0+tv) - f(x_0)}{t} \mbox{\quad for any $s < 0$ and $0<t$, (but small enough to stay in $U$),} $$
we see that the limit in the definition of the right-hand directional derivative always exists.

\medskip

We can now give the basic properties of the subdifferential mapping $x \mapsto \partial \varphi(x)$.

\begin{lemma}[\!\!\mbox{\cite[Proposition 1.11]{PhelpsBook}}] \label{nonempty} Let $U$ be a nonempty open and convex subset of a normed linear space $(X,\norm)$ and let $\varphi:U\rightarrow\r$ be a continuous convex function.  If $x_0 \in U$
then $\partial \varphi(x_0) \not= \varnothing$.
\end{lemma} 
\begin{proof} Let $x_0 \in U$ and define $p:X \to \R$ by, $p(x) := f_+'(x_0;x)$ for all $x \in X$. Note that $p$ is well-defined.  Let $0 < \mu < \infty$ and let $x \in X$ then
\begin{eqnarray*}
p(\mu x) &=& \lim_{\lambda \to 0^+} \frac{\varphi(x_0+\lambda(\mu x)) - \varphi(x_0)}{\lambda} \\
&=& \mu  \lim_{\lambda \to 0^+} \frac{\varphi(x_0+(\lambda\mu)x) - \varphi(x_0)}{\lambda \mu} \\
&=& \mu  \lim_{\lambda' \to 0^+} \frac{\varphi(x_0+\lambda' x) - \varphi(x_0)}{\lambda'}  \mbox{ \quad \quad (where, $\lambda' := \lambda \mu$.)}\\
&=& \mu p(x).
\end{eqnarray*}
So $p$ is positively homogeneous on $X$.  Next, choose $\delta >0$ such that $B[x_0,\delta] \subseteq U$.  We claim that $p$ is convex on $B[0,\delta]$.
Fix $n \in \N$ and define $p_n:B[0,\delta] \to \R$ by, 
$$ p_n(x) := \frac{\varphi(x_0 + (1/n)x)-\varphi(x_0)}{(1/n)} \mbox{\quad \quad for all $x \in B[0,\delta]$.}$$
Since, $x \mapsto x_0 + (1/n)x$, is an affine map, $x \mapsto \varphi(x_0 + (1/n)x)$, is convex, and so $p_n$ is also convex.
Now, $p(x) = \lim_{n \to \infty}p_n(x)$ for each $x \in B[0,\delta]$.  Therefore, $p|_{B[0,\delta]}$ is convex, as the pointwise limit of convex functions is again convex.
Since $p$ is also positively homogeneous on $X$ it is an easy exercise to show that $p$ is sublinear on $X$.

\medskip

Let $y_0$ be any element of $S_X$ and define $f:\mathrm{span}\{y_0\} \to \R$ by, $f(\lambda y_0) := \lambda p(y_0)$ for all $\lambda \in \R$. Then 
$f(\lambda y_0) =  \lambda p(y_0) = p( \lambda y_0) \leq p( \lambda y_0)$ for all $0 < \lambda < \infty$. Now, fix 
$0<\lambda < \infty$, then $$0 = p(0) = p((-\lambda)y_0 + \lambda y_0) \leq p((-\lambda) y_0) + p(\lambda y_0).$$
Therefore, $(-\lambda) p(y_0) =  -p(\lambda y_0) \leq p((-\lambda) y_0)$. Thus,
$$f((-\lambda) y_0) = (-\lambda)p(y_0) \leq p((-\lambda )y_0).$$
Hence, $f(\lambda y_0) \leq p(\lambda y_0)$ for all $\lambda \in \R$.  Thus, by the Hahn-Banach Theorem (Theorem \ref{Hahn-Banach}) there exists a linear
functional $F:X \to \R$ such that $F(x) \leq p(x)$ for all $x \in X$.  Note also, that by Proposition \ref{Lipschitz} and the definition of $p$, there exists an $L >0$
such that $F(x) \leq p(x) \leq L\|x\|$ for all $x \in X$.  Thus, $F \in X^*$. We claim that $F \in \partial \varphi(x_0)$.  To see this, let $x \in U$ then
\begin{eqnarray*}
F(x) - F(x_0) &=& F(x-x_0) \\
&\leq& p(x-x_0) \\
&=& \lim_{\lambda \to 0^+} \frac{\varphi(x_0+\lambda (x-x_0))-\varphi(x_0)}{\lambda} \\
&\leq&  \frac{\varphi(x_0+1 (x-x_0))-\varphi(x_0)}{1}  \mbox{\quad since, $\lambda \mapsto \frac{\varphi(x_0+\lambda(x-x_0)) - \varphi(x_0)}{\lambda}$, is increasing over $(0,1]$.}\\
&=& \varphi(x) - \varphi(x_0).
\end{eqnarray*}
This completes the proof.
\end{proof}

\begin{proposition}[\!\!\mbox{\cite[Proposition 1.11]{PhelpsBook}}] \label{SubdifferentialProperties}
		Let $U$ be a nonempty open and convex subset of a normed linear space $(X,\norm)$ and let $\varphi:U\rightarrow\r$ be a continuous convex function. 
		If $x_0\in U$, then $\partial\varphi(x_0)$ is a weak$^*$-compact convex subset of $X^*$. 
		Moreover, the map $x\mapsto \partial\varphi(x)$ is locally bounded at $x_0$. That is, there exists an $L>0$ and a $\delta >0$
		such that $B(x_0,\delta) \subseteq U$ and $\|x^*\|\leq M$ whenever $x\in B(x_0,\delta)$ and $x^*\in\partial\varphi(x)$.  
\end{proposition}
\begin{proof} For each $x \in U$, let $F_x := \{x^* \in X^*: x^*(x-x_0) \leq \varphi(x) - \varphi(x_0)\} = (\widehat{x-x_0})^{-1}(-\infty, \varphi(x)-\varphi(x_0)]$.
Thus, each set $F_x$ is weak$^*$ closed and convex. Now, $\partial \varphi (x_0) = \bigcap_{x \in U} F_x$. Therefore, $\partial \varphi(x_0)$ is weak$^*$
closed and convex. Let us now show that, $x \mapsto \partial \varphi(x)$, is locally bounded at $x_0$ (Note: this will then automatically show that $\partial \varphi(x)$ is weak$^*$ compact,
by Theorem \ref{Banach-Alaoglu}). By Proposition \ref{Lipschitz}, there exists a $L >0$ and a $\delta >0$ such that $B(x_0,\delta) \subseteq U$ and $|\varphi(x)-\varphi(y)\| \leq L\|x -y\|$
for all $x,y \in B(x_0,\delta)$. We claim that $\|x^*\| \leq L$ whenever $x \in B(x_0,\delta)$ and $x^* \in \partial \varphi(x)$. To this end, let $x \in B(x_0,\delta)$ and $x^* \in \partial \varphi(x)$.
Let $v \in S_X$ and choose $0<\mu$ such that $x + \mu v \in B(x_0,\delta)$.  Then,
$$x^*(v) = \frac{x^*((x+\mu v)-x)}{\mu} \leq \frac{\varphi(x+\mu v) - \varphi(x)}{\mu} \leq \frac{L\|\mu v\|}{\mu} = L.$$
Thus, $\|x^*\| \leq L$.  Note: we used here the simple fact that if $x^*(v) \leq L$ for all $v \in S_X$ then $\|x^*\| \leq L$.
\end{proof}

\medskip

One of the most important features of the subdifferential mapping of a convex function is that it belongs to a much studied class of set-valued mappings called ``monotone operators''.

\medskip

Let $T:X \to 2^{X^*}$ be a set-valued mapping from $(X, \norm)$ be a normed linear space into subsets of its dual $X^*$.  $T$ is said to be a {\it monotone operator} provided $(x^*-y^*)(x-y) \geq 0$
whenever $x,y \in X$ and $x^* \in T(x)$, $y^* \in T(y)$.

\medskip

\begin{proposition}[\!\!\mbox{\cite[Example 2.2]{PhelpsBook}}] \label{monotone} If $\varphi:U\rightarrow\r$ be a continuous convex function defined on a nonempty open convex subset $U$ of a normed linear space $(X,\norm)$ then
$T:X \to 2^{X^*}$ defined by,
\[ T(x):= \begin{cases}
	\partial \varphi(x), & x\in U \\
	\varnothing, & x\notin U
\end{cases}
\]
is a monotone operator on $X$.
\end{proposition} 
\begin{proof} Let $x^*, y^* \in X^*$ and suppose that $x^* \in T(x)$ and $y^* \in T(y)$ for some $x,y \in X$.  Then $x,y \in U$ since $T(x) \not=\varnothing$ and $T(y) \not= \varnothing$.
In fact, $T(x) = \partial \varphi(x)$ and $T(y) = \partial \varphi(y)$.  Therefore,
$$x^*(y-x) \leq \varphi(y)-\varphi(x) \mbox{ \quad and \quad } (-y^*)(y-x) = y^*(x-y) \leq \varphi(x) - \varphi(y).$$
If we add these two inequalities together we get $(x^*-y^*)(y-x) \leq 0$ and so $(x^*-y^*)(x-y) \geq 0$. Hence, $T$ is indeed a monotone operator.
\end{proof}

\subsection{Minimal Uscos}

In order to prove our final ``convex analysts'' proof of James' theorem, we will need to briefly consider some notions from set-valued analysis.

\medskip

A set-valued mapping $\varphi$ from a topological space $A$ into subsets of a topological space $(X,\tau)$ is {\it $\tau$-upper semicontinuous} at a point $x_0 \in A$ if for each $\tau$-open set $W$ in $X$, containing
$\varphi(x_0)$, there exists an open neighbourhood $U$ of $x_0$ such that $\varphi(U) \subseteq W$.  If $\varphi$ is $\tau$-upper semicontinuous at each point of $A$ then we say that $\varphi$ is 
{\it $\tau$-upper semicontinuous on $A$}. In the case when $\varphi$ also has nonempty compact images then we call $\varphi$ a {\it $\tau$-usco} mapping. Finally, if  $(X,\tau)$ is a linear topological space then 
we call a $\tau$-usco mapping into convex subsets of $X$  a {\it $\tau$-cusco} mapping.

\medskip

Our interest in cusco mappings is revealed in the next proposition.

\begin{proposition}[\!\!\mbox{\cite[Proposition 2.5]{PhelpsBook}}]\label{uscoprop} If $\varphi:U \to \R$ is a continuous convex function defined on a nonempty open convex subset $U$ of a normed linear space $(X,\norm)$, then the subdifferential mapping, $x \mapsto \partial \varphi(x)$,
 is a weak$^*$-cusco on $U$.
\end{proposition} 
\begin{proof}  It follows from Lemma \ref{nonempty} and Proposition \ref{SubdifferentialProperties} that we need only show that, $x \mapsto \partial \varphi(x)$, is weak$^*$-upper semicontinuous on $U$.  
So suppose, in order to obtain a contradiction, that $\partial \varphi$ is not weak$^*$ upper semicontinuous at some point $x_0 \in U$.
Then there exists a weak$^*$ open subset $W$ of $X^*$, containing $\partial \varphi (x_0)$, such that for every $0<\delta$, $\partial \varphi(B(x_0,\delta)) \not\subseteq W$. Therefore, in particular, there exist sequences $(x_n:n \in \N)$
in $U$ and $(x^*_n:n \in \N)$ in $X^*$ such that $\lim_{n \to \infty}x_n = x_0$ and $x^*_n \in \partial \varphi(x_n) \setminus W$. Furthermore, by Proposition~\ref{SubdifferentialProperties}, we can assume that the sequence $(x^*_n:n \in \N)$
is norm bounded in $X^*$. Hence, by the Banach-Alaoglu Theorem (Theorem \ref{Banach-Alaoglu}), the sequence $(x^*_n:n \in \N)$ has a weak$^*$ cluster-point $x_\infty^*$, which must lie in $X^* \setminus W$.
We will obtain our desired contradiction by showing that $x_\infty^* \in \partial \varphi(x_0) \subseteq W$.  To this end, fix $x \in U$ and $\varepsilon >0$.  Since $\varphi$ is continuous at $x_0$ there exists an $N \in \N$
such that $|\varphi(x_n)-\varphi(x_0)| < \varepsilon$ for all $n >N$.  Let $n >N$ then,
$$(\widehat{x-x_0})(x_n^*) = x_n^*(x-x_0) \leq \varphi(x) - \varphi(x_n) = [\varphi(x) - \varphi(x_0)] +[\varphi(x_0) - \varphi(x_n)] <  [\varphi(x) - \varphi(x_0)] + \varepsilon.$$
Therefore, $x_\infty^*(x-x_0) = (\widehat{x-x_0})(x_\infty^*) \leq [\varphi(x) - \varphi(x_0)] + \varepsilon$.  Since $\varepsilon >0$ was arbitrary, we have that $x_\infty^*(x-x_0) \leq \varphi(x) - \varphi(x_0)$. Since $x \in U$
was arbitrary, we have that $x_\infty^* \in \partial \varphi(x_0)$, as desired.
\end{proof}

Among the class of usco (cusco) mappings, special attention is given to the so-called minimal usco (minimal cusco) mappings.

\medskip

An usco (cusco) from a topological space $A$ into subsets of a topological space $X$ (linear topological space $X$) is said to be a {\it minimal usco} ({\it minimal cusco}) if its graph
does not contain, as a proper subset, the graph of any other usco (cusco) on $A$.

\medskip

It is not immediately obvious from this definition that there are any interesting minimal usco mappings at all, apart from single-valued continuous functions (e.g. $f:A \to X$), which are trivially minimal uscos once
one replaces $f(x)$ with $\{f(x)\}$ -  to make them set-valued mappings.  So our first task is to show that there are always many minimal uscos.

\begin{proposition}[\!\!\cite{Christensen}] \label{existence-of-minimal}
Suppose that  $(X, \tau)$ and $(Y, \tau')$ are topological spaces and $\varphi:X \to 2^Y$ is an usco on $X$. If $(Y, \tau')$ is Hausdorff then there exists a minimal usco mapping $\Psi:X \to 2^Y$ such that $\mathrm{Gr}(\Psi) \subseteq \mathrm{Gr}(\varphi)$ (i.e., every usco contains a minimal usco). 
\end{proposition}

\begin{proof} Let $\mathcal{U}$ denote the family of all usco mappings defined on
$X$ whose graphs are contained in the graph of $\varphi$. Obviously
$\mathcal{U}\neq \varnothing$ as the mapping $\varphi$ is contained in
$\mathcal{U}$. We may now partially order $\mathcal{U}$ as follows. If
$\Psi_{1}$ and $\Psi_{2}$ are members of $\mathcal{U}$, then we write
$\Psi_{1} \leq \psi_{2}$ if $\Psi_{1}(x) \subseteq \Psi_{2}(x)$ for
each $x\in X$. Next, we apply Zorn's lemma to show that $(\mathcal{U}, \le)$ possesses a minimal element. To this end, let
$\{\Psi_{\gamma}:\gamma\in \Gamma\}$ be a totally ordered subset of
$\mathcal{U}$ and let $\varphi_{M}:X\rightarrow 2^{Y}$ be defined by,
$\varphi_{M}(x):= \bigcap\{\Psi_{\gamma}(x):\gamma\in \Gamma\}$.
Since each $\Psi_{\gamma}(x)$ is nonempty and compact,
$\varphi_{M}(x)$ too is nonempty and compact. Let $W$ be an open
subset of $Y$ and consider $U:=\{x\in X:\varphi_{M}(x)\subseteq W\}$. We need to show that $U$ is open in $X$. 
We may, without loss of generality, assume that $U\neq \varnothing$ and consider $x_{0}\in U$.
By the finite intersection property, there exists some \mbox{$\gamma_{0}\in \Gamma$} such that 
$\Psi_{\gamma_{0}} (x_{0})\subseteq W$. Hence
there exists an open neighbourhood $U_{0}$ of $x_{0}$ such that
$\Psi_{\gamma_{0}} (U_{0})\subseteq W$, which means that
$\varphi_{M}(U_{0})\subseteq W$. Therefore $x_{0}\in U_{0}\subseteq U$
and so $U$ is open in $X$. From this, it follows that $\varphi_{M}\in
\mathcal{U}$ and $\varphi_{M}\leq \Psi_{\gamma}$ for each $\gamma \in
\Gamma$. Thus, by Zorn's lemma, $(\mathcal{U},\le)$ possesses a
minimal element. It is now easy to see that this element is in fact a
minimal usco. \end{proof}

A similar argument shows that every cusco contains a minimal cusco.  However, there is a much more concrete supply of minimal cuscos.

\begin{proposition}\label{subgradientis minimal} Let $\varphi:A \to 2^{X^*}$ be a weak$^*$-cusco defined on a nonempty open subset $A$ of a normed linear space $(X, \norm)$.  If the mapping $T:X \to 2^{X^*}$ defined by,
$T(x) := \varphi(x)$ if $x \in A$ and by $T(x) := \varnothing$ if $x \in X \setminus A$, is a monotone operator, then $\varphi$ is a minimal weak$^*$-cusco.
\end{proposition}
\begin{proof} Suppose, in order to obtain a contradiction, that $\varphi$ is not a minimal weak$^*$-cusco. Then there exists a weak$^*$-cusco $\Psi:A \to 2^{X^*}$ such that $\Psi(x) \subseteq \varphi(x)$
for all $x \in A$, but $\Psi(x_0) \not= \varphi(x_0)$ for some $x_0 \in A$. Choose $x_0^* \in \varphi(x_0)\setminus \Psi(x_0) = T(x_0) \setminus \Psi(x_0)$. By the Separation Theorem (Theorem \ref{separation-theorem}),
applied in $(X^*, \mathrm{weak}^*)$, there exists a $y \in X$ such that $\sup_{y^* \in \Psi(x_0)} \widehat{y}(y^*) < \widehat{y}(x_0^*)$. Let $W := \{x^* \in X^*: \widehat{y}(x^*) < \widehat{y}(x_0^*)\}$. Then $W$
is a weak$^*$-open subset of $X^*$, containing $\Psi(x_0)$.  Therefore, there exists an open neighbourhood $U \subseteq A$ of $x_0$ such that $\Psi(U) \subseteq W$.  Choose $0<t<\infty$
such that $x_0+ty \in U$. Let $y^* \in \Psi(x_0+ty) \subseteq \varphi (x_0+ty) = T(x_0+ty)$.  Since $T$ is a monotone operator,  $x^*_0 \in T(x_0)$ and  $y^* \in T(x_0+ty)$, we have that:
$$t(y^*-x_0^*)(y) =(x_0^*-y^*)(-ty)  = (x_0^*-y^*)(x_0 - (x_0+ty)) \geq 0;$$
which implies that $y^*(y) \geq x_0^*(y)$. However, this contradicts the fact that $y^* \in W$, i.e., $\widehat{y}(y^*) < \widehat{y}(x_0^*)$. Thus, $\varphi$ must be a minimal weak$^*$-cusco on $A$.
\end{proof}

\begin{corollary}\label{subdifferential-is-minimal} If $\varphi:U \to \R$ is a continuous convex function defined on a nonempty open convex subset $U$ of a normed linear space $(X,\norm)$, then the subdifferential mapping, $x \mapsto \partial \varphi(x)$,
 is a minimal weak$^*$-cusco on $U$.
\end{corollary} 
\begin{proof} By Proposition \ref{uscoprop} we have that, $x \mapsto \partial \varphi(x)$, is a weak$^*$-cusco on $U$.  So the result follows from Proposition \ref{monotone} and Proposition \ref{subgradientis minimal}.
\end{proof}

We will end our detour into set-valued analysis by giving two more results concerning uscos.  The first one shows that minimal usco behave a lot like quasi-continuous mappings, while the
last result shows how to convert an usco into a cusco.

\begin{proposition}[\!\!\cite{Christensen}]\label{quasicontinuity} Let $\varphi:A \to 2^X$ be a minimal $\tau$-usco acting from a topological space $A$ into nonempty subsets of a topological space $(X, \tau)$. Then, for every pair of 
open subsets $U$ of $A$ and $W$ of $X$ such that $\varphi(U) \cap W \not= \varnothing$, there exists a nonempty open subset $V$ of $U$ such that $\varphi(V) \subseteq W$.
\end{proposition}
\begin{proof} Let $U$ be an open subset of $A$ and let $W$ be an open subset of $X$ such that $\varphi(U) \cap W \not= \varnothing$.  We consider two cases. 

\medskip

{\bf Case(I):} If there exists a $x \in U$ such that $\varphi(x) \subseteq W$, then the result follows directly from the $\tau$-upper semicontinuity of $\varphi$.  

\medskip

{\bf Case(II):} Suppose that for each $x \in U$, $\varphi(x) \not\subseteq W$.  Let $\Psi:A \to 2^X$ be defined by,
$\Psi(x) := \varphi(x) \cap (X \setminus W)$ if $x \in U$ and by $\Psi(x) := \varphi(x)$ if $x \not\in U$.  Then, by assumption, $\Psi$ has nonempty compact images.  In fact, we claim that $\Psi$ is a $\tau$-usco
on $A$. To show this, we need only show that $\Psi$ is $\tau$-upper semicontinuous. Let $x_0 \in A$ and let $W'$ be a $\tau$-open set in $X$ containing $\Psi(x_0)$.  If $x_0 \not\in U$ then
clearly there exists an open neighbourhood $U$ of $x_0$ such that $\Psi(U) \subseteq W'$ since, in this case, $\varphi(x_0) = \Psi(x_0) \subseteq W'$ and $\Psi(x) \subseteq \varphi(x)$ for all $x \in A$.
So we are left to consider the case when $x_0 \in U$.  Suppose $x_0 \in U$.  Then $\varphi(x_0) \subseteq W' \cup  W$, since $\varphi(x_0) \cap (X\setminus W) = \Psi(x_0) \subseteq W'$. Since $\varphi$
is $\tau$-upper semicontinuous there exists an open neighbourhood $U$ of $x_0$ such that $\varphi(U) \subseteq W' \cup W$.  Therefore,
$$\Psi(U) = \varphi(U) \cap (X \setminus W) \subseteq (W' \cup W) \cap (X \setminus W) = W' \cap (X \setminus W) \subseteq W'.$$
This shows that $\Psi$ is an $\tau$-usco. Since, $\varphi$ is a minimal $\tau$-usco, we must have that  $\varphi = \Psi$, but then $\varphi(U) = \Psi(U) \subseteq (X \setminus W)$, which contradicts our
original assumption that $\varphi(U) \cap W \not= \varnothing$. Therefore, Case(II) does not occur, and so the result follows from Case(I).
\end{proof}

\begin{proposition}[\!\!\cite{Jokl, PhelpsBook}] \label{convex-usco}
Suppose that $\varphi:A \to 2^X$ is a $\tau$-usco acting from a topological space $A$
into nonempty subsets of a locally convex space $(X, +, \cdot, \tau)$. If for each $t \in A$, $\overline{\mbox{co}}^\tau \varphi(t)$ is a compact subset of $X$, 
then the mapping $\Psi:A \to 2^X$ defined by, $\Psi(t) := \overline{\mbox{co}}^\tau \varphi(t)$ for all $t \in A$, is a $\tau$-cusco on $A$.
\end{proposition} 
\begin{proof} Clearly, $\Psi$ has nonempty, compact convex images.  So it is sufficient to show that $\Psi$ is $\tau$-upper semicontinuous on $A$.  Let $x_0 \in A$
and let $W$ be a $\tau$-open subset of $X$, containing $\Psi(x_0)$.  Since vector addition is continuous, for each $x \in \Psi(x_0)$ there exist $\tau$-open convex neighbourhoods
$U_x$ of $x$ and $V_x$ of $0$ such that $x = x+0 \subseteq U_x + V_x \subseteq W$. Since linear topological spaces are also regular we can assume, by possibly making
$V_x$ smaller, that $U_x + \overline{V_x}^\tau \subseteq W$.  Now, $\{U_x:x \in \Psi(x_0)\}$ is an open cover of $\Psi(x_0)$.  Therefore, there exists a finite subcover 
$\{U_{x_k}:1 \leq k \leq n\}$ of $\{U_x:x \in \Psi(x_0)\}$.  Let $V := \bigcap_{1 \leq k \leq n} V_{x_k}$.  Then $V$ is a convex open neighbourhood of $0$ and futhermore,
$$\Psi(x_0) +\overline{V}^\tau \subseteq (\mbox{$\bigcup_{1 \leq k \leq n}$} U_{x_k}) + \overline{V}^\tau = \mbox{$\bigcup_{1 \leq k \leq n}$} (U_{x_k} + \overline{V}^\tau) =
\subseteq \mbox{$\bigcup_{1 \leq k \leq n}$} (U_{x_k}  +\overline{V_{x_k}}^\tau) \subseteq W.$$
Since $\varphi(x_0) \subseteq \Psi(x_0) +V$, which is $\tau$-open, there exists an open neighbourhood $U$ of $x_0$ such that $\varphi(U) \subseteq \Psi(x_0) +V$.  Let $x \in U$.  Then
$$\Psi(x) = \overline{\mbox{co}}^\tau \varphi(x) \subseteq \overline{\mbox{co}}^\tau (\Psi(x_0) +V) \subseteq \Psi(x_0) + \overline{V}^\tau \subseteq W \mbox{\quad since, $\Psi(x_0) + \overline{V}^\tau$ is closed and convex.}$$
Here we used the fact that the sum of a closed set with a compact set is closed.   \end{proof}

\subsection{A generalisation of James' Theorem} 

By making an obvious modification to Corollary \ref{SubsequenceUsefulJames}, we obtain the following lemma. 

\begin{lemma}\label{needed lemma}
Let $\varphi:A \to \R$ be a $\tau$-continuous convex function defined on a nonempty convex subset $A$ of a locally convex space $(X, \tau)$ and
let $\tau'$ is a Hausdorff locally convex topology on $X$ such that (i) $\tau' \subseteq \tau$ and (ii) $\varphi$ is $\tau'$-lower semicontinuous. 
If $T$ is a nonempty $\tau'$-closed and convex subset of $X$ and $S$ is any $\tau$-separable subset of $A$ such that $S-T\seq A$
then, for every sequence $\widetilde{x}:=(x_n:n\in\n)$ in $T$, there exists a subsequence, $\widetilde{x}|_J$, of $\widetilde{x}$ such that $\varphi(y-aK_{\tau'}(\widetilde{x}|_J))$
is at most a singleton for all $y \in S$ and all $a\in [0,1]$.
\end{lemma}

The following lemma shows us that in Theorem \ref{JamesUsco} we get the weak$^*$ lower semicontinuity of $\varphi$ for free.

\begin{lemma}\label{LSC}
	Let $(X,\norm)$ be a Banach space and let $A$ be a nonempty, open, convex subset of $X^*$. 
	If $\varphi:A\rightarrow\r$ is a continuous, convex function and $\partial\varphi(x^*)\cap\widehat{X}\neq\varnothing$ for all $x^*\in A$, then $\varphi$ is weak$^*$-lower-semicontinuous on $A$.
\end{lemma}

\begin{proof}
	Let $x^*_0\in A$ and let $\epsilon>0$. Then, there exists an $\widehat{x}\in\partial\varphi(x_0^*)\cap\widehat{X}$. Define $h:A\rightarrow\r$ to be $$h(x^*):=\widehat{x}(x^*)-\widehat{x}(x^*_0)+\varphi(x^*_0) \ \ \text{for all} \ x^*\in A.$$
	Then observe that, since $\widehat{x}\in\partial\varphi(x_0^*)$, we have $h(x^*)\leq \varphi(x^*)$ for all $x^*\in A$.
	Now the set
	$$U:=\{x^*\in A:|\widehat{x}(x^*-x^*_0)|<\epsilon\},$$
	is a weak$^*$-open neighbourhood of $x^*_0$, and for  all $x^*\in U$, we have that
	$$\varphi(x^*_0)-\epsilon<\widehat{x}(x^*)-\widehat{x}(x^*_0)+\varphi(x^*_0)=h(x^*)\leq\varphi(x^*).$$
	Therefore, $\varphi$ is weak$^*$-lower-semicontinuous at $x^*_0$. Since $x^*_0$ was arbitrary, we conclude that $\varphi$ is weak$^*$-lower-semicontinuous on $A$.
\end{proof}

At last, we can present our ``convex analysts'' version of James' weak compactness theorem.

\begin{theorem}\label{JamesUsco}
	Let $(X,\norm)$ be a Banach space and let $A$ be a nonempty, open, convex subset of $X^*$. 
	If $\varphi:A\rightarrow\r$ is a continuous, convex function and $\partial\varphi(x^*)\cap\widehat{X}\neq\varnothing$ for all $x^*\in A$, then $\partial\varphi(x^*)\seq\widehat{X}$ for all $x^*\in A$.
\end{theorem}

\begin{proof}
	Let $x_0^*\in A$. Without loss of generality, we may assume that $x_0^*=0$. Indeed, if not, we consider the function $\psi:(A-x_0^*)\rightarrow\r$ given by $\psi(x^*):=\varphi(x^*+x_0^*)$. Note that $\psi$ is continuous and convex and that $\partial\varphi(x^*+x^*_0)=\partial\psi(x^*)$ for all $x^*\in A$. In particular, $\partial\psi(x^*)\cap\widehat{X}\neq\varnothing$ for all $x^*\in (A-x^*_0)$ and $\partial\varphi(x^*_0)=\partial\psi(0)$. So, if $x_0^*\neq 0$, we can simply translate $\varphi$ and use the argument at 0. \\ \\
	Since $A$ is open and since $x^*\mapsto\partial\varphi(x^*)$ is locally bounded (Proposition \ref{SubdifferentialProperties}), there exist $m,L>0$ such that $m\ball\seq A$ and $\|x^{**}\|\leq L$ for all $x^{**}\in\partial\varphi(B(0,m))$.
	Let $(\beta_n:n\in\n)$ be a sequence of strictly positive numbers such that $\sum_{n=1}^\infty\beta_n<m/2$ and $\lim_{n\rightarrow\infty}\frac{1}{\beta_n}\sum_{i=n+1}^\infty\beta_i=0$.  \\ \\
	Since $x^*\mapsto \partial\varphi(x^*)$ is a minimal weak$^*$-cusco, (see, Corollary \ref{subdifferential-is-minimal})
	we know that there exists a minimal weak$^*$-usco, $M:A\rightarrow 2^{X^{**}}$, such that $M(x^*)\seq\partial\varphi(x^*)$ for all $x^*\in A$, by Proposition \ref{existence-of-minimal}. In fact, by Proposition \ref{convex-usco},
	we know that $\partial\varphi(x^*)=\overline{\text{co}}^{w^*}[M(x^*)]$ for all $x^*\in A$. \\ \\
	Therefore, to show that $\partial\varphi(0)\seq\widehat{X}$, it suffices to show that $M(0)\seq\widehat{X}$. 
	This is because if $M(0)\seq \widehat{X}$, then $M(0)$ is weakly compact (see Remark \ref{reflexive}) and then, by the Krein-Smulian Theorem (Corollary \ref{Closed convex hull}), $\overline{\text{co}}[M(0)]$ is also weakly compact.
	Since the weak$^*$ topology is weaker than the weak topology, $\overline{\text{co}}[M(0)]$ is clearly weak$^*$-compact and hence weak$^*$-closed. Therefore,
	$$\partial\varphi(0)=\overline{\text{co}}^{w^*}[M(0)]=\overline{\text{co}}[M(0)]\seq\widehat{X}.$$
	So suppose, for a contradiction, that $M(0)\not\seq\widehat{X}$. Then there exists an $F\in M(0)\setminus\widehat{X}$.
Since $\widehat{X}$ is a closed subspace of $X^{**}$, this means there must exist an $0<\epsilon<\text{dist}(F,\widehat{X})$. \\ \\
\underline{\textbf{Part I:}} Let $f_0:=0$. We inductively create sequences $(f_n:n\in\n)$ in $S_{X^*}$, $(v_n:n\in\n)$ in $B(0,m)$, and $(\widehat{x}_n:n\in\n)$ in $\widehat{X}$, such that the statements 
	\begin{itemize}
	    \item $(A_n):-$ $\|v_n\|<m/n$ and $\widehat{x}_n\in\partial\varphi(v_n)$.
		\item $(B_n): -$ $(F-\widehat{x}_n)(f_j)|\leq\epsilon/2$ for all $0\leq j<n$.
		\item $(C_n):-$ $F(f_n)>\epsilon$ and $\widehat{x}_j(f_n)=0$ for all $1\leq j\leq n$. 
	\end{itemize} 
	are true for all $n\in\n$. For the first step, choose any $v_1\in B(0,m)\seq A$. 
	Then $\partial\varphi(v_1)\cap\widehat{X}\neq\varnothing$ and so we may choose $\widehat{x}_1\in \partial\varphi(v_1)\cap\widehat{X}$ which clearly satisfies $|(F-\widehat{x}_1)(f_0)|=0\leq\epsilon/2$. 
Now note that 
$$\text{dist}(F,\text{span}\{\widehat{x}_1\})\geq \text{dist}(F,\widehat{X})>\epsilon,$$
	and so, by Lemma \ref{Separation}, there exists $f_1\in S_X$ such that $F(f_1)>\epsilon$ and $\widehat{x}_1(f_1)=0$. So the statements $(A_1)$, $(B_1)$ and $(C_1)$ hold. \\ \\
	Now fix $k\in\n$. Suppose that we have created $\{v_1,\dots, v_k\}$, $\{\widehat{x}_1,\dots,\widehat{x}_k\}$ and $\{f_1,\dots, f_k\}$ such that the statements $(A_k)$, $(B_k)$ and $(C_k)$ hold true. Then consider the set
$$W:=\mbox{$\bigcap_{j=0}^{k}$}\lbrace G\in X^{**}:|(F-G)(f_j)|<\epsilon/2\rbrace.$$
 Note that $F\in M(0)\cap W$ and so $M(B(0,\frac{m}{k+1}))\cap W\neq\varnothing$. Therefore, by the minimality of $M$ and Proposition~\ref{quasicontinuity}, there exists a nonempty open set $V\seq B(0,\frac{m}{n+1})$ such that $M(V)\seq W$. \\ \\
 Choose $v_{k+1}\in V$. Then $\|v_{k+1}\|<m/(k+1)$. 
 By hypothesis, since $v_{k+1}\in A$, we have that $\partial\varphi(v_{k+1})\cap\widehat{X}\neq\varnothing$, and so we may choose $\widehat{x}_{k+1}\in \widehat{X}$ such that  
 $$\widehat{x}_{k+1}\in\partial\varphi(v_{k+1})=\overline{\text{co}}^{w^*}[M(v_{k+1})]\seq \overline{W}^{w^*}\!\seq \lbrace G\in X^{**}:|(F-G)(f_j)|\leq\epsilon/2\rbrace.$$
 Thus the statements $(A_{k+1})$ and $(B_{k+1})$ hold. Finally, observe that 
	$$\text{dist}(F,\text{span}\{\widehat{x}_1,\dots,\widehat{x}_{k+1}\})\geq \text{dist}(F,\widehat{X})>\epsilon,$$
	and so, by Lemma \ref{Separation}, there exists $f_{k+1}\in S_X$ such that $F(f_{k+1})>\epsilon$ and $\widehat{x}_j(f_{k+1})=0$ for all $1\leq j\leq k+1$.
	Therefore the statement $(C_{k+1})$ also holds. This completes the induction. \\ \\
	\underline{\textbf{Part II:}} Now let $(n_k:k\in\n)$ be a strictly increasing sequence of natural numbers. Then for all $k\in\n$, define $v'_k:=v_{n_k}$ and $x'_k:=x_{n_k}$ and $f'_k:=f_{n_k}$. 
	Also define $f'_0:=0$.
	Then the sequences $(v'_n:n\in\n)$, $(\widehat{x}'_n:n\in\n)$ and $(f'_n:n\in\n)$ still satisfy $(A_n)$, $(B_n)$ and $(C_n)$ for all $n\in\n$. Therefore, the properties $(A_n)$, $(B_n)$ and $(C_n)$ are stable   
	under passing to subsequences.\\ \\
	Now, since $\partial\varphi(x^*)\cap\widehat{X}\neq\varnothing$ for all $x^*\in A$, we have that $\varphi$ is weak$^*$-lower-semicontinuous on $A$, by Lemma~\ref{LSC}.
	Let $S := \frac{m}{2}\ball\cap\text{span}\{f_n:n\in\n\}$ and $T := \frac{m}{2}\ball$ and note that $S-T \subseteq mB_X \subseteq A$.
	Then by passing to a subsequence and relabelling if necessary, we may assume that the set $\varphi(f-a[(m/2)K_{w^*}(f_n:n\in\n)])$ is a singleton for all $0\leq a\leq 1$ and all $f\in S$, by 
	Lemma \ref{needed lemma}.  
Since $(f_n:n\in\n)$ is a sequence in $\ball$ (which is weak$^*$-compact), it must a have a weak$^*$-cluster point, call it $f_\infty$.\\ \\
	\underline{\textbf{Part III:}}  As in Part III of the proof of Theorem \ref{JamesSequential}, we can derive that $\widehat{x}_n(f_k-f_\infty)>\epsilon/2$ for all $n>k$ from the statements $(B_n)$ and $(C_n)$. 
	We also note that, from the statement $(A_n)$, we have $v_n\rightarrow 0$ in norm. 
	Therefore, since $\varphi$ is norm-continuous, there exists $N_0\in\n$ such that $$|\varphi(v_n)-\varphi(0)|<\beta_1\epsilon/8, \quad \text{for all} \ n>N_0. $$ 
	Lastly observe that for all $n\in\n$, $v_n\in B(0,m)$ and $\widehat{x}_n\in\partial\varphi(v_n)$ and thus $\|\widehat{x}_n\|\leq L$ by the local-boundedness of $\partial\varphi$. Therefore, if $n>\frac{8Lm}{\beta_1\epsilon}$, we have that
	$$|\widehat{x}_n(v_n)|\leq\|\widehat{x}_n\|\|v_n\|\leq \frac{Lm}{n}<\frac{\beta_1\epsilon}{8}.$$  
    \underline{\textbf{Part IV:}} For each $n\in\n$, let $K_n:=\text{co}\{f_k:k\geq n\}-f_\infty$ and note that $(K_n:n\in\n)$ is a decreasing sequence of nonempty, convex subsets of $X^*$. Set $r:=\epsilon/8$. 
    Then we have that
    $$\beta_1r+\varphi(0)<\inf_{f\in K_1}\varphi(\beta_1f).$$ 
	 To see this, let $f\in K_1$. Then $f=\sum_{i=1}^k\lambda_{i}f_{n_i}-f_\infty$ where $\lambda_i\geq 0$ for all $1\leq i\leq k$ and  $\sum_{i=1}^k\lambda_{i}=1$.  
	 Set $N>\max\lbrace n_1,\dots, n_k, N_0, \frac{8Lm}{\beta_1\epsilon}\rbrace$. Then we have,
  	\begin{align*}
	\varphi(\beta_1f)\!-\!\varphi(0)&=[\varphi(\beta_1f)-\varphi(v_N)]+[\varphi(v_N)-\varphi(0)] \\
	&\geq \widehat{x}_{N}(\beta_1f)-\widehat{x}_{N}(v_N)+[\varphi(v_N)-\varphi(0)] \ \ \mbox{(since $\widehat{x}_N\! \in\! \partial\varphi(v_N)$)} \\
	&>\beta_1\widehat{x}_N(f)-\widehat{x}_N(v_N)-\beta_1\epsilon/8 \\
	&>\beta_1\widehat{x}_N(f)-\beta_1\epsilon/8-\beta_1\epsilon/8 \\
	&=\beta_1( \widehat{x}_{N}\mbox{$(\sum_{i=1}^k\lambda_{i}f_{n_i}-f_\infty))$}-\beta_1\epsilon/4 \\
	&=\beta_1\mbox{$\sum_{i=1}^k\lambda_{i}\widehat{x}_N(f_{n_i}-f_\infty)$}-\beta_1\epsilon/4 >\beta_1\epsilon/4.
	\end{align*}
	 Therefore, since $f\in K_1$ was arbitrary, we have that $\beta_1r+\varphi(0)<\inf_{f\in K_1}\varphi(\beta_1f)$ as claimed. 
	 So, by Lemma \ref{increasing}, there exists a sequence $(g_n:n\in\n)$ such that for all $n\in\n$:
	\begin{enumerate}[label=(\roman*)]
	\item $g_n\in \text{co}\lbrace f_k:k\geq n\rbrace$ and
	\item $\displaystyle \varphi($\mbox{$\sum_{i=1}^n\beta_i(g_i-f_\infty)$}$)+\beta_{n+1}r<\varphi($\mbox{$\sum_{i=1}^{n+1}\beta_i(g_i-f_\infty)).$} \qquad \mbox{$(*)$}  
\end{enumerate}
\underline{\textbf{Part V:}} Since $(g_n:n\in\n)$ is a sequence in $\ball$, it must have a weak$^*$-cluster point.
So, let $g_\infty$ be a weak$^*$-cluster point of $(g_n:n\in\n)$. Then, by Proposition \ref{ClusterConvex}, we have that $g_\infty\in K_{w^*}(f_n:n\in\n)$. Since for all $n\in\n$, we have that
 $\sum_{i=1}^n\beta_ig_i\in  S = \frac{m}{2}\ball\cap\text{span}\{f_n:n\in\n\}$, and $0\leq \sum_{i=1}^n\beta_i\leq m/2$, then
$$
    \varphi(\mbox{$\sum_{i=1}^n$}\beta_i(g_i-g_\infty))=\varphi(\mbox{$\sum_{i=1}^{n}$}\beta_ig_i-\mbox{$\sum_{i=1}^{n}$}\beta_i\cdot g_\infty)
  =\varphi(\mbox{$\sum_{i=1}^{n}$}\beta_ig_i-\mbox{$\sum_{i=1}^{n}$}\beta_i\cdot f_\infty) 
  =\varphi(\mbox{$\sum_{i=1}^n$}\beta_i(g_i-f_\infty)), \qquad\mbox{$(**)$}
$$
  by the observation made in Part II. As in Part V of the proof of Theorem \ref{JamesSequential}, we set $g:=\sum_{i=1}^\infty\beta_i(g_i-g_\infty)$ and deduce that $g\in X^*.$ \\ \\
\underline{\textbf{Part VI:}} 	Lastly note that $\|g\|\leq 2\sum_{i=1}^\infty\beta_i\leq m$
 and so $g\in m\ball\seq A$. 
 Therefore, in order to contradict our original assumption, and thus complete the proof, it suffices to show that $\partial\varphi(g)\cap\widehat{X}=\varnothing$.  So, suppose that there exists $\widehat{x}\in\widehat{X}$ such that $\widehat{x}\in\partial\varphi(g)$. Then, if $n>1$, 
\begin{align*}
	\beta_n r&<\varphi(\mbox{$\sum_{i=1}^n$}\beta_i (g_i-f_\infty))-\varphi(\mbox{$\sum_{i=1}^{n-1}$}\beta_i (g_i-f_\infty)) \qquad \mbox{(by $(*)$)}\\
	&=\varphi(\mbox{$\sum_{i=1}^n$}\beta_i (g_i-g_\infty))-\varphi(\mbox{$\sum_{i=1}^{n-1}$}\beta_i (g_i-g_\infty)) \qquad \mbox{(by $(**)$)} \\
	&\leq \varphi(g)-\varphi(\mbox{$\sum_{i=1}^{n-1}$}\beta_i (g_i-g_\infty)) \qquad\mbox{(since $\displaystyle\varphi(g)=\mbox{$\sup_{n\in\n}$}\varphi(\mbox{$\sum_{i=1}^n$}\beta_i(g_i-g_\infty))$} \\
	&\leq \widehat{x}(g)-\widehat{x}(\mbox{$\sum_{i=1}^{n-1}$}\beta_i (g_i-g_\infty)) \qquad \mbox{(since $\widehat{x}\in \partial\varphi(g))$} \\
	&=\widehat{x}(\mbox{$\sum_{i=n}^{\infty}$}\beta_i (g_i-g_\infty))  =\beta_n\widehat{x}(g_n-g_\infty)+\mbox{$\sum_{i=n+1}^{\infty}$}\beta_i \widehat{x}(g_i-g_\infty).
\end{align*}	
Rearranging gives us that 
$$r<\widehat{x}(g_n-g_\infty)+\frac{1}{\beta_n}\mbox{$\sum_{i=n+1}^{\infty}$}\beta_i \widehat{x}(g_i-g_\infty)\leq\widehat{x}(g_n-g_\infty)+\frac{2\|\widehat{x}\|}{\beta_n}\mbox{$\sum_{i=n+1}^{\infty}$}\beta_i.$$
Taking the limit as $n\rightarrow\infty$ we get that 
$$r\leq\liminf_{n\rightarrow\infty}g_n(x)-g_\infty(x)+2\|\widehat{x}\|\left(\lim_{n\rightarrow\infty}\frac{1}{\beta_n}\mbox{$\sum_{i=n+1}^{\infty}$}\beta_i\right)=\liminf_{n\rightarrow\infty}g_n(x)-g_\infty(x),$$
which contradicts the inequality $\displaystyle\liminf_{n\rightarrow\infty}g_n(x)\leq g_\infty(x)$. Thus, $\partial\varphi(g)\cap\widehat{X}=\varnothing$, which contradicts our original assumption concerning
the function $\varphi$.  This completes the proof. \end{proof}

\begin{remark} To see that Theorem \ref{JamesUsco} is indeed a generalisation of Theorem \ref{JamesFull} consider the following.  Suppose that $C$ is a nonempty closed and bounded
convex subset of a Banach space $(X,\norm)$ with $0 \in C$.  Define $p:X^* \to \R$ by, $p(x^*) := \sup_{c \in C} x^*(c)$ for all $x^* \in X^*$. 
Then  $\widehat{C} \subseteq \partial p(0)$.
If every $x^*\in X^*$ attains its supremum over $C$ then $\partial p(x^*) \cap \widehat{X} \not= \varnothing$ for every $x^* \in X^*$.  This last fact follows because,
if $x^* \in X^*\setminus\{0\}$, $c \in C$ and $p(x^*) = x^*(c)$ then $\widehat{c} \in  \partial p(x^*)$.  Thus, by Theorem \ref{JamesUsco}, 
$$\overline{\widehat{C}}^{w^*} \subseteq \partial p(0) \subseteq \widehat{X} \mbox{\quad since, $\partial p(0)$ is weak$^*$-closed.}$$
Hence, $C$ is weakly compact by Remark \ref{reflex-remark}.  Let us also note that an earlier version of Theorem \ref{JamesUsco} appeared in \cite[Theorem 2.2]{Warren1996}.
\end{remark}


\section{Variational Principle}


\noindent The corner stone of this section is the Br\o ndsted-Rockafellar Theorem which gives the existence of subgradients for lower semicontinuous convex functions defined on Banach spaces.
The key notion behind this theorem is the notion of an ``$\varepsilon$-subgradient''. Suppose that $f:X \to \R\cup\{\infty\}$ is a convex proper lower semicontinuous function on a normed linear
 space $(X, \norm)$ and $x \in \mbox{Dom}(f)$.  Then, for any $\varepsilon >0$, we define the {\it $\varepsilon$-subdifferential $\partial_\varepsilon f(x)$} by, 
$$\partial_\varepsilon f(x) := \{x^* \in X^*: x^*(y) - x^*(x) \leq f(y) - f(x) + \varepsilon \mbox{ for all $y \in \mbox{Dom}(f)$}\}.$$

\begin{theorem}[\!\!\cite{BrondstedRockafellar}]\label{B.R} Suppose that $f:X \to \R\cup\{\infty\}$ is a convex proper lower semicontinuous function on a Banach space $(X, \norm)$. Then, given any point $x_0 \in \mbox{Dom}(f)$, $\varepsilon >0$ and any $x_0^* \in \partial_\varepsilon f(x_0)$, there exists $x \in \mbox{Dom}(f)$ and $x^* \in X^*$ such that $x^* \in \partial f(x)$, $\|x-x_0\| \leq \sqrt{\varepsilon}$ and $\|x^*-x_0^*\| \leq \sqrt{\varepsilon}$. 
\end{theorem}

We will use the Br\o ndsted-Rockafellar Theorem (Theorem \ref{B.R}) to show that certain functions attain their maximum value in a rather strong way, that we now make precise.  We shall say that a function
$f:X\to [-\infty, \infty)$ defined on a normed linear space $(X,\norm)$ attains a (or has a)  {\it strong maximum at $x_0 \in X$} if, $f(x_0) = \sup_{x \in X} f(x)$ and $\lim_{n \to \infty} x_n =x_0$
whenever $(x_n:n \in \N)$ is a sequence in $X$ such that $\lim_{n \to \infty} f(x_n) =  \sup_{x \in X} f(x) = f(x_0)$.

\medskip

In addition to the  Br\o ndsted-Rockafellar Theorem and the definition of a strong maximum, we shall require one more definition.   
Let $\varphi:X \to 2^Y$ be a set-valued mapping acting between a
topological space $(X, \tau)$ and a normed linear space $(Y,\norm)$. Then we say that $\varphi$ is {\it single-valued and norm upper semicontinuous at 
$x_0 \in X$} if, $\varphi(x_0) =: \{y_0\}$ is a singleton
subset of $Y$ and for each $\varepsilon >0$ there exists an open neighbourhood $U$ of $x_0$ such that $\varphi(U) \subseteq B[y_0, \varepsilon]$.

\medskip

We shall now combine the Br\o ndsted-Rockafellar Theorem with these definitions in order to obtain the following preliminary result.

\medskip

\begin{proposition}\label{prelim} Suppose that $f:X \to \R\cup\{\infty\}$ is a proper function on a Banach space $(X, \norm)$ and suppose that $f^*:X^* \to \R\cup\{\infty\}$ 
(the Fenchel conjugate of $f$)  is defined by, 
$$f^*(x^*) := \sup_{x \in X} [x^*(x) - f(x)] \  = \sup_{x \in \mathrm{Dom}(f)}  [x^*(x) - f(x)].$$  
Then, 
\begin{enumerate}
\item[{\rm (i)}] $f^*$ is a convex and weak$^*$ lower semicontinuous function on $\mathrm{Dom}(f^*)$;
\item[{\rm (ii)}] $f^*$ is continuous on $\mathrm{int}(\mathrm{Dom}(f^*))$;
\item[{\rm (iii)}] if $x^* \in \mathrm{Dom}(f^*)$ and  $x \in \mathrm{argmax} (x^*-f)$ then $\widehat{x} \in \partial f^*(x^*)$;
\item[{\rm(iv)}] if $\varepsilon >0$, $x^* \in \mathrm{Dom}(f^*)$, $x \in X$ and $f^*(x^*) - \varepsilon < x^*(x) - f(x)$ then $\widehat{x} \in \partial_\varepsilon f^*(x^*)$;
\item[{\rm (v)}] if $x_0^* \in  \mathrm{int}(\mathrm{Dom}(f^*))$, $x \in \mathrm{argmax}(x_0^*-f)$ and $x^* \mapsto \partial f^*(x^*)$ is single-valued and norm upper semicontinuous at
$x_0^*$ then $x_0^* - f$ has a strong maximum at $x$.
\end{enumerate}
\end{proposition}

\begin{proof} For those people familiar with the Fenchel conjugate, they may want to skip the proofs of (i)-(iv).
\begin{enumerate}
\item[{\rm (i)}] For each $x \in \mathrm{Dom}(f)$ define $g_x:X^* \to \R$ by, $g_x(x^*) := \widehat{x}(x^*) - f(x)$. Then each function $g_x$ is weak$^*$ continuous and affine.
Now for each $x^* \in X^*$, $f^*(x^*) = \sup_{x \in \mathrm{Dom}(f)} g_x(x^*)$.
Thus, as the pointwise supremum of a family of weak$^*$ continuous affine mappings, the Fenchel conjugate of $f$, is itself convex and weak$^*$ lower semicontinuous.
[Recall the general fact that the pointwise supremum of a family of convex functions is convex and the pointwise supremum of a family of lower semicontinuous mappings is again lower semicontinuous].
\item[{\rm (ii)}] Since this statement is vacuously true when  $\mathrm{int}(\mathrm{Dom}(f^*)) = \varnothing$, we will assume that $\mathrm{int}(\mathrm{Dom}(f^*))$ is nonempty. Let us first recall that by Proposition \ref{Lipschitz}, and by the fact that $f^*$ is convex and the fact that $\mathrm{int}(\mathrm{Dom}(f^*))$ is also convex, it is sufficient to show that $f^*$ is locally 
bounded above on $\mathrm{int}(\mathrm{Dom}(f^*))$. In fact, as we shall now show, it is sufficient to show that $f^*$ is locally bounded above at a single point $x_0^*\in \mathrm{int}(\mathrm{Dom}(f^*))$. To this end, suppose that $f^*$ is locally bounded above at $x_0^* \in \mathrm{int}(\mathrm{Dom}(f^*))$. Then there exist an $0<M$ and a $0< \delta$ such that 
$f^*(y^*) \leq M$ for all $y^* \in B[x_0^*, \delta]$. 
Let $x^*$ be any point in $\mathrm{int}(\mathrm{Dom}(f^*))$. Since $\mathrm{int}(\mathrm{Dom}(f^*))$ is an open convex set, there exists a point $y^* \in \mathrm{int}(\mathrm{Dom}(f^*))$ and a 
$0 < \lambda <1$ such that $x^* = \lambda y^* + (1 - \lambda) x_0^*$. Let $M^* := \max\{M, f^*(y^*)\}$ and note that
$$x^* \in B[x^*,(1-\lambda)\delta] = \lambda y^* + (1-\lambda)B[x_0^*,\delta] \subseteq \mathrm{int}(\mathrm{Dom}(f^*)), \mbox{\quad since $\mathrm{int}(\mathrm{Dom}(f^*))$ is convex.}$$
We claim that $f^*$ is bounded above by $M^*$ on $B[x^*,(1-\lambda)\delta]$. To see this, let $z^*$ be any element of $B[x^*,(1-\lambda)\delta]$. Then $z^* = \lambda y^* + (1-\lambda)w^*$
for some $w^* \in B[x_0^*,\delta]$ since, 
$$B[x^*,(1-\lambda)\delta] = x^*+ (1-\lambda)B[0,\delta] =  \lambda y^* +(1-\lambda)x_0^*+ (1-\lambda)B[0,\delta] =      \lambda y^* + (1-\lambda)B[x_0^*,\delta].$$
 Therefore,
$$f^*(z^*) = f^*( \lambda y^* + (1-\lambda)w^*) \leq \lambda f^*(y^*) + (1-\lambda) f^*(w^*) \leq \lambda M^* + (1- \lambda) M^* = M^*.$$
Next, we will use that fact that since $\mathrm{int}(\mathrm{Dom}(f^*))$ is a nonempty open subset of a complete metric space, it is itself a Baire space with the relative topology.
Now, for each $n \in \N$, let 
$$F_n := \{x^* \in \mathrm{int}(\mathrm{Dom}(f^*)): f^*(x^*) \leq n\}.$$ 
Since $f^*$ is weak$^*$ lower semicontinuous, it is lower semicontinuous with respect
to the norm topology too.  Therefore, each set $F_n$ is closed with respect to the relative norm topology on $\mathrm{int}(\mathrm{Dom}(f^*))$. Since $\mathrm{int}(\mathrm{Dom}(f^*))
= \bigcup_{n \in \N} F_n$, there exists an $n_0 \in \N$ such that $\mathrm{int}(F_{n_0}) \not= \varnothing$. Hence, $f^*$ is locally bounded above at each point of  $\mathrm{int}(F_{n_0})$.
This completes the proof of part (ii).
\item[{\rm (iii)}] Let $y^*$ be any element of  $\mathrm{Dom}(f^*)$. Then,
$$\widehat{x}(y^*) - \widehat{x}(x^*) = y^*(x) - x^*(x) = [y^*(x) - f(x)] - [x^*(x) - f(x)] =  [y^*(x) - f(x)] - f^*(x^*) \leq f^*(y^*) - f^*(x^*).$$
Therefore, $\widehat{x} \in \partial f^*(x^*)$.
\item[{\rm(iv)}]  Let $y^*$ be any element of  $\mathrm{Dom}(f^*)$. Then,
\begin{eqnarray*}
\widehat{x}(y^*) - \widehat{x}(x^*) &=& y^*(x) - x^*(x) = [y^*(x) - f(x)] - [x^*(x) - f(x)]  \\
&\leq &  [y^*(x) - f(x)] - [f^*(x^*) - \varepsilon] \leq f^*(y^*) - f^*(x^*) + \varepsilon.
\end{eqnarray*}
Therefore, $\widehat{x} \in \partial_\varepsilon f^*(x^*)$.
\item[{\rm (v)}] Let $(x_n:n \in \N)$ be a sequence in $X$ such that 
$$\lim_{n \to \infty} (x_0^*-f)(x_n) = \sup_{x \in X} (x_0^* -  f)(x) = f^*(x_0^*).$$
We will show that $(x_n:n \in \N)$ converges to $x$.
Let $\varepsilon >0$.  By (iii) and the assumption that $\partial f^*(x_0^*)$ is a singleton we have that $\partial f^*(x_0^*) = \{\widehat{x}\}$.  Since, $x^* \mapsto \partial f^*(x^*)$,
is norm upper semicontinuous at $x_0^*$ there exists a $0 < \delta < \varepsilon$ such that  if $\|x^* - x_0^*\| \leq \delta$ then $\|F-\widehat{x}\| < \varepsilon$ for all
$F \in \partial f^*(x^*)$.  Choose $N \in \N$ such that $(x_0^* - f)(x_n) > f^*(x_0) - \delta^2$ for all $n >N$.  Then, by (iv), $\widehat{x_n} \in \partial_{\delta^2} f^*(x_0^*)$ for
all $n >N$.  Let $n >N$. Then, by the  Br\o ndsted-Rockafellar Theorem, there exist $x_n^* \in \mathrm{Dom}(f^*)$ and $F_n \in X^{**}$ such that $F_n \in \partial f^*(x_n^*)$,
$\|x_n^*- x_0^*\| \leq \delta$ and $\|F_n - \widehat{x_n}\| \leq \delta < \varepsilon$.  Therefore, 
$$\|x_n- x\| = \|\widehat{x_n} - \widehat{x}\| 
\leq  \|\widehat{x_n} - F_n\| + \|F_n - \widehat{x}\| 
\leq  \varepsilon + \varepsilon  = 2\varepsilon.$$
\end{enumerate}
This completes the proof. \end{proof}

Our first variational principle applies to dual differentiation spaces, \cite{Warren4Authors}.  Recall that a Banach space $(X,\norm)$ is called a {\it dual differentiability space (or DD-space for short)}
if every continuous convex function $\varphi:A \to \R$ defined on a nonempty open convex subset $A$ of $X^*$ such that 
$\{x^* \in A:\partial \varphi(x^*) \cap \widehat{X} \not= \varnothing\}$ contains a dense and $G_\delta$ subset of $A$, has the property that its subdifferential mapping $\partial \varphi :A \to 2^{X^{**}}$ is single-valued and norm upper semicontinuous at each point of a dense and $G_\delta$ subset of $A$ (or equivalently, $\varphi$ is Fr\'echet differentiable at the points of a dense and $G_\delta$ subset of $A$, \cite[Proposition 2.8]{PhelpsBook}).

\begin{theorem}\label{DD theorem} Let $f:X \to \R\cup\{\infty\}$ be a proper function on a dual differentiation space $(X, \norm)$.  If there exists a nonempty open subset $A$ of $\mathrm{Dom}(f^*)$
and a dense and $G_\delta$ subset $R$ of $A$ such that $\mathrm{argmax}(x^*-f) \not= \varnothing$ for each $x^* \in R$, then there exists a dense and $G_\delta$ subset $R'$ of $A$ such that
$(x^*-f):X \to \R \cup\{-\infty\}$ has a strong maximum for each $x^* \in R'$. In addition, if $0 \in A$ and $\varepsilon >0$ then there exists an 
$x_0^* \in X^*$ with $\|x^*_0\| < \varepsilon$ such that $(x_0^*-f):X \to \R \cup\{-\infty\}$ has a strong maximum.
\end{theorem}

\begin{proof} Consider $\partial f^*:A \to 2^{X^{**}}$. Then by Proposition \ref{prelim} part (iii) 
$$R_1 := \{x^* \in A: \partial f^*(x^*) \cap \widehat{X} \not= \varnothing\}$$
contains a dense and $G_\delta$ subset of $A$.  Since $X$ is a dual differentiation space,  
$$R_2 := \{x^* \in A: \partial f^* \mbox{ is single-valued and norm upper semicontinuous at $x^*$}\}$$
 contains a dense and $G_\delta$ subset of $A$. Let $R' := R_1 \cap R_2$. Then $R'$ contains a dense and $G_\delta$ subset of 
$A$ and by Proposition \ref{prelim} part (v), $(x^* -f)$ has a strong maximum for each $x^* \in R'$.
\end{proof}

\begin{remark}\label{DD-space} There are two main weaknesses with this theorem: (i) although it is known that many Banach spaces (e.g.~all spaces with the Radon-Nikod\' ym property, \cite{Warren4Authors}, all weakly Lindel\"of spaces, \cite{WarrenLindelof}, all spaces that admit an equivalent locally uniformly rotund norm \cite{GilesMoors} and all spaces whose dual space $X^*$ is weak Asplund, \cite{Warren4Authors}) are dual differentiation spaces, it is still
an open question as to whether every Banach space is a dual differentiation space; (ii) it is not clear how one would go about showing that there exists a ``large'' subset $R$ of
$\mathrm{int}(\mathrm{Dom}(f^*))$ with the property that $\mathrm{argmax}(x^*-f) \not= \varnothing$ for each $x^* \in R$.
\end{remark}

For our next result, and main variational principle, we address concern (i) of Remark \ref{DD-space} by giving a variational principle that holds in all Banach spaces.  Unfortunately, there is a cost for this level of generality. Namely, we need to impose a strong assumption upon the mapping $x^* \mapsto \mathrm{argmax}(x^*- f)$.  We also need to employ the following non-trivial result concerning minimal weak cuscos, 
which was first proved by J.~Christensen in \cite{Christensen},  using topological games (in the domain space), and later rephrased in \cite{Warren4Authors}.

\begin{theorem}\label{Warren4Authors} A minimal weak$^*$ cusco $\varphi: A \to 2^{X^{**}}$ from a complete metric space $A$ into subsets of the second dual $X^{**}$ of a Banach space $(X,\norm)$, where the set 
$\{x \in A: \varphi(x) \subseteq \widehat{X}\}$ contains a dense and $G_\delta$ subset of $A$, is single-valued and norm upper semicontinuous at the points of a dense and $G_\delta$ subset of $A$.
\end{theorem}

\medskip

In order to extend the applicability of Theorem \ref{Warren4Authors}, we will show that some sets that are not necessarily complete metric spaces under their given metrics can be ``re-metrized'' to become a complete metric space under a new metric, while retaining the same topology.  Indeed, suppose that $A$ is a nonempty open subset of a complete metric space $(M, d)$. Then $(M \times \R, \rho)$
is also a complete metric space under the metric, 
$$\rho((x_1, r_1),(x_2, r_2)) := d(x_1, x_2) + |r_1 -r_2|.$$
Let $f:A \to \R$ be defined by, $f(x) := \inf\{d(x,y) \in \R: y \in M \setminus A\} = \mathrm{dis}(x, M\setminus A)$. Note that $f$ is continuous on $A$.  Let $G:= \{(x,r) \in M \times \R:
x \in A \mbox{ and } r = 1/f(x)\}$.  Then $G$ is a closed subset of $M \times \R$, and hence is a complete metric space with respect to the restriction of the metric $\rho$ to $G$. Finally, let us note
that $G$ is homeomorphic to $A$. Indeed, the mapping $\pi:G \to A$ defined by, $\pi(x,r) := x$, is such a homeomorphism.  Thus, a nonempty open subset of a complete metric space is
``completely metrisable''.

\begin{theorem}\label{variation theorem} Let $f:X \to \R\cup\{\infty\}$ be a proper function on a Banach space $(X, \norm)$.  If there exists a nonempty open subset $A$ of $\mathrm{Dom}(f^*)$
 such that $\mathrm{argmax}(x^*-f) \not= \varnothing$ for each $x^* \in A$, then there exists a dense and $G_\delta$ subset $R'$ of $A$ such that
$(x^*-f):X \to \R \cup\{-\infty\}$ has a strong maximum for each $x^* \in R'$. In addition, if $0 \in A$ and $\varepsilon >0$ then there exists an 
$x_0^* \in X^*$ with $\|x^*_0\| < \varepsilon$ such that $(x_0^*-f):X \to \R \cup\{-\infty\}$ has a strong maximum.
\end{theorem}

\begin{proof} Consider $\partial f^*:A \to 2^{X^{**}}$. Then, by Proposition \ref{prelim} part (iii), $\partial f^*(x^*) \cap \widehat{X} \not= \varnothing$ for all $x^* \in A$. Thus, by
Theorem \ref{JamesUsco}, $\partial f^*(x^*) \subseteq \widehat{X}$ for all $x^* \in A$. Hence, $x^* \mapsto \partial f^*(x^*)$, is a minimal weak cusco on $A$.  
Therefore, by Theorem \ref{Warren4Authors}, there exists a dense and $G_\delta$ subset $R'$ of $A$ such that $\partial f^*$ is single-valued and norm upper semicontinuous at each point of $R'$. 
So, by Proposition \ref{prelim} part (v), $(x^*-f)$ has a strong maximum for each $x^* \in R'$.
\end{proof}

Note that the conclusion of this theorem is identical to that of Stegall's variational principle, see \cite{Stegall}.

\begin{question} Is every Banach space $(X,\|\cdot\|)$ a dual differentiation space?
\end{question}
If the answer to this question is ``yes'' then Theorem \ref{DD theorem} will supersede Theorem \ref{variation theorem}.

\section*{Index of notation and assumed knowledge}

\begin{itemize}

\item The {\it natural numbers}, $\mathbb{N}:=\{1,2,3,\ldots\}$.
\item The {\it integers}, $\mathbb{Z}:=\{\ldots,-2,-1,0,1,2\ldots\}$.
\item The {\it rational numbers}, $\mathbb{Q}:=\left\{a/b:a,b\in \mathbb{Z}, b\neq 0\right\}$.
\item The {\it real numbers}, $\R$.

\item For any set $X$, $\mathcal{P}(X)$ is the set of all subsets of $X$.

\item For any subset $A$ of a topological space $\left(X, \tau\right)$, we define
\begin{itemize}
\item $\mathrm{int}(A)$, called the {\it interior} of $A$, is the union of all open sets contained in $A$;
\item $\overline{A}$, called the {\it closure} of $A$, is the intersection of all closed sets containing $A$;
\item $\mathrm{Bd}(A)$, called the {\it boundary} of $A$, is $\overline{A}\setminus \mathrm{int}(A)$,
\end{itemize}

\item For any points $x$ and $y$ in a vector space $X$, we define the following intervals:
\begin{itemize}
\item $[x,y]:=\{x+\lambda(y-x):0\leq \lambda \leq 1\};$
\item $(x,y):=\{x+\lambda(y-x):0< \lambda < 1\};$
\item $[x,y):=\{x+\lambda(y-x):0\leq \lambda < 1\};$
\item $(x,y]:=\{x+\lambda(y-x):0< \lambda \leq 1\}.$
\end{itemize}

\item For any normed linear space $\left(X,\norm{\cdot}\right)$, we define
\begin{itemize}
\item $B[x,r]:=\left\{y\in X: \|x-y\|\leq r\right\}$, for any $x \in X$ and $r > 0$;
\item $B(x;r):=\left\{y\in X: \|x-y\|< r\right\}$ , for any $x \in X$ and $r>0$;
\item $B_{X}:=B[0,1]$;
\item $S_{X}:=\left\{x\in X: \|x\|=1\right\}$ .
\end{itemize}


\item
Given a compact Hausdorff space $K$, we write $C(K)$ for the set of all real-valued continuous functions on $K$. 
This is a vector space under the operations of pointwise addition and pointwise scalar multiplication. $C(K)$ becomes a 
Banach space when equipped with the uniform norm $\norm_\infty$, defined by
\[\|f\|_\infty:=\sup_{x\in K}|f(x)|, \mbox{\quad for all } f\in C(K). \]

\item Let $A$ and $B$ be sets. Given a function $f:A\rightarrow B$, we define $f(A):=\bigcup_{a\in A}\{f(x)\}$.
Similarly, given a set valued mapping $\varphi:A\rightarrow \mathcal{P}(B)$, we define $\varphi(A):=\bigcup_{a\in A}\varphi(x)$.

\item For a normed linear space $(X,\norm{\cdot})$, $X^*$, the set of bounded linear maps from $X$ to $\mathbb{R}$, is called the 
{\it dual space} of $X$. $X^*$ is a Banach space when equipped with the operator norm, given by
\[\|f\|:=\sup_{x\in B_X}|f(x)|, \text{ \quad for all } f\in X^*.\]

\item Let $X$ be a set and $Y$ a totally ordered set. For any function $f:X\rightarrow Y$ we define 
\begin{align*}
\text{argmax}(f):&=\{x\in X:f(y)\leq f(x) \text{ for all } y\in X\},\\
\text{argmin}(f):&=\{x\in X:f(x)\leq f(y) \text{ for all } y\in X\}.
\end{align*}

\item Let $A$ be a subset of a vector space $X$. Then the {\it convex hull} of $A$, denoted by {\it $\mathrm{co}(A)$}, is defined to be the intersection of all convex subsets of $X$ that contain $A$. 

\item Let $X$ be a set and let $f:X \rightarrow \mathbb{R}\cup \{\infty\}$ a function. Then $$\mbox{Dom}(f):=\{x\in X: f(x)<\infty\}.$$  
We say that the function $f$ is a {\it proper function} if $\mathrm{Dom}(f) \not=\varnothing$.

\item Let $(X,\norm)$ be a normed linear space and $f:X \to [-\infty,\infty]$. Then the {\it Fenchel conjugate} of $f$ is the function $f^*: X^* \to [-\infty,\infty]$ defined by,
$$f^*(x^*) := \sup_{x \in X} \{x^*(x) - f(x)\}.$$
The function $f^*$ is convex and if $f$ is a proper function then $f^*$ never takes the value $-\infty$.


\item If $f$ is a convex function defined on a nonempty convex subset $K$ of a normed linear space $(X, \norm{\cdot})$ and $x \in K$, then we define the 
{\it subdifferential of $f$ at $x$} to be the set $\boldsymbol{\partial f(x)}$ of all $x^* \in X^*$ satisfying
$$x^*(y-x) \leq f(y) - f(x) \mbox{\quad for all $y \in K$}.$$

\item It is assumed that the reader has a basic working knowledge of metric spaces,  normed linear spaces and even basic general topology.  In particular, it is assumed that the reader is familiar with 
Tychonoff's theorem.

\begin{thm*}[Tychonoff's Theorem \cite{Engelking}] The Cartesian product $\prod_{s \in S} S_s$, where $X_s \not= \varnothing$ for all $s \in S$, is compact if, and only if, all spaces $X_s$ are compact.
\end{thm*}

\end{itemize}


\bibliographystyle{amsplain}
\bibliography{SurveyPaperReferences}

\medskip

\hrule

\bigskip

{\it Corresponding author}:  Warren B. Moors.  

\medskip

{\it Email}:  w.moors@auckland.ac.nz  

\medskip

{\it Full mailing address}: Warren B. Moors, Department of Mathematics \\
The University of Auckland, Private Bag 92019, Auckland Mail Centre \\
Auckland 1142, NEW ZEALAND.

\end{document}